%% file: main.tex
\documentclass{amsart}
\usepackage[utf8]{inputenc}

\usepackage[shortlabels]{enumitem}

\usepackage[bookmarks,bookmarksnumbered,pdfstartview=FitBH,backref=page]{hyperref}
\hypersetup{
colorlinks=true, 
linkcolor=blue,
citecolor= blue,
linktoc= all, 
}
\numberwithin{equation}{section} 

\usepackage[mathscr]{eucal}
\usepackage{setspace}
\usepackage{amsmath, amsthm}
\usepackage{amssymb}
\usepackage{thmtools}
\usepackage{mathrsfs}
\usepackage{bbding}
\usepackage{bbold}
\usepackage{graphicx}
\usepackage{pifont}
\usepackage{verbatim}
\usepackage{eucal}
\usepackage{fancyhdr}
\usepackage{bm}
\usepackage{upgreek}
\usepackage{mathtools}
\usepackage{stmaryrd}
\usepackage [english]{babel}
\usepackage [autostyle, english = american]{csquotes}
\MakeOuterQuote{"}

\usepackage[capitalise]{cleveref}

\newtheorem{theorem}[equation]{Theorem}
\crefname{theorem}{Theorem}{Theorems}

\newtheorem{lemma}[equation]{Lemma}

\crefname{claim}{Claim}{Claims}
\newtheorem{claim}{Claim}
\newtheorem*{claim*}{Claim}
\newtheorem{corollary}[equation]{Corollary}
\newtheorem{proposition}[equation]{Proposition}

\theoremstyle{definition}
\newtheorem{definition}[equation]{Definition}
\newtheorem{example}[equation]{Example}

\theoremstyle{remark}
\newtheorem{remark}[equation]{Remark}

\declaretheoremstyle[
spaceabove=\topsep, 
spacebelow=6pt,
headfont=\normalfont\itshape,
notefont=\normalfont, notebraces={(}{)},
bodyfont=\normalfont,
postheadspace=4pt,
qed=\mbox{\smaller[4]$\boxtimes$}
]{claimproofstyle}

\usepackage{tikz,ifthen}
\usetikzlibrary{cd}
\usetikzlibrary{babel}
\usetikzlibrary{decorations.markings}
\usetikzlibrary {arrows.meta,bending,positioning}
\usetikzlibrary{calc}
\usepackage{caption}
\usepackage{subcaption}

\tikzstyle{->-} = [decoration={markings,mark=at position 0.7 with {\arrow{Stealth}}},
postaction={decorate}]

\newcommand{\Kc}{\mathcal{K}}

\newcommand{\K}{\mathscr{K}}
\newcommand{\M}{\mathbf{M}}
\newcommand{\G}{\mathcal{G}}
\renewcommand{\H}{\mathcal{H}}
\newcommand{\F}{\mathcal{F}}
\newcommand{\E}{\mathbf{E}}

\newcommand{\bigmid}{\;\big|\;}

\renewcommand{\epsilon}{\varepsilon}
\newcommand{\acts}{\curvearrowright}

\newcommand{\bij}{\xrightarrow{\,\smash{\raisebox{-0.6ex}{\ensuremath{\scriptstyle\sim}}}\,}}
\newcommand{\set}[1]{\left\{ #1 \right\}}
\newcommand{\w}{\mathrm{w}}

\usepackage{dsfont}

\newcommand{\deriv}{\mathrm{d}}

\newcommand{\MAlg}{\mathrm{MALG}}

\let\ker\relax
\DeclareMathOperator{\ker}{ker}
\DeclareMathOperator{\im}{im}
\DeclareMathOperator{\dom}{dom}

\DeclareMathOperator{\proj}{proj}

\usepackage{xcolor}

\let\dfn\textbf

\newcommand{\orig}{\mathtt{o}}
\newcommand{\targ}{\mathtt{t}}

\let\tilde\widetilde
\let\~\widetilde
\let\hat\widehat

\advance \topmargin by -\headheight
\advance \topmargin by -\headsep
\evensidemargin \oddsidemargin
\DeclareSymbolFont{AMSb}{U}{msb}{m}{n}

\newcommand{\N}{\mathbf{N}}
\newcommand{\Z}{\mathbf{Z}}
\newcommand{\R}{\mathbf{R}}
\newcommand{\Q}{\mathbf{Q}}

\newcommand{\bR}{\mathbf{R}}
\newcommand{\bN}{\mathbf{N}}
\newcommand{\RR}{{\bR_{\scriptscriptstyle{>0}}}}
\newcommand{\NN}{{\bN_{\scriptscriptstyle{>0}}}}

\renewcommand{\emptyset}{\varnothing}

\DeclareMathOperator{\Aut}{\mathrm{Aut}}

\DeclareMathAlphabet\mathbfcal{OMS}{cmsy}{b}{n}

\usepackage[draft,inline,nomargin,index]{fixme}

\fxsetup{theme=color,mode=multiuser}
\FXRegisterAuthor{dg}{adg}{\color{orange}DG}
\FXRegisterAuthor{ap}{aap}{\color{violet}AP}
\FXRegisterAuthor{at}{aat}{\color{red}AT}
\FXRegisterAuthor{rtd}{ardt}{\color{blue}RTD}
\FXRegisterAuthor{kw}{akw}{\color{teal}KW}

\definecolor{mygreen}{RGB}{45, 160, 22}
\FXRegisterAuthor{cite}{acite}{\color{mygreen}cite}

\usepackage[all]{xy}

\newcommand{\BS}{\mathrm{BS}}

\newcommand\PoincRetTime{\tau} 
\newcommand\PoincRetMap{\mathcal{P}} 
\newcommand{\actcoc}{\rho}

\newcommand\KKK{\mathcal K}

\newcommand\EE{R}

\allowdisplaybreaks 

\usepackage[titletoc]{appendix}

\title{The measure equivalence classification of Baumslag--Solitar groups}
\author{Damien Gaboriau, Antoine Poulin,\\ Anush Tserunyan,  Robin Tucker-Drob, Konrad Wr\'obel}

\date{\today}

\begin{document}

\begin{abstract}
In this article we show that all the nonamenable Baumslag--Solitar groups with nonunimodular Bass-Serre tree, i.e., $\mathrm{BS}(p,q)$ with $|p|,|q| \neq 1$ and $|p|\neq |q|$, are measure equivalent to each other, thereby completing the measure equivalence classification of Baumslag--Solitar groups: each $\BS(p,q)$ belongs to one of three measure equivalence classes according to whether it is amenable ($|p|=1$ or $|q|=1$), virtually isomorphic to $\mathbf{F}_n \times \mathbf{Z}$ ($|p|=|q|\neq 1$), or is nonamenable with nonunimodular Bass--Serre tree ($|p|,|q| \neq 1$ and $|p| \neq |q|$).
More generally, we obtain the similar classification of the generalized Baumslag--Solitar groups of rank $1$ (i.e., groups that act cocompactly on a tree with edge and vertex stabilizers infinite cyclic).

Our proof also establishes a classification of all treeable lcsc groups into exactly six measure equivalence classes. 
Our new contribution which completes this taxonomy is: all nonunimodular nonamenable treeable lcsc groups are measure equivalent.

Along the way we obtain several results about type $\mathrm{III}$ equivalence relations and graphings of independent interest.
For example, we show that every ergodic type $\mathrm{III}_0$ graphing admits an ergodic hyperfinite subgraph for which the natural factor map of Krieger flows is an isomorphism.

\end{abstract}



\maketitle

\tableofcontents


\section{Introduction}

Given two nonzero integers $p$ and $q$, the associated \textbf{Baumslag--Solitar group} $\BS(p,q)$, introduced by Baumslag and Solitar in \cite{Baumslag-Solitar-1962}, is the two-generated one-relator group with presentation
\begin{equation}\label{BS_presentation}
\BS(p,q)=\langle t,a \mid ta^p t^{-1}=a^q \rangle.    
\end{equation}
More than 25 years ago, Whyte \cite{Whyte} proved that all nonamenable Baumslag--Solitar groups $\BS(p,q)$ with nonunimodular Bass-Serre trees (i.e., those for which $|p|,|q|\neq 1$ and $|p|\neq |q|$) are quasi-isometric to each other, thereby completing the quasi-isometry classification of Baumslag--Solitar groups initiated by Farb and Mosher \cite{Farb-Mosher}.
Since then, the question as to their measure equivalence (in the sense of Gromov \cite[0.5.E]{Gromov-asympt-inv}) has remained a tantalizingly open problem.
Here, we solve this problem, establishing the measure equivalence counterpart to Whyte's theorem.

\begin{theorem}\label{thm:BS_ME}
All Baumslag--Solitar groups $\BS(p,q)$ with $|p|,|q|\neq 1$ and $|p|\neq |q|$ are measure equivalent to each other.
\end{theorem}

While there is a fruitful analogy between measure equivalence and quasi-isometry, there is no implication between the two notions in general. 
For example, while any two infinite amenable groups are measure equivalent, many different quasi-isometry classes of amenable groups can be distinguished by growth rate. 
In the other direction, property (T) is a measure equivalence invariant \cite[Section 3.1]{Fur}, but not a quasi-isometry invariant \cite[Section 3.6]{BdlHV};
see \cite[p. 1814, footnote $\dagger\dagger$]{Gab05} for other examples.

Theorem \ref{thm:BS_ME} is perhaps surprising in light of the measure equivalence coupling rigidity results for Baumslag--Solitar groups established by Kida in \cite{Kida:BSGroups} and by Houdayer and Raum in \cite{HoudayerRaum:BSGroups}.
It is shown in these articles that, given groups $\BS (p_0,q_0)$ and $\BS (p_1,q_1)$ as in Theorem \ref{thm:BS_ME}, along with a hypothetical measure equivalence coupling $\Omega$ between them, if all nontrivial elliptic subgroups of $\BS (p_i,q_i)$ act ergodically on $\Omega /\BS (p_{1-i},q_{1-i})$,  then $\BS (p_0,q_0)$ and $\BS (p_1,q_1)$ are necessarily isomorphic. 
The couplings produced in the proof of \cref{thm:BS_ME} in fact have the property that no elliptic subgroup acts ergodically on the quotient.

Theorem \ref{thm:BS_ME} also stands in contrast to the von Neumann algebra rigidity exhibited by Meesschaert and Vaes in \cite{MeesschaertVaes:BSgroups}. 
Using Popa's deformation/rigidity theory, they showed that if two nonamenable Baumslag--Solitar groups $\BS(p_0,q_0)$ and $\BS(p_1,q_1)$ are stably $W^*$-equivalent (i.e., their group von Neumann algebras $L(\BS(p_0,q_0))$ and $L(\BS(p_1,q_1))$ are stably isomorphic), then $|\frac{p_0}{q_0}|=|\frac{p_1}{q_1}|^{\pm 1}$.
Measure equivalence has strong parallels with $W^*$-equivalence, just as it does with quasi-isometry, and points of divergence are particularly intriguing, e.g., \cite{ChifanIoana:MEvsW*,IPV:superrigidity,CdSS:products}. 

Combined with previous work of Ornstein and Weiss \cite{OW87} and Kida \cite{Kida:BSGroups}, \cref{thm:BS_ME} completes the classification of Baumslag--Solitar groups up to measure equivalence, as we now explain.

\subsection{Measure equivalence classification}
Each $\BS (p,q)$ admits a natural action on the (oriented) Bass--Serre tree $T_{|p|,|q|}$ associated to the presentation \labelcref{BS_presentation} as an HNN-extension. 
We then have the following three broad families:
\begin{enumerate}[(F1), leftmargin=*]
\item  $|p| = 1$ or $|q| = 1$. 
Then $\BS (p,q)$ is solvable, and hence amenable.
\item  $|p|=|q| \neq 1$.
Then $\Aut(T_{|p|,|q|})$ is unimodular, and $\mathrm{BS} (p,q)$ is nonamenable and virtually isomorphic to the direct product $\mathbf{F} \times \mathbf{Z}$ of a finitely generated nonabelian free group $\mathbf{F}$ with an infinite cyclic group.
\item $|p|, |q|\neq 1$ and $|p|\neq |q|$. 
Then $\Aut(T_{|p|,|q|})$ is nonunimodular and $\BS (p,q)$ is nonamenable.
\end{enumerate}

It follows from the landmark article \cite{OW87} of Ornstein and Weiss that groups in the family $\mathrm{(F1)}$ are all measure equivalent to one another and not measure equivalent to groups in either of the other two families. 
Groups in the family $\mathrm{(F2)}$, being pairwise commensurable, are all measure equivalent to each other, and Kida showed in \cite{Kida:BSGroups} that they are not measure equivalent to any of the groups in the family $\mathrm{(F3)}$.
Thus, together with Theorem \ref{thm:BS_ME}, we obtain the following classification up to measure equivalence. 

\begin{theorem}[ME Classification of Baumslag--Solitar groups]\label{ME_classification}
Let $p$ and $q$ be nonzero integers. 
Then the Baumslag--Solitar group $\BS(p,q)$ belongs to exactly one of three measure equivalence classes, according to whether it belongs to the family $\mathrm{(F1)}$, $\mathrm{(F2)}$, or $\mathrm{(F3)}$.

Groups in the family $\mathrm{(F1)}$ are all measure equivalent to $\mathbf{Z}$,
those in the family $\mathrm{(F2)}$ are all measure equivalent to $\mathbf{F}_2\times\mathbf{Z}$ (where $\mathbf{F}_2$ is a free group on two generators), and those in the family $\mathrm{(F3)}$ are all measure equivalent to $\BS (2,3)$. 
\end{theorem}

We note that \cref{ME_classification} also gives the classification of Baumslag--Solitar groups up to orbit equivalence, since it follows from \cite{Kida:stability} that measure equivalence and orbit equivalence coincide for these groups.

\cref{ME_classification} extends to all $1$-GBS (generalized Baumslag--Solitar) groups, i.e., groups acting cocompactly and without inversion on a tree with all vertex and edge stabilizers infinite cyclic:

\begin{theorem}[ME Classification of $1$-GBS groups]\label{thm:1-GBS}
Let $G$ be the Bass--Serre fundamental group of a finite graph of infinite cyclic groups (i.e., a $1$-GBS group). 
Then $G$ is measure equivalent to
\begin{enumerate}
\item $\mathbf{Z}$ if $G$ is amenable;
\item $\mathbf{F}_2\times\mathbf{Z}$ if $G$ contains a finite index subgroup isomorphic to $\mathbf{F}_r\times\mathbf{Z}$ for some $r\geq 2$;
\item $\BS (2,3)$ otherwise.
\end{enumerate}
\end{theorem}

The proof of \cref{thm:BS_ME} passes through the world of measure equivalence between nonunimodular locally compact groups (defined in \cite{DL14} and developed in \cite{Koivisto-Kyed-Raum-ME-non-unim}); as a major step in the proof we obtain measure equivalence couplings between the nonunimodular groups $\Aut (T_{p_0,q_0})$ and $\Aut (T_{p_1,q_1})$.  

\begin{theorem}\label{thm:AutTpqME}
All the groups $\Aut(T_{p,q})$ with $p,q\geq 2$ and $p\neq q$ are measure equivalent to each other. There are exactly four measure equivalence classes among the groups $\Aut(T_{p,q})$, $p,q\geq 1$:
\begin{enumerate}[(1)]
\item\label{item:uni_amen} $\Aut(T_{1,1})=\Z$;
\item\label{item:nonuni_amen} $\Aut(T_{1,q})$, for  $q\geq 2$ 
(nonunimodular amenable);
\item $\Aut(T_{p,p})$ for $p\geq 2$ (unimodular nonamenable);
\item $\Aut(T_{p,q})$, for $p\not=q$ and $p,q\geq 2$ (nonunimodular nonamenable).
\end{enumerate}
\end{theorem}

These groups share a fundamental property that is invariant under measure equivalence, namely, treeability. 
A locally compact second countable group $G$ is \textbf{treeable} if it admits a free probability measure preserving Borel action $G\acts (X,\mu)$ whose orbit equivalence relation admits a Borel structuring by trees; see \cite[Appendices]{CGMTD}.
For example, the group $\mathrm{SL}_2(\R )$ is treeable, as are all of the groups $\Aut (T_{p,q})$ for $p,q\geq 1$. 
Further examples are the groups $\Aut(T_{(p_0,q_0),(p_1,q_1)})$ of all automorphisms of a directed, edge-labeled tree in which every vertex has $p_i$-many incoming and $q_i$-many outgoing edges with label $i$, for $i\in\{0,1\}$; for these groups, the image of the modular homomorphism is $\langle \tfrac{p_0}{q_0},\tfrac{p_1}{q_1}\rangle$, which is dense in $\RR$ whenever $\langle \tfrac{p_0}{q_0}\rangle$ and $\langle \tfrac{p_1}{q_1}\rangle$ are nontrivial and noncommensurable.

On the other hand, the nonamenable Baumslag--Solitar groups $\BS(p,q)$, and more generally, the nonamenable $1$-GBS groups (as in \cref{thm:1-GBS}), are not treeable; indeed, they all have cost $1$ since they have an infinite commensurated amenable subgroup \cite[Critère VI.24(3)]{Gab00a}  and treeable cost $1$ groups are necessarily amenable \cite[Théorème 4]{Gab00a}.
When $2\leq p_i<q_i$, the nonamenable direct products $\Aut (T_{p_0,q_0})\times \Aut (T_{p_1,q_1})$ furnish nonunimodular groups that are non-treeable by \cite[Theorem 10]{Kechris00} (see also \cite[Corollary 5.13]{T-TD:amenable_in_tree_antitreeability}).

Our methods yield the following generalization of the fourth assertion of \cref{thm:AutTpqME}.

\begin{theorem}\label{thm:nonunim_nonamen_treeable_ME}
All nonunimodular nonamenable treeable lcsc groups are measure equivalent. 
\end{theorem}

\begin{remark}
While we will not be using this, it is interesting to note that every nonunimodular nonamenable treeable lcsc group $G$ admits a compact open subgroup by \cite[Corollary 5.16]{T-TD:amenable_in_tree_antitreeability}.
Thus, if $G$ is additionally compactly generated, then $G$ can essentially be interpreted as a closed, vertex transitive subgroup of the automorphism group of some connected locally finite graph.
\end{remark}

Combining \cref{thm:nonunim_nonamen_treeable_ME} with the known results in the amenable and unimodular cases gives the following taxonomy of treeable locally compact groups up to measure equivalence; for background on cost in the locally compact setting, see \cite[Appendix A]{CGMTD}.

\begin{corollary}
There are exactly 6 measure equivalence classes of treeable lcsc groups:
\begin{enumerate}
\item compact, e.g., the trivial group;
\item unimodular amenable non-compact, e.g., $\mathbf{Z}$;
\item nonunimodular amenable, e.g., $\Aut(T_{1,2})$;
\item unimodular nonamenable finite cost, e.g., $\mathbf{F}_2$;
\item unimodular nonamenable infinite cost, e.g., $\mathbf{F}_\infty$;
\item nonunimodular nonamenable, e.g., $\Aut(T_{2,3})$.
\end{enumerate}
\end{corollary}

\begin{proof}
The class of compact groups forms a single measure equivalence class, as does the class of noncompact unimodular amenable groups \cite{OW87,KoivistoKyedRaum:measure_equivalence_unimodular}, as well as the class of nonunimodular amenable groups \cite[Theorem B]{Koivisto-Kyed-Raum-ME-non-unim}.
Since the trivial group, $\mathbf{Z}$, and $\Aut(T_{1,2})$ are all treeable, the invariance of treeability under measure equivalence shows that all three of these classes consist of treeable groups.

Unimodularity is invariant under measure equivalence by \cite[Corollary 16]{Koivisto-Kyed-Raum-ME-non-unim} and, among unimodular groups, having finite cost is invariant under measure equivalence \cite[p.~44]{Gab00a}.
Moreover, all unimodular nonamenable treeable groups have cost strictly greater than $1$ by \cite[Théorème 4]{Gab00a}; for such a group $G$, applying the induction formula for cost \cite[Proposition II.6]{Gab00a} and Hjorth's results \cite[Lemmas 4.1 and 4.2]{Hjorth:cost-attained} to an ergodic treeable cross-section relation establishes a measure equivalent with either $\mathbf{F}_2$ or $\mathbf{F}_\infty$ according to whether the cost of $G$ is finite or infinite.
The treeable groups that remain are those that are nonunimodular and nonamenable, and \cref{thm:nonunim_nonamen_treeable_ME} shows that these groups comprise a single measure equivalence class.
\end{proof}

\subsection{Related results} 
In \cite{Moldavanskii:BSGroupsIsomorphism}, Moldavanskii showed that two Baumslag--Solitar groups $\BS(p_0,q_0)$ and $\BS(p_1,q_1)$ are isomorphic if and only if there exists $\epsilon\in\{-1,1\}$ such that either $\epsilon p_0 =p_1$ and $\epsilon q_0=q_1$, or $\epsilon p_0=q_1$ and $\epsilon q_0=p_1$.
Since $\BS (p,q)$ and $\BS (-p,q)$ have isomorphic index two subgroups, for the remainder of the introduction we will focus on the groups $\BS (p,q)$ with $1\leq p\leq q$. 

In the late 1990s Farb and Mosher in \cite{Farb-Mosher} classified the solvable Baumslag--Solitar groups up to quasi-isometry by showing that $\BS(1,q_0)$ and $\BS(1,q_1)$ are quasi-isometric if and only if they are abstractly commensurable, which they showed holds if and only if there exist $k,n,m\in \mathbf{N}$ such that $q_0=k^n$ and $q_1=k^m$.
For the classification of Baumslag-Solitar groups up to abstract commensurability in general, see \cite{Casals-RuizKazachkovZakharov:CommensurabilityOfBS}.
Whyte then completed the classification up to quasi-isometry in \cite{Whyte} by showing that groups in the family $\mathrm{(F3)}$ are all pairwise quasi-isometric, and not quasi-isometric to any group in the family $\mathrm{(F2)}$.

The first major progress in understanding measure equivalence couplings between nonamenable Baumslag--Solitar groups was made by Kida in the article \cite{Kida:BSGroups}, which ruled out the possibility of such a coupling between a group in the family $\mathrm{(F2)}$ and one in the family $\mathrm{(F3)}$. 
In \cite{Kida:stability}, Kida showed that Baumslag-Solitar groups are stable in the sense of Jones and Schmidt \cite{JS87}, which in particular implies that $\BS (p,q)$ is measure equivalent to $\BS (p,q)\times \mathbf{Z}$. 

In \cite{HoudayerRaum:BSGroups}, Houdayer and Raum generalized rigidity results from \cite{Kida:BSGroups} with an approach highlighting the relevance of the Schlichting completion of $\BS (p,q)$ (i.e., the closure of $\BS (p,q)$ in $\Aut(T_{p,q})$) in understanding measure equivalence couplings amongst groups in $\mathrm{(F3)}$. 
See \cite{Meesschaert:BSactions} for a further generalization to group measure space von Neumann algebras.

Kida further demonstrated in \cite[Corollary 7.5]{Kida:BSGroups} that a measure equivalence between two Baumslag--Solitar groups in the family $\mathrm{(F3)}$ entails a measure equivalence between their respective modular kernels.
At the time, there was some hope that this could be used toward potentially distinguishing measure equivalence classes in the family $\mathrm{(F3)}$.
However, in the recent article \cite{Kida:TreeingsHNN} Kida proved that all such modular kernels are indeed measure equivalent to each other and to $\mathbf{F}_{\infty}\times\mathbf{Z}$, thereby ruling out this approach.

In a recent breakthrough \cite{Forester:BSComplexes}, Forester showed that, for $k\geq 2$, the groups $\BS (k,k^2)$ and $\BS (k^2,k^4)$ contain finite index subgroups which are both cocompact lattices in the same locally compact group.
This provided the first examples of a measure equivalence between nonamenable Baumslag--Solitar groups that are not abstractly commensurable.
While one might hope to extend the techniques from \cite{Forester:BSComplexes} to obtain Theorem \ref{thm:BS_ME}, such an approach is ruled out by work in progress of Carette, Cornulier, Le Boudec, and Tessera \cite{CCBT}; they show that if the cyclic subgroups $\langle\frac{p_0}{q_0}\rangle$ and $\langle\frac{p_1}{q_1}\rangle$ of $(\RR,\cdot)$ are not commensurable, then $\BS(p_0,q_0)$ is not commable with $\BS(p_1,q_1)$, and hence these groups cannot have finite index subgroups that are both cocompact lattices in the same locally compact group. 

\subsection{Outline of proofs of {\cref{thm:1-GBS,thm:nonunim_nonamen_treeable_ME}}}

We first reduce \cref{thm:1-GBS} to \cref{thm:nonunim_nonamen_treeable_ME}. 
Let $\Gamma$ be a 1-GBS group, i.e., $\Gamma$ acts cocompactly and without inversion on a tree $T_\Gamma$ with infinite cyclic edge and vertex stabilizers.
We assume that $\Gamma$ is nonamenable and does not contain a finite index subgroup isomorphic to $\mathbf{F}_r\times\mathbf{Z}$, toward the goal of showing that $\Gamma$ is measure equivalent with $\BS(2,3)$.

In \cref{subsec:latticeEmbedding}, we record from \cite[Proposition 4.4]{Cornulier_Valette_2015} that, after passing to an index $2$ subgroup of $\Gamma$ if necessary, the group $\Gamma$ is a lattice in the modular semidirect product $\overline \Gamma \ltimes \R$, where $\overline \Gamma$ is the closure of the image of $\Gamma$ in the automorphism group of $T_\Gamma$.
Likewise, $\BS(2,3)$ is a lattice in $\Aut(T_{2,3})\ltimes \R$.
Hence it suffices to show that the modular semidirect products $\overline \Gamma \ltimes \R$ and $\Aut(T_{2,3}) \ltimes \R$ are measure equivalent.
\cref{modular_semidirect_ME-reduction} then reduces this to showing that the nonamenable nonunimodular treeable groups $\overline\Gamma$ and $\Aut(T_{2,3})$ are measure equivalent.

\subsubsection*{Outline of proof of \cref{thm:nonunim_nonamen_treeable_ME}}

The Krieger flow associated to an ergodic type III measure class preserving (mcp) countable Borel equivalence relation plays a major role throughout the proof. 
In \cref{sec:scalar_potentials} we introduce the language of \textbf{scalar potentials and potential flows}, which provides a streamlined framework for studying the Krieger flow and its factors.

We assume that $G$ is a nonamenable nonunimodular treeable lcsc group with modular homomorphism $\Delta :G\to \RR$, and let $p$ and $q$ be integers with $2 \leq p < q$, with the goal of proving that $G$ and $\Aut(T_{p,q})$ are measure equivalent.
Toward this end, we isolate in \cref{def:TpqTreeing} the notion of a \textbf{$\mu$-homogeneous $T_{p,q}$-treeing} of an mcp countable Borel equivalence relation $(E,\mu )$. 

\begin{enumerate}[(1), leftmargin=*]
\item\label{step:reduction_Tpq-treeing} We fix a vertex of $T_{p,q}$ and let $K\coloneq K_{p,q}\leq \Aut(T_{p,q})$ denote its stabilizer.
Given a free ergodic probability measure preserving (pmp) action of $\Aut(T_{p,q})$ on $(\tilde{Y},\tilde{\nu} )$, the associated orbit equivalence relation descends in the quotient $(Y,\nu )\coloneqq (\tilde{Y}/K,\tilde{\nu}/K )$ to an mcp type III countable Borel equivalence relation $E_Y$ that comes naturally equipped with a $\nu$-homogeneous
$T_{p,q}$-treeing.
Furthermore, every ergodic $\nu$-homogeneous $T_{p,q}$-treeing arises in this way (\cref{lem:cocycleTpqTreeingArisesFromAutTpqAction}).

Thus, using the correspondence between measure equivalence and stable orbit equivalence \cite[Theorem A]{Koivisto-Kyed-Raum-ME-non-unim}, 
finding a measure equivalence coupling between $G$ and $\Aut(T_{p,q})$ is essentially equivalent to finding an appropriate free ergodic treeable pmp action of $G$ on a space $(\tilde{X},\tilde{\mu} )$, 
along with a uniformly discrete cross-section $(X,\mu_X)$ (where $\mu_X$ is the canonical cross-section probability measure, given by \cref{pmp_cross-section}), such that the measure class $[\mu_X]$ contains a probability measure $\mu_0$ on $X$ for which the cross-section mcp equivalence relation $(E_X, \mu_0 )$ admits a $\mu_0$-homogeneous $T_{p,q}$-treeing.  

\medskip

\item If $G\curvearrowright (\tilde{X},\tilde{\mu})$ is any action of our mystery nonamenable nonunimodular treeable group which successfully implements the strategy described in step \labelcref{step:reduction_Tpq-treeing}, then $(E_X, \mu_X )$ necessarily satisfies: 
\begin{enumerate}[(i)]
\item \label{property:erg_nonamen_treeable} $(E_X,\mu _X)$ is treeable, nonamenable, and ergodic;

\item \label{property:Krieger_factors} the Krieger flow $\RR\curvearrowright \K_{E_X}$ of $(E_X,\mu_X )$ factors onto the circle $\RR\acts \RR/\langle \tfrac{p}{q}\rangle$; 

\item \label{property:inv-meas} the measure class on $\K_{E_X}$ contains an $\RR$-invariant probability measure (\cref{th: pmp G-action implies potential flow}).
\end{enumerate}

\smallskip

We will show that these conditions together with the following property \labelcref{property:type_III_0} are sufficient.

\begin{enumerate}[resume*]
\item \label{property:type_III_0} $(E_X,\mu_X )$ is type $\mathrm{III}_0$.
\end{enumerate}

\noindent More precisely, we prove in \cref{circle_factor_pmp_Krieger_gives_Tpq-treeing} that if a free pmp action $G \acts (\~X,\~\mu)$ is given, with a uniformly discrete cross-section $(X,\mu)$ satisfying \labelcref{property:erg_nonamen_treeable,property:type_III_0,property:Krieger_factors,property:inv-meas}, then there exists a measure $\mu_0$ equivalent to $\mu_X$ such that $(E_X,\mu_0 )$ admits a $\mu_0$-homogeneous $T_{p,q}$-treeing. 
We then use treeability, nonunimodularity, and nonamenability of $G$ to produce such an action $G \acts (\~X, \~\mu)$, as explained in the proof of \cref{ME_with_Autpq}. 

The main difficulty lies in the proof of the sufficiency of conditions \labelcref{property:erg_nonamen_treeable,property:type_III_0,property:Krieger_factors,property:inv-meas} (\cref{circle_factor_pmp_Krieger_gives_Tpq-treeing}), and we discuss this proof in the steps below.

\medskip

\item Property \labelcref{property:type_III_0} ensures that the Krieger flow $\K_{E_X}$ is aperiodic, and property \labelcref{property:Krieger_factors} allows us to find a Borel cross-section $Z \subseteq \K_{E_X}$ that is \textbf{$\frac{p}{q}$-equispaced}, i.e., with $Z\cap \RR . z = \langle\frac{p}{q}\rangle .z$ for each $z\in Z$.
Using property \labelcref{property:inv-meas}, we then apply \cref{changing measure to reach a pmp in Krieger flow} to this cross-section $Z$ to obtain an equivalent probability measure $\mu_0 \sim \mu$ and a Krieger scalar potential $\pi_0: X \to \K_{E_X}$ for $(E_X,\mu_0)$ such that the pushforward measure $(\pi_0)_*\mu_0$ is in the Ramsay measure-class on $Z$ (see \cref{measure-class_on_cross-section}) and is an invariant measure for the cross-section equivalence relation $E^\RR_Z$ on $Z$.

The cross-section $Z$ being $\tfrac{p}{q}$-equispaced corresponds to the Radon--Nikodym cocycle $\w_{\mu_0}(y,x)$ taking all, and only, values in the subgroup $\langle \tfrac{p}{q}\rangle$ of $\RR$, for each fixed $x\in X$, as $y$ varies over the $E_X$-class of $x$; 
this is certainly a prerequisite for $(E_X,\mu_0 )$ to admit a $\mu_0$-homogeneous $T_{p,q}$-treeing, but the invariance of $(\pi_0)_*\mu_0$ will also be instrumental in the final step below. 

\medskip

\item The general strategy at this point is to measurably reassemble an initial treeing $\mathcal{T}_0$ of $E_X$ (provided by property \labelcref{property:erg_nonamen_treeable}) into a $\mu_0$-homogeneous $T_{p,q}$-treeing $\mathcal{T}_1$ of $(E_X,\mu_0 )$.
To accomplish this, we first find an appropriate ergodic hyperfinite subgraph $\mathcal{H}$ of $\mathcal{T}_0$ which generates an equivalence subrelation $F\leq E_X$ onto which we will place guiding tracks for future edge slides.
To ensure that $F$ sufficiently captures the measurable behavior of the Radon--Nikodym cocycle $\w_{\mu_0}$, we will require not just ergodicity of $F$, but moreover that the natural factor map from the Krieger flow of $(F,\mu_0 )$ to that of $(E_X,\mu_0 )$ is an isomorphism.
Towards this end, we make essential use of the ergodic hyperfinite subgraph theorem from \cite{Tserunyan:hyperfin_erg_subgraph} to prove in \cref{thm:sameTypeHyperfiniteSubgraph} that \textit{every ergodic type $\mathrm{III}_0$-graphing admits an ergodic hyperfinite subgraph for which the natural factor map of Krieger flows is an isomorphism}.
We anticipate this statement being quite useful for future measured combinatorial constructions as well.

\medskip 

\item Assembly of $\mathcal{T}_1$ now occurs in two stages. 
In the first, we discard $\mathcal{H}$ and instead find a more suitable directed graphing of $F$ that will guide the second stage. 
We do this via \cref{prop:graphWithMinCocycleValues}, which shows that $F$ can be generated by a directed graph $\mathcal{S}$ having undirected degree two everywhere, with the property that $\w_{\mu_0} (y,x) =\frac{p}{q}$ whenever there is a directed edge from $x$ to $y$ in $\mathcal{S}$.

\medskip

\item The second stage, carried out in \cref{prop:edgeSlideTpqTreeing}, consists of three phases of sliding the edges of $\mathcal{T}_0 \smallsetminus \mathcal{H}$ along $\mathcal{S}$ to obtain $\mathcal{T}_1$.
Here, we make critical use of the fact that the natural factor map between the Krieger flows of $F$ and $E_X$ is an isomorphism as well as the invariance of the measure $(\pi_0)_*\mu_0$ for the cross-section equivalence relation $E^\RR_Z$ on $Z$.

Since $E_X$ is nonamenable, the edges of $\mathcal{T}_0\smallsetminus\mathcal{H}$ meet a.e.\ $E_X$-class. 
The fact that $F$ is type $\mathrm{III}$ then ensures, by \cite[Theorem 3.1]{Poulin}, that we can first slide heads of these edges so that almost every vertex has in-degree equal to $p$. 
Our choice of $F$ further ensures that we may then slide tails of edges so that $\w_{\mu_0}(y,x)=\tfrac{p}{q}$ whenever there is a directed edge from $x$ to $y$. 
At this point, the expected out-degree is $q$.
But more precisely, a mass transport argument using the invariance of the measure $(\pi_0)_*\mu_0$ shows that the expected out-degree is $q$ within a.e.\ ergodic component of the kernel equivalence subrelation $F\cap (\w_{\mu_0}=1)$.
Towards obtaining out-degree exactly $q$ a.e., we will then slide tails of edges along the kernel $F\cap (\w_{\mu_0}=1)$; since this kernel is pmp, any such sliding will preserve the relative weights along directed edges as well as the in-degrees of vertices.
This establishes the desired $\mu_0$-homogeneous $T_{p,q}$ treeing of $E_X$, concluding the proof that $G$ and $\Aut(T_{p,q})$ are measure equivalent. 
\end{enumerate}

\subsubsection*{Acknowledgments}
ATs thanks Grigory (Greg) Terlov for illuminating discussions in Paris regarding the hyper-type  $\mathrm{II}_1$ phenomenon in the percolation context.
RTD would like to thank Cyril Houdayer and Yoshikata Kida for stimulating discussion over the years that touched on this problem.
We also thank David Kerr and Chiranjib Mukherjee for inviting all of the authors to the workshop
\textit{Probability, Dynamics, and the Geometry of Groups} in
September 9--13, 2024, at the University of Münster,
during which we completed the proof of \cref{thm:nonunim_nonamen_treeable_ME} and took photos with all the authors in decreasing order of heights.

DG is supported by the French C.N.R.S.
ATs was supported by NSERC Discovery Grant RGPIN-2020-07120. 
RTD was supported by NSF grant DMS-2246684.

\subsubsection*{AI statement.} 

\enlargethispage{\baselineskip}The research for this article was entirely conducted by the five named authors, without any consultation with or input from AI.
The first version of the proof of \cref{ME_classification} was completed by July 2024, and the first version of the proof of \cref{thm:nonunim_nonamen_treeable_ME} was completed in September 2024. 
Since then the proof went through various reconceptualizations, including the emergence of the theory of scalar potentials, and this was completed by April 2026, without any use of AI.

ChatGPT-4o was used in 2024 as a thesaurus when writing a preliminary version of the introduction.

\section{Preliminaries}

\subsection{1-GBS groups as lattices in modular semidirect products}\label{subsec:latticeEmbedding}

Let $G$ be a locally compact second countable group, abbreviated as \textbf{lcsc}, with left Haar measure $\mu_G$.
The group $G$ acts by automorphisms on the additive group $(\mathbf{R},+)$ via multiplication through the modular homomorphism $\Delta_G : G\to \RR$ (satisfying $\mu_G(Bg) = \Delta_G (g) \mu_G (B)$ for $B\subseteq G$ Borel). 
The corresponding semidirect product
\[
G\ltimes\mathbf{R}=\langle G,\mathbf{R}\mid grg^{-1}=\Delta_G(g)r\ (g\in G,\ r\in\mathbf{R})\rangle,
\]
which we call the \textbf{modular semidirect product} associated to $G$, is unimodular, with two-sided invariant measure $\mu_G\times\mu_{\mathbf{R}}$.

In what follows we assume $p$ and $q$ are integers with $2\leq p < q$. 
The group $\BS(p,q)=\langle t,a \mid ta^pt^{-1}=a^q \rangle$ can be expressed as the HNN-extension of the infinite cyclic group $\langle a\rangle$ with respect to the isomorphism from the subgroup $\langle a^p\rangle$ to the subgroup $\langle a^q\rangle$ sending $a^p$ to $a^q$.
We let $T_{p,q}= V_{p,q}\sqcup E_{p,q}$ denote the Bass--Serre tree associated to this HNN-extension: the vertex set $V_{p,q}$ of $T_{p,q}$ is the set $\BS(p,q)/\langle a\rangle$ of left $\langle a\rangle$-cosets in $\BS(p,q)$, 
the (positively oriented) edge set $E_{p,q}^+$ of $T_{p,q}$ is the set $\BS(p,q)/\langle a^q\rangle$ of left $\langle a^q\rangle$-cosets in $\BS(p,q)$, where the edge $\gamma \langle a^q \rangle$ has origin $\gamma \langle a\rangle$ and terminus $\gamma t \langle a\rangle$.
Then each vertex $\gamma \langle a \rangle$ of $T_{p,q}$ has $p$-many incoming positive edges (the edges $\gamma a^it^{-1}\langle a ^q \rangle$ for $0\leq i <p$) and $q$-many outgoing positive edges (the edges $\gamma a^i\langle a ^q \rangle$ for $0\leq i < q$).

Let $\Aut(T_{p,q})$ denote the group of all graph automorphisms of $T_{p,q}$ that preserve the edge orientations.
This group is naturally an lcsc group with respect to the topology of pointwise convergence on $V_{p,q}$, where $V_{p,q}$ is given the discrete topology.
The translation action of $\BS(p,q)$ on each of the coset spaces $V_{p,q}$ and $E_{p,q}^+$ embeds $\BS(p,q)$ as a (non-discrete) subgroup of $\Aut(T_{p,q})$, and in what follows we identify $\BS(p,q)$ with its image in $\Aut(T_{p,q})$.

The stabilizer $K_{p,q}$ of the vertex $\langle a\rangle$ is a compact open subgroup of $\Aut(T_{p,q})$. 
We have that $\BS(p,q)K_{p,q} = \Aut(T_{p,q})$, since $\BS(p,q)$ acts transitively on $V_{p,q}$.
Let $\mu _{p,q}$ denote left Haar measure on $\Aut(T_{p,q})$, normalized so that $\mu _{p,q}(K_{p,q})=1$.

The group $\Aut(T_{p,q})$ is nonunimodular, with modular homomorphism $\Delta : \Aut(T_{p,q})\to \RR$ uniquely determined by $\Delta (t)=q/p$, and $\Delta (K_{p,q}) = \{ 1 \}$.
The kernel $\ker(\Delta )$ of the modular homomorphism is an open subgroup, and $\Aut(T_{p,q})$ is the internal semidirect product $\Aut(T_{p,q})=\langle t\rangle \ltimes \ker(\Delta) $.

Taking the modular semidirect product associated to $\Aut(T_{p,q})$, we obtain the group $\Aut(T_{p,q})\ltimes \mathbf{R}$.
The homomorphism $i : \BS(p,q)\to \Aut(T_{p,q})\ltimes\mathbf{R}$, uniquely determined by $i (a) \coloneqq (a,1 )$ and $i (t) \coloneqq (t,0)$, is injective, and the image $i (\BS(p,q))$ is a cocompact lattice with fundamental domain $K_{p,q}\times [0,1)$.
This may be viewed as a special case of \cite[Proposition 4.4]{Cornulier_Valette_2015}: Cornulier and Valette exhibit $\BS(p,q)$ as a lattice in the smaller group $\overline \BS(p,q)\ltimes \mathbf{R}$, where $\overline \BS(p,q)$ is the closure of the natural embedding of $\BS(p,q)$ in $\Aut(T_{p,q})$; since $\overline \BS(p,q)$ is cocompact and of finite covolume in $\Aut(T_{p,q})$, this is equivalent to exhibiting $\BS(p,q)$ as a lattice in $\Aut(T_{p,q})\ltimes \R$.

More generally, assume $\Gamma$ is a generalized Baumslag-Solitar group of rank $1$ ($1$-GBS group), equivalently $\Gamma$ acts cocompactly and without inversion on a tree $T_\Gamma$ with infinite cyclic edge and vertex stabilizers.
Let $\Aut_\Gamma(T_\Gamma)$ be the group of all automorphisms of $T_\Gamma$ that setwise fix each $\Gamma$-orbit, i.e. $\Aut_\Gamma(T_\Gamma)\backslash T_\Gamma=\Gamma\backslash T_\Gamma$ as a graph. 
Commensuration of a vertex group in $\Gamma$ induces a homomorphism $\delta : \Gamma \to \Q^{\times}\subseteq \R^\times$ (see \cite[p.~483]{Levitt:GBS}). 
Let $\Gamma_0\coloneq \delta^{-1}(\RR )\leq \Gamma$ be the subgroup (of index at most $2$) taking values in $\RR$, and let $\overline{\Gamma}_0$ be the closure of the natural homomorphic image of $\Gamma_0$ in $\Aut_\Gamma(T_\Gamma)$.
The homomorphism induced by $\delta$ on the image of $\Gamma_0$ extends to the modular homomorphism of $\overline{\Gamma}_0$.
Then, by \cite[Proposition 4.4]{Cornulier_Valette_2015}, $\Gamma_0$ is a cocompact lattice in the modular semidirect product $\overline{\Gamma}_0\ltimes \R$.

\subsection{Measure class preserving equivalence relations and graphs}

A \textbf{countable Borel equivalence relation} $E$ is a Borel equivalence relation $E \subset X^2$ on a standard Borel space $X$ such that for all $x \in X$, the equivalence class of $x$ is countable. 
Given a subset $A\subseteq [x]_E$, we write $[A]_E\coloneqq \{x \in X \mid x \mathrel{E} y\in A\}$.
A \textbf{complete section} for $E$ is a Borel subset of $X$ that intersects every $E$-class.
A \textbf{transversal} for $E$ is a Borel subset of $X$ that intersects every class of $E$ in exactly one point. If $E$ admits a  transversal, then $E$ is said to be \textbf{smooth}.

Given a countable Borel equivalence relation $E$ on $X$, we define its \textbf{full group}, denoted $[E]$, to be the group of Borel bijections $\phi: X \to X$ such that $\phi(x) \mathrel{E} x$ for all $x \in X$. We define its \textbf{full pseudogroup}, denoted $[[E]]$, to be the group of Borel bijections $\phi : A \to B$ such that $\phi(x) \mathrel{E} x$ for all $x \in A$, where $A,B$ are Borel subsets of $X$.

A countable Borel equivalence relation $E$ on a standard probability space $(X, \mu)$ is said to be \textbf{probability measure preserving (pmp)} and $\mu$ is said to be $E$-invariant if $\phi_*\mu = \mu$ for every $\phi \in [E]$.
A countable Borel equivalence relation $E$ on a standard measure space $(X,\mu)$ is said to be \textbf{measure class preserving (mcp)} if the measures $\phi_*\mu$ and $\mu$ are equivalent ($\phi_*\mu \sim\mu$) for every $\phi \in [E]$.
An mcp relation $E$ is \dfn{ergodic} if every $E$-invariant set is either $\mu$-null or $\mu$-conull. A subequivalence relation  $F$ of $E$ is \dfn{ergodic relative} to $E$
if $[A]_E\smallsetminus [A]_F$ is $\mu$-null for every Borel set $A\subseteq X$.

We say two ergodic mcp countable Borel equivalence relations $(E,X,\mu)$ and $(F,Y,\nu)$ are \textbf{isomorphic (resp.\ stably isomorphic)} if there exist conull (resp.\ positive measure) sets $X'\subseteq X$ and $Y'\subseteq Y$ and a Borel isomorphism $\phi: X'\longrightarrow Y'$ such that $\phi_*(\mu |_{X'})\sim \nu|_{Y'}$ and for every $x\in X'$ and for every $y\in [x]_E\cap X'$ we have $\phi(y)\in [\phi(x)]_F$.

Given a countable Borel equivalence relation $E$ on a standard measure space $(X,\mu)$, we define the left and right fiber measures  $\mu^\ell_E \coloneqq \int_X c_x \, \deriv \mu(x)$ and $\mu^r_E \coloneqq \int_X c^x \, \deriv \mu(x)$
where $c_x$ and $c^x$ denote counting measure on $\{ x \}\times [x]_{E}$ and $[x]_{E}\times \{ x \}$ respectively.
Observe that $E$ is pmp if and only 
if $\mu^\ell_E=\mu^r_E$.
Similarly, $E$ is mcp if and only if $\mu^\ell_E \sim \mu^r_E$. 
In this case, the Radon--Nikodym derivative $\frac{\deriv \mu^\ell_E}{\deriv \mu^r_E} :E\to (0,\infty)$ is called the \textbf{Radon--Nikodym cocycle} of $E$ with respect to $\mu$, and we denote it by $\w_\mu$.
After discarding an $E$-invariant Borel null set, the Radon--Nikodym cocycle $\w_\mu$ satisfies the cocycle identity $\w_\mu(z,y)\w_\mu(y,x) = \w_\mu(z,x)$ for all points $x,y,z$ in the same $E$-class.
In this context, the defining property of the Radon--Nikodym derivative is often called the \textbf{mass transport principle}: for every Borel map $f:E\to [0, \infty]$, we have
\begin{equation}\label{mass_transport_formula}
\int_X \sum_{y \in[x]_E} f(x,y) \, \deriv \mu(x) = \int_X \sum_{y \in[x]_E} f(y,x) \w_\mu(y,x) \, \deriv \mu(x).
\end{equation}
For $A\subseteq [x]_E$, we define $\w_\mu(A \times \{x\}) \coloneqq \{\w_\mu(y,x)\;: \; y\in A\}$. 

By a \dfn{Borel graph} on a standard Borel space $X$, we mean a graph with vertex set $X$, and whose edge set $\G$ is a standard Borel space whose origin and terminus maps $\orig_{\G} :\G\to X$, $\targ_{\G} : \G \to X$ are both Borel.
We say that $\G$ is locally countable if every vertex is incident with at most countably many edges.
We denote by $E_{\G}$ the equivalence relation on $X$ whose classes are the connected components of $\G$.
For a locally countable Borel graph $\G$ and a Borel probability measure $\mu$ on $X$, we say that $(\G, X, \mu)$ is pmp, mcp, or ergodic according to whether $(E_\G,X, \mu)$ has the corresponding property.

\subsection{Properties of the Radon--Nikodym cocycle and types of mcp relations}

Let $E$ be an mcp countable Borel equivalence relation on a standard measure space $(X,\mu)$. 
The \textbf{kernel of the Radon--Nikodym cocycle} $\w_\mu$ is defined to be the equivalence relation $\ker_E \w_\mu\leq E$ on $(X,\mu)$ where two points $x,y\in X$ are $\ker_E\w_\mu$-equivalent if they are $E$-equivalent and $\w_\mu(y,x)=1$.
Note that the kernel $\ker_E \w_\mu$ preserves the measure $\mu$.
The Radon--Nikodym cocycle $\w_\mu$ is said to be \textbf{dis-closed} if $\w_\mu ([x]_E\times \{ x \} )$ is a closed and discrete subset of $\RR$ for almost every $x\in X$.

Let $E$ be an ergodic mcp countable Borel equivalence relation on a standard probability space $(X,\mu)$. 
The equivalence relation $E$ is said to be \textbf{type I} if $\mu$ concentrates on a single $E$-class; 
$E$ is \textbf{type II} if $E$ is not type I and there exists an equivalent sigma-finite measure $\nu \sim \mu$ such that $\nu$ is $E$-invariant. 
We further say $E$ is \textbf{type II$_1$} (resp.\ \textbf{II$_\infty$}) if there exists a probability (resp.\ infinite sigma-finite) measure $\nu\sim\mu$ such that $\nu$ is $E$-invariant. 
$E$ is \textbf{type III} if it is neither type I nor type II. 

In turn, ergodic type III equivalence relations are divided into {\em subtypes} depending on the {\em ratio set} $\Delta(E,\mu)$ of the Radon--Nikodym cocycle of $E$. 
These notions were introduced by Krieger in \cite{Krieger:nonsingular_trans_I_II,Krieger-1970-Araki-Woods} and we refer the reader to \cite[Section 2]{KatznelsonWeiss:NonsingularTransformations} for a comprehensive modern overview.
An ergodic type III equivalence relation $E$ is said to be
\begin{itemize}
\item \textbf{type III$_0$} if $\Delta(E,\mu) = \{ 1 \} \cup \{ 0, \infty \}$, 
\item \textbf{type III$_\lambda$} for $\lambda \in (0,1)$ if $\Delta(E,\mu) = \lambda^\Z \cup \{0, \infty \}$,
\item \textbf{type III$_1$} if $\Delta(E,\mu) = \RR \cup \{0, \infty\}$. 
\end{itemize}
Incidentally, an ergodic mcp countable Borel equivalence relation $E$ is type I or II if and only if $\Delta(E,\mu) = \{1\}$.

We do not use ratio sets in our arguments and instead, we use equivalent definitions of the subtypes III$_\lambda$, $\lambda \in [0,1]$, in terms of Krieger flow given in \cref{defn:classbytype}.

We refer to \cite[Section 3]{FM77a},  \cite{HamachiOsikawa:ErgGroupsAndKrieger}, \cite[Chapter 2]{Schmidt-1990-Algebraic-ideas}, and \cite{ KatznelsonWeiss:NonsingularTransformations} and  for more about type III relations, but here we mention a couple of consequences we use in our arguments. 

\begin{proposition}[\cite{Hopf-1932,Dang73:typeIII_elasticity}]\label{typeIII elasticity}
Let $(E,X,\mu)$ be an ergodic mcp type $\mathrm{III}$ countable Borel equivalence relation.
Then for any two Borel sets $A,B \subseteq X$ of positive measure, there is $\varphi \in [[E]]$ with $\dom(\varphi) = A$ and $\im(\varphi) = B$ up to a null set.
In particular, stably isomorphic ergodic type $\mathrm{III}$  countable Borel equivalence relations are isomorphic.
\end{proposition}

\begin{proposition}\label{bounded-unbounded_cocycle_types}
Let $(E,X,\mu)$ be an mcp countable Borel equivalence relation. Set 
\begin{center}$s_\mu(x) \coloneq \sup_{y \in [x]_E \cap Y} \w_\mu(y,x)$ \ \ and \ \ $i_\mu(x) \coloneq \inf_{y \in [x]_E \cap Y} \w_\mu(y,x) $.\end{center}
\begin{enumerate}[(a), leftmargin=*]
\item\label{item:sup<infinity} If $s_\mu(x) < \infty$ for a.e.\ $x \in X$, then the measure $\nu$ given by $d\nu \coloneq s_\mu \, d\mu$ is an $E$-invariant sigma-finite measure equivalent to $\mu$.

\item\label{item:inf>0} If $i_\mu(x) > 0$ for a.e.\ $x \in X$, then the measure $\nu$ given by $d\nu \coloneq i_\mu \, d\mu$ is an $E$-invariant finite measure equivalent to $\mu$.
\end{enumerate}
In particular, if $E$ is ergodic and type $\mathrm{III}$, then $s_\mu \equiv \infty$ and $i_\mu \equiv 0$ a.e.
\end{proposition}

\begin{proof}
Let $g \in \set{i_\mu, s_\mu}$ and suppose that $0 < g < \infty$ a.e.
Then by the cocycle identity, we have $g(x) = g(z) \w_\mu(z,x)$ for a.e.\ pair $x,z \in X$ of $E$-equivalent points.
Thus, letting $\nu$ be the measure defined by $d \nu \coloneq g \, d \mu$, we check the $E$-invariance of $\nu$ via mass transport: for each Borel $f : E \to [0,\infty]$,
\begin{align*}
\int \sum_{y \in [x]_E} f(x,y) d\nu(x) 
&=
\int \sum_{y \in [x]_E} f(x,y) g(x) d\mu(x)
=
\int \sum_{y \in [x]_E} f(y,x) g(y) \w_\mu(y, x) d\mu(x)
\\
&=
\int \sum_{y \in [x]_E} f(y,x) g(x) d\mu(x)
=
\int \sum_{y \in [x]_E} f(y,x) d\nu(x).
\end{align*}
It remains to note that $i_\mu \le 1$, so $\nu \le 1$ when $g = i_\mu$.
\end{proof}

Recall that an mcp countable Borel equivalence relation $R$ is \textbf{$\mu$-nowhere smooth} if there does not exist a positive measure Borel subset $B$ of $X$ for which the restriction $R |_B$ is smooth.
A special case of the following appears in \cite[Lemma 6.5]{Krieger:Flow}.

\begin{proposition}\label{prop:kernelIsNowhereSmooth}
Let $(E, X, \mu)$ be a $\mu$-nowhere smooth mcp countable Borel equivalence relation whose Radon--Nikodym cocycle $\w_\mu$ is dis-closed. 
Then $\ker_E \w_\mu$ is $\mu$-nowhere smooth as well.
\end{proposition}

\begin{proof}
Let $A$ be any partial Borel transversal of $\ker_E \w_\mu$, i.e., a transversal of $\ker_E \w_\mu$ restricted to some Borel set. 
It suffices to show that $A$ is null.

For $x\in A$, consider \[
W_>(x) \coloneq \{ \w_\mu(y,x) : y\in [x]_E\cap A \text{ and } \w_\mu(y,x)>1\}.
\]
Since $A$ is a partial transversal for $\ker_E\w_\mu$, the set $A_0 \coloneq \set{x \in A : W_>(x) = \emptyset}$ is a partial transversal for $E$, and therefore null since $E$ is $\mu$-nowhere smooth. Thus $W_>(x)$ is nonempty for a.e. $x\in A$.

Furthermore, because the Radon--Nikodym cocycle is dis-closed, the set $W_>(x)$ admits a minimum element for a.e.\ $x \in A$, and hence there is a unique $T(x)$ in the partial $\ker_E \w_\mu$-transversal $A$ achieving this minimum, i.e., with $\w_\mu(T(x), x) = \min W_>(x)$.

Now, if $A$ were non-null, then the mass transport principle applied to the indicator function of the graph $D:= \{ (x,T(x)): x\in A \}$ of the injective function 
$T: A \to A, x\mapsto T(x) $
would give a contradiction via:
\begin{align*}
\mu (A) = \int_A \sum_{y \in[x]_E} 1_D(x,y) \, \deriv \mu(x) 
&= 
\int_{T(A)} \sum_{y \in[x]_E} 1_D (y,x) \w_\mu(y,x) \, \deriv \mu(x)
\\
&=
\int_{T(A)} \w_\mu(T^{-1}(x),x) \, \deriv \mu(x)
<
\mu(T(A)) \le \mu(A)
.
\qedhere
\end{align*}
\end{proof}

\subsection{Toolkit for measure disintegration}

Let $(X,\mu)$ be a standard probability space and $\pi : X \to Z$ be a Borel function to a standard Borel space $Y$.
Let $\mu = \int_Z \mu_z \, \deriv \pi_*\mu (z)$ be the measure disintegration of $\mu$ with respect to $\pi$.
For a Borel function $f : X \to \R$, we denote by $\E[f \mid \pi]$ its conditional expectation with respect to the sigma-algebra generated by $\pi$.
Thus $\E[f \mid \pi](x) = \int f \, \deriv \mu_{\pi(x)}$ for $\mu$-a.e.\ $x\in X$.
If $R$ is a pmp equivalence relation on $(X,\mu)$ with ergodic decomposition map $\pi$, then the conditional expectation $\E[f \mid R]$ with respect to the sigma-algebra of $R$-invariant Borel sets coincides with $\E [f\mid \pi ]$.

\begin{lemma}[pmp redistribution]
\label{typeIISlide}
Let $(R, X, \mu)$ be a pmp countable Borel equivalence relation with ergodic decomposition map $\pi :X\to Z$ and associated disintegration $\mu = \int_Z \mu_z \, \deriv \pi_*\mu (z)$. 
Let $g: X \to \N \cup \{ \infty\}$ be a Borel function such that the set $\{x \in X \mid g(x) > 0 \}$ is a complete section for $R$. 
Let $Y$ be a standard Borel space, and let $\phi : Y \to X$ be a countable-to-one Borel map satisfying $\int_X |\phi^{-1}(x)|\, \deriv \mu_z = \int_X g(x)\, \deriv \mu_z$ for $\pi_*\mu$-a.e.\ $z\in Z$. 
Then there is a Borel map $\psi: Y \to X$ with $(\psi(y),\phi (y))\in \mathbin{R}$ for all $y \in Y$ and 
$ 
\left| \psi^{-1} (x)\right| = g(x)
$
for $\mu$-a.e.\ $x \in X$.
\end{lemma}

\begin{proof}
By measure-theoretic exhaustion using \cite[Lemma 7.10]{Kechris-Miller}.
\end{proof}

\begin{lemma}[Mass transport across a Borel homomorphism]\label{lem:mass_transport_homomorphism}
Let $(E,X,\mu)$ and $(F, Z, \nu)$ be mcp countable Borel equivalence relations and let $\pi : X \to Z$ be a measure-preserving Borel homomorphism, i.e., satisfying $\pi_* \mu = \nu$ and $(\pi\times\pi )(E)\subseteq F$.
Let $\mu = \int \mu_z \, \deriv \pi_*\mu$ denote the measure disintegration of $\mu$ with respect to $\pi$.
Then for every Borel function $f : E \to [0,\infty]$ and for a.e.\ $z \in Z$, we have
\begin{equation}\label{eq:mass_transport_accross_hom}
\int_X \sum_{y \in [x]_E} f(x,y) \; \deriv \mu_z(x) = \sum_{z' \in [z]_F} \w_\nu(z',z) \int_X \sum_{\substack{y \in [x]_E  \\ \pi(y) = z}} f(y,x) \w_\mu(y,x) \; \deriv\mu_{z'}(x).
\end{equation}
\end{lemma}

\begin{proof}
Let $g: E \to [0,\infty]$. Then
\begin{align*}
\int_Z \int_X \sum_{y \in [x]_E} g(x,y) \; \deriv \mu_z(x) \; \deriv\nu(z) &= \int_X \sum_{y \in [x]_E} g(x,y) \; \deriv \mu(x) \\
&= \int_X \sum_{y \in [x]_E} g(y,x) \w_\mu(y,x) \; \deriv \mu(x) \\
&= \int_Z \int_X \sum_{y \in [x]_E} g(y,x) \w_\mu(y,x) \; \deriv \mu_z(x)  \; \deriv\nu(z),
\intertext{by the measure disintegration of $\mu$ over $\nu$ (used twice) and the mass transport principle \labelcref{mass_transport_formula} over $(E,X,\mu)$. Now the mass transport principle over $(F,Z,\nu)$ yields:}
&= \int_Z \int_X\sum_{z' \in [z]_F} \sum_{\substack{{y \in [x]_E}\\ \pi(y) = z' }} g(y,x) \w_\mu(y,x) \; \deriv \mu_z(x)  \; \deriv\nu(z) \\
&= \int_Z  \sum_{z' \in [z]_F} \left(\int_X \sum_{\substack{{y \in [x]_E}\\ \pi(y) = z' }} g(y,x) \w_\mu(y,x) \; \deriv \mu_z(x) \right)  \; \deriv\nu(z) 
\\
&= \int_Z \sum_{z' \in [z]_F} \w_\nu(z',z) \int_X \sum_{\substack{{y \in [x]_E}\\ \pi(y) = z }} g(y,x) \w_\mu(y,x) \; \deriv \mu_{z'}(x)\; \deriv\nu(z).
\end{align*}

Now given $f: E \to [0,\infty]$, for each Borel set $B \subseteq Z$ we apply the just established equality to the function $g(x,y) := f(x,y) \mathbf{1}_{B}(\pi(x))$, which gives
\[
\int_B \int_X \sum_{y \in [x]_E} f(x,y) \; \deriv \mu_z(x) \; \deriv\nu(z) 
= 
\int_B \sum_{z' \in [z]_F} \w_\nu(z',z) \int_X \sum_{\substack{{y \in [x]_E}\\ \pi(y) = z }} f(y,x) \w_\mu(y,x) \; \deriv \mu_{z'}(x)\; \deriv\nu(z).
\]
This implies the equality of the integrands $\nu$-a.e., namely, \cref{eq:mass_transport_accross_hom}.
\end{proof}

We use the following folklore statement in our proofs, a version of which was given in \cite[Lemma 2.1]{Krieger:Flow} and a special case of which for ergodic decompositions was recorded in \cite[Proposition 2.2]{LeMaitre:full-groups} as Maharam's lemma of Dye, originally presented in \cite[Maharam's lemma]{Dye:groups_mp_transformations_I}.
For a standard Borel space $X$, let $P(X)$ denote the standard Borel space of Borel probability measures on $X$.

\begin{proposition}\label{sets_with_prescribed_fiber-measures}
Let $X, Z$ be standard Borel spaces and $z \mapsto \mu_z : Z \to P(X)$ be a Borel bundle of atomless Borel probability measures on $X$.
Then for each Borel function $f : Z \to [0,1]$, there is a Borel set $B \subseteq Z \times X$ with $\mu_z(B_z) = f(z)$ for all $z \in Z$.

Moreover, if $\pi : X \to Z$ is a Borel function such that $\mu_z$ is supported on $\pi^{-1}(z)$ for all $z \in Z$, then there is a Borel subset $A \subseteq X$ with $\mu_z(A) = f(z)$ for all $z \in Z$.
\end{proposition}

\begin{proof}
By the Borel isomorphism theorem, we may assume $X = [0,1]$.
Define $g : Z \times [0,1] \to Z \times [0,1]$ by $(z,x) \mapsto (z, \mu_z([0,x]))$, and let $B_{\le f} \coloneqq \set{(z,x) \in Z \times [0,1] : x \le f(z)}$. 
Then $B \coloneqq g^{-1}(B_{\le f})$ is as desired.

For the additional part, take the bijection $\tilde\pi : x \mapsto (\pi(x), x) : X \to Z \times X$ from $X$ to the (inverse of the) graph of $\pi$. Applying the main statement to the map $z \mapsto \tilde\pi_* \mu_z$, we obtain a Borel set $B \subseteq Z \times X$ with $\tilde\pi_* \mu_z(B_z) = f(z)$, so the set $A \coloneqq \tilde\pi^{-1}(B)$ is as desired.
\end{proof}

\subsection{Measure equivalence}
\label{sect:ME}

Let $G$ be an lcsc group.
A \textbf{left Borel $G$-space} is a standard Borel space $X$ equipped with a jointly Borel left action of $G$.
Given a left Borel $G$-space $X$, we equip the quotient space $G\backslash X$ with the quotient measurable structure.
If $\mathscr{C}$ is a measure class on $X$, then we write $G\backslash\mathscr{C}$ for the pushforward measure class on $G\backslash X$.
Likewise, we define all these notions with ``right'' instead of ``left''.

A \textbf{measure-class-preserving} (\textbf{mcp}) \textbf{Borel $G$-space} is a Borel $G$-space equipped with a probability measure $\mu$ whose class is invariant under the $G$-action, i.e., $g_*\mu\sim\mu$ for each $g\in G$.
We define a \textbf{probability-measure-preserving} (\textbf{pmp}) \textbf{Borel $G$-space} analogously, with $g_*\mu=\mu$ for each $g\in G$.

Measure equivalence for countable discrete groups was defined by Gromov \cite[0.5.E]{Gromov-asympt-inv}.
While our main goal concerns measure equivalence between the discrete Baumslag--Solitar groups, we will make use of measure equivalence between general, possibly nonunimodular, lcsc groups.
The theory of measure equivalence at this level of generality was developed in \cite{DL14,Koivisto-Kyed-Raum-ME-non-unim}.
By a Borel \textbf{$(G,H)$-biaction} on a standard Borel space $\Omega$ we mean a Borel left action of $G$ on $\Omega$ and a Borel right action of $H$ on $\Omega$ such that these two actions commute.

\begin{definition}\label{def:strictMEcoupling}
Let $G$ and $H$ be lcsc groups.
A \textbf{strict $(G,H)$-ME-coupling} is a standard Borel space $\Omega$ together with a measure class $\mathscr{C}$ of a $\sigma$-finite Borel measure, equipped with an mcp Borel $(G,H)$-biaction such that:
\begin{itemize}
\item $G$ and $H$ each act freely;
\item the quotient measurable left $G$-space $\Omega/H$ is standard and admits a $G$-invariant probability measure in the measure class $\mathscr{C}/H$; and
\item the quotient measurable right $H$-space $G\backslash\Omega$ is standard and admits an $H$-invariant probability measure in the measure class $G\backslash\mathscr{C}$.
\end{itemize}
The groups $G$ and $H$ are said to be \textbf{measure equivalent} if they admit a strict $(G,H)$-ME-coupling.
\end{definition}

It was shown in \cite[Proposition~3.6]{DL14} that this definition coincides with that of \cite[0.5.E]{Gromov-asympt-inv} for countable discrete groups and with that of \cite[Definition~1.1]{BFS-13} for general unimodular groups, where the $(G,H)$-biaction is required to be measure preserving.

\begin{remark}\label{rem:algebraic_coupling}
In \cref{def:algebraic_ME_coupling} of Appendix~\ref{app:algebraic_ME}, we reformulate measure equivalence in algebraic terms, using continuous Boolean biactions on standard measure-class algebras. 
\cref{prop:algebraic_to_strict} shows that the two notions are equivalent: every strict $(G,H)$-ME-coupling induces an algebraic $(G,H)$-ME-coupling on its measure-class algebra, and conversely every algebraic $(G,H)$-ME-coupling is induced by a strict one.
\end{remark}

\subsection{Cross-sections of actions of lcsc groups}

For an lcsc group $G$, we write $\mu_G$ for a left Haar measure on $G$.
We let $\Delta _G  :G\to \RR$ denote the modular homomorphism of $G$, denoting it simply by $\Delta$ when $G$ is clear from the context. 
Thus, $\mu_G(Bg) = \Delta (g)\mu _G(B)$ for all $g\in G$ and Borel sets $B \subseteq G$.

Let $G$ be an lcsc group and let $\Omega$ be a left Borel $G$-space. 
We denote by $E^G_\Omega$ the \textbf{orbit equivalence relation} of this action, i.e., where two points of $\Omega$ are equivalent if and only if they are in the same $G$-orbit. 
A Borel \textbf{cross-section} for the $G$-space $\Omega$ is a Borel subset $Z$ of $\Omega$ 
such that for every $x\in \Omega$ the set $\{g\in G \mid g.x\in Z\}$ is a nonempty  closed and discrete subset of $G$.
A Borel cross-section $Z$ for the action is called \textbf{uniformly discrete} if there exists an open neighborhood $U$ of $e_G$ such that $U.z_0$ and $U.z_1$ are disjoint for all distinct $z_0,z_1\in Z$.
Notice that uniformly discrete Borel cross-sections always exist \cite{Kec92}.

Given a Borel cross-section $Z$, we write $E^G_Z$ for the restriction of $E^G_\Omega$ to $Z$, and call it the \dfn{cross-section equivalence relation} on $Z$ induced by the action $G \acts \Omega$. 

\begin{proposition}[Ramsay {\cite[Theorem 6.17]{Ra71}}]\label{measure-class_on_cross-section}
Let $G$ be an lcsc group and $Z$ be a Borel cross-section of a free Borel $G$-space $\Omega$. 
Then for each $G$-invariant measure class $\mathscr{C}$ on $\Omega$ there exists a unique $E^G_Z$-invariant measure class $\mathscr{C}_Z$ on $Z$ whose $E^G_Z$-invariant null sets are precisely the intersections with $Z$ of the $G$-invariant null sets of $\mathscr{C}$. 
We call $\mathscr{C}_Z$ the \dfn{Ramsay measure class} on $Z$ induced by  $\mathscr{C}$.
\end{proposition}

We will use the following characterization of null subsets of mcp actions of lcsc groups.

\begin{lemma}[Nullness in terms of Haar measure]\label{null vs Haar null}
Let $G$ be an lcsc group and $G \acts (X,\mu)$ be an mcp action.
Then a measurable set $A \subset X$ is null if and only if for a.e.\ $x \in X$, the set $\hat{A}^x:=\{g \in G : g.x \in A\}$ is Haar-null.
\end{lemma}

\begin{proof}
Let $\hat{A} \subseteq G \times X$ denote the preimage of $A$ under the multiplication map $G \times X \to X$.
By Fubini's theorem, it suffices to check that $A$ is null if and only if for each $g \in G$, the fiber $\hat{A}_g = g^{-1} A$ is null, which follows from the action being mcp.
\end{proof}

\begin{proposition}\label{mcp class determined by invariant null sets}
Let $G$ be an lcsc group and $G \acts (X,\mu)$ be an mcp Borel action.
Then the measure class $[\mu]$ is uniquely determined by the collection of $G$-invariant $\mu$-null sets, i.e.,
if $\nu$ is a Borel probability measure on $X$ whose class is preserved by the action of $G$, then $\nu \sim \mu$ if and only if $\nu$ and $\mu$ have the same $G$-invariant null sets.
\end{proposition}

\begin{proof}
Suppose that $\nu$ and $\mu$ agree on $G$-invariant null sets and let $A \subseteq X$ be a Borel set.
Letting $\hat{A} \subseteq G \times X$ denote the preimage of $A$ under the multiplication map $G \times X \to X$, note that the set 
\[
B \coloneq \{x \in X : \hat{A}^x \text{ is Haar-null}\}
\]
is $G$-invariant because $\hat{A}^{hx} = \hat{A}^x h^{-1}$ for all $h \in G$.
Thus, $B$ is $\mu$-conull if and only if it is $\nu$-conull, so the conclusion follows from \cref{null vs Haar null}.
\end{proof}

\begin{proposition}[Canonical cross-section measure]\label{pmp_cross-section} 
Let $G\acts (\Omega ,\xi )$ be a free ergodic pmp Borel action of the lcsc group $G$, and let $Z\subseteq \Omega$ be a uniformly discrete Borel cross-section.
Then there exists a unique probability measure $\mu_Z$ on $Z$ in the Ramsay class induced by $[\xi]$, along with a constant $\mathrm{covol}(Z)\in(0,\infty)$ such that 
\begin{equation}\label{eq:cross-section_measure}
\mu_G\times \mu_Z (B)= \mathrm{covol}(Z) \cdot \int _\Omega |B \cap {\Theta^{-1}(\omega)}| \, d\xi (\omega)  
\end{equation}
for every Borel subset $B\subseteq G\times Z$,
where $\Theta : G \times Z \to \Omega$ is the map $\Theta(g,z) = g.z$.
Moreover, the Radon--Nikodym cocycle for $E^G_Z$ with respect to $\mu_Z$ is given by 
\begin{equation}\label{eq:RN_through_modular}
\w_{\mu_Z} (g.z,z) = \Delta _G(g)    
\end{equation}
for $\mu _Z$-a.e.\ $z\in Z$ and all $g\in G$ with $g.z \in Z$. 

We call $\mu_Z$ the \dfn{canonical cross-section measure} on $Z$.
\end{proposition}

\begin{proof}
The proof is identical to the proof of parts (2) and (3) of \cite[Proposition 4.3]{KPV15}; 
note that, while the statement of \cite[Proposition 4.3]{KPV15} assumes that $G$ is unimodular, this assumption is never used in the proof of part (2), and the proof of part (3) in fact establishes \cref{eq:RN_through_modular} and only invokes unimodularity of $G$ to further deduce that the measure on the cross-section is invariant for the cross-section equivalence relation in that case. 
\end{proof}

\subsection{Real flows} 
\label{sect: real flows}
By a \textbf{Borel flow} we mean a Borel $\RR$-space.
Given a \emph{free} Borel flow $\RR \acts \Omega$, we additionally equip each $\RR$-orbit $\RR.\omega$ with the linear order making the orbit map $\RR\to\RR.\omega$, $\lambda\mapsto \lambda .\omega$ an order isomorphism. 
A cross-section $Z$ of the $\RR$-action is called \textbf{bi-recurrent} if the intersection of $Z$ with each $\RR$-orbit has neither a maximum nor a minimum element.
Given a bi-recurrent Borel cross-section $Z\subseteq\Omega$, the associated \textbf{return time function} for $Z$ is the function $\PoincRetTime_Z : Z \to \RR$ given by
\begin{equation}\label{eq:return_fct}
\PoincRetTime_Z(z) \coloneqq \min \{ \lambda \in \RR : \lambda >1 \text{ and } \lambda .z \in Z\}.
\end{equation}
Then the \textbf{Poincaré return map} associated to the cross-section $Z$ is the map $\PoincRetMap_Z : Z \to Z$ given by 
\begin{equation}\label{eq:Poincare_return_map}
\PoincRetMap_Z(z) \coloneqq \PoincRetTime_Z(z).z.
\end{equation}
We define $\actcoc_\Omega : E^{\RR}_\Omega\to \RR$ to be the \textbf{action-cocycle} determined by 
\begin{equation}\label{eq:action-cocycle_for_R-flow}
\actcoc_\Omega (z, y) . y = z
\end{equation}
for all $z,y\in \Omega$ in the same $\RR$-orbit. 
In the case where $Z$ is bi-recurrent and $\PoincRetTime_Z$ takes a constant value $\alpha \in \RR$, we say that the cross-section $Z$ is \textbf{$\alpha$-equispaced}.

\section{Lifting couplings through semi-direct products}

In this section we record a general lemma on lifting couplings through semi-direct products.
We then explain how applying this to the relevant modular semi-direct products allows us to reduce \cref{thm:1-GBS} to \cref{thm:nonunim_nonamen_treeable_ME}. 

\subsection{ME-couplings and semidirect products} 

Here, we show that a measure equivalence between two nonunimodular groups induces a measure equivalence between their modular semi-direct products.

\begin{lemma}\label{lem:lift_ME_semidirect}
Let $G$, $H$, and $A$ be lcsc groups, and let $\mathcal{L}$ be a Polish group acting continuously on $A$ by continuous group automorphisms. Let $\sigma:G\to\mathcal{L}$ and $\tau:H\to\mathcal{L}$ be continuous group homomorphisms, and let $\tilde{G}=G\ltimes_\sigma A$ and $\tilde{H}=H\ltimes_\tau A$ be the associated semidirect product groups, viewed as internal semidirect products. 

Let $(\Omega,\mathscr{C})$ be a strict $(G,H)$-ME-coupling, and suppose that there exists a measurable map $\Psi:\Omega\to\mathcal{L}$ that is a.e.\ bi-equivariant, i.e., for every $g\in G$ and every $h\in H$,
\begin{equation}\label{eq:Psi_biequivariant}
\Psi_{g.\omega.h}=\sigma_g\Psi_\omega\tau_h
\quad\text{for a.e.\ }\omega\in\Omega.
\end{equation}
Let $\mu$ be a representative of $\mathscr{C}$, let $\mu_A$ be a left Haar measure on $A$, and let $\tilde{\mathscr{C}}$ be the measure class of $\mu\times\mu_A$ on $\tilde{\Omega}\coloneqq \Omega\times A$. Then the formulas
\begin{align}
ga.(\omega,c)&\coloneqq (g.\omega,\Psi_\omega^{-1}(a)c), \label{eq:tildeG_act}\\
(\omega,c).hb&\coloneqq (\omega.h,\tau_h^{-1}(c)b) \label{eq:tildeH_act}
\end{align}
define a near $(\tilde{G},\tilde{H})$-biaction on $\tilde{\Omega}$, and the associated continuous $(\tilde{G},\tilde{H})$-biaction on $\MAlg(\tilde{\Omega},\tilde{\mathscr{C}})$ is an algebraic $(\tilde{G},\tilde{H})$-ME-coupling in the sense of \cref{def:algebraic_ME_coupling}.
\end{lemma}

\begin{proof}
It is clear that \cref{eq:tildeH_act} defines a free measure-class-preserving action of $\tilde{H}$ on $(\tilde{\Omega},\tilde{\mathscr{C}})$.
To see that \cref{eq:tildeG_act} defines a near action of $\tilde{G}$, we verify that $g.a.g^{-1}.(\omega,c)=\sigma_g(a).(\omega,c)$ for $\tilde{\mathscr{C}}$-a.e.\ $(\omega,c)$. We have
\begin{align*}
g.a.g^{-1}.(\omega,c)=g.a.(g^{-1}.\omega,c)
=g.(g^{-1}.\omega,\Psi_{g^{-1}.\omega}^{-1}(a)c) &=(\omega,\Psi_{g^{-1}.\omega}^{-1}(a)c)
=(\omega,\Psi_\omega^{-1}(\sigma_g(a))c)
=\sigma_g(a).(\omega,c),
\end{align*}
where the penultimate equality holds for $\mathscr{C}$-a.e.\ $\omega\in\Omega$ by \cref{eq:Psi_biequivariant}.

The left and right near actions commute almost everywhere: for every $ga\in\tilde{G}$ and $hb\in\tilde{H}$, for a.e.\ $\omega\in\Omega$ and every $c\in A$,
\begin{align*}
ga.((\omega,c).hb)
=ga.(\omega.h,\tau_h^{-1}(c)b)
&=(g.\omega.h,\Psi_{\omega.h}^{-1}(a)\tau_h^{-1}(c)b) \\
&=(g.\omega.h,\tau_h^{-1}(\Psi_\omega^{-1}(a))\tau_h^{-1}(c)b) \\
&=(g.\omega.h,\tau_h^{-1}(\Psi_\omega^{-1}(a)c)b)\\
&=(ga.(\omega,c)).hb.
\end{align*}
Thus \cref{eq:tildeG_act,eq:tildeH_act} induce a Boolean $(\tilde{G},\tilde{H})$-biaction on $\mathcal{M}\coloneqq \MAlg(\tilde{\Omega},\tilde{\mathscr{C}})$.

Let $X\coloneqq \Omega/H$ and $Y\coloneqq G\backslash\Omega$, and choose invariant Borel probability measures $\nu_X\in \mathscr{C}/H$ and $\nu_Y\in G\backslash\mathscr{C}$. 
Choosing Borel sections of the quotient maps identifies $(\Omega,\mathscr{C})$ with both $(X\times H,[\nu_X\times\mu_H])$ and $(G\times Y,[\mu_G\times\nu_Y])$. 
It follows that $(\tilde{\Omega},\tilde{\mathscr{C}})=(\Omega\times A,\tilde{\mathscr{C}})$ is mcp Borel isomorphic to both $(X\times\tilde{H},[\nu_X\times\mu_{\tilde{H}}])$ and $(\tilde{G}\times Y,[\mu_{\tilde{G}}\times\nu_Y])$.

Under the first identification, the right $\tilde{H}$-action is just right translation on the $\tilde{H}$-coordinate, so $\mathcal{M}^{\tilde{H}}\cong \MAlg(X,\nu_X)$ and $\mathcal{M}\cong \mathcal{M}^{\tilde{H}}\mathbin{\widehat\otimes}\MAlg(\tilde{H},[\mu_{\tilde{H}}])$ as Boolean $\tilde{H}$-spaces; 
under the first of these isomorphisms, the action of $\tilde{G}$ on $\mathcal{M}^{\tilde{H}}$ is isomorphic to $\tilde{G}$ acting on $\MAlg(X,\nu_X)$ through the quotient map to $G$, and therefore $\nu_X$ corresponds to a $\tilde{G}$-invariant faithful probability measure on $\mathcal{M}^{\tilde{H}}$.

Likewise switching the roles of $\tilde{G}$ and $\tilde{H}$. 
It follows that $\mathcal{M}$ is an algebraic $(\tilde{G},\tilde{H})$-ME-coupling.
\end{proof}

We will need some basic facts about measured groupoids; we refer the reader to \cite{Hahn:HaarMeasureGroupoids} for a thorough treatment. Given a free pmp left action $G\curvearrowright (X,\mu_X)$, the associated left-action groupoid is $\mathcal{G}=G\ltimes X$, with source map $s(g,x)=x$, range map $r(g,x)=g.x$, and Haar measure $\mu_G\times\mu_X$. 
Likewise, a free pmp right action $(Y,\mu_Y)\curvearrowleft H$ gives the right-action groupoid $\mathcal{H}=Y\rtimes H$, with source map $s(y,h)=y.h$, range map $r(y,h)=y$, and Haar measure $\mu_Y\times\mu_H$. 
The modular function of a measured groupoid $(\mathcal{K},\Lambda)$ is the Radon--Nikodym cocycle $\Delta_{\mathcal{K}}=d\Lambda/d\Lambda^{-1}$. 
For the two action groupoids above, a direct computation shows that $\Delta_{\mathcal{G}}(g,x)=\Delta_G(g)$ and $\Delta_{\mathcal{H}}(y,h)=\Delta_H(h)$.

\begin{proposition}
\label{modular_semidirect_ME-reduction}
Let $G$ and $H$ be lcsc groups, and let $\tilde{G}=G\ltimes\mathbf{R}$ and $\tilde{H}=H\ltimes\mathbf{R}$ be their associated modular semidirect products. If $G$ and $H$ are measure equivalent, then $\tilde{G}$ and $\tilde{H}$ are measure equivalent.
\end{proposition}

\begin{remark}
Observe that parts \labelcref{item:uni_amen} and \labelcref{item:nonuni_amen} of \cref{thm:AutTpqME} give a counter-example to the converse of \cref{modular_semidirect_ME-reduction}: while $\Aut(T_{1,1})$ and $\Aut(T_{1,2})$ are not measure equivalent (distinguished by unimodularity), their modular semi-direct products become measure equivalent (being both amenable and unimodular).
\end{remark}

\begin{proof}
Assume that $G$ and $H$ are measure equivalent. If one of these groups is unimodular, then so is the other by \cite[Corollary 16]{Koivisto-Kyed-Raum-ME-non-unim}, and in that case the claim is immediate because the modular semidirect products are just direct products. We may therefore assume that both $G$ and $H$ are nonunimodular, and in particular non-discrete.

By \cite[Theorem A and Lemma 4]{Koivisto-Kyed-Raum-ME-non-unim}, there exist a free pmp left $G$-space $(X,\mu_X)$ and a free pmp right $H$-space $(Y,\mu_Y)$ such that the associated measured groupoids $\mathcal{G}\coloneqq (G\ltimes X,\mu_G\times\mu_X)$ and $\mathcal{H}\coloneqq (Y\rtimes H,\mu_Y\times\mu_H)$ are mcp Borel isomorphic. Thus there is a Borel groupoid isomorphism $\Phi:G\ltimes X\to Y\rtimes H$ such that $\Phi_*(\mu_G\times\mu_X)$ and $\mu_Y\times\mu_H$ lie in the same measure class.

Write $\Phi(g,x)=(\phi(g.x),\alpha(g,x))$ and $\Phi^{-1}(y,h)=(\beta(y,h),\phi^{-1}(y.h))$, where $\phi:(X,\mu_X)\to (Y,\mu_Y)$ is an mcp Borel isomorphism, $\alpha:G\times X\to H$ is a Borel cocycle, and $\beta:Y\times H\to G$ is a Borel cocycle. Then $\phi(g.x)=\phi(x).\alpha(g,x)^{-1}$, and $\phi^{-1}(y.h)=\beta(y,h)^{-1}.\phi^{-1}(y)$, and
\[
\beta(\phi(g.x),\alpha(g,x))=g,
\qquad
\alpha(\beta(\phi(x),h)^{-1},x)^{-1}=h.
\]
Conjugating the right action of $H$ on $Y$ by $\phi$, we obtain an mcp right action of $H$ on $X$, given by $x.h\coloneqq \phi^{-1}(\phi(x).h)$. We then equip $(G\times X,\mu_G\times\mu_X)$ with the mcp $(G,H)$-biaction
\[
g.(g_0,x)\coloneqq (gg_0,x),
\qquad
(g_0,x).h\coloneqq (g_0\beta(\phi(x),h),x.h).
\]
This is a strict $(G,H)$-ME-coupling.

Now $\Phi_*(\mu_G\times\mu_X)$ is another Haar measure on $\mathcal{H}$, with associated modular function $\Delta_{\mathcal{G}}\circ\Phi^{-1}$. By \cite[Corollary 3.14]{Hahn:HaarMeasureGroupoids}, there exists a Borel map $u:Y\to \RR$ such that $\Delta_{\mathcal{H}}=(u\circ r)(\Delta_{\mathcal{G}}\circ\Phi^{-1})(u\circ s)^{-1}$ almost everywhere on $\mathcal{H}$, i.e., 
\[
\Delta_H(h)=u(y)\Delta_G(\beta(y,h))u(y.h)^{-1}
\]
for a.e.\ $(y,h)\in Y\times H$.

Define $\Psi:G\times X\to \RR$ by $\Psi(g_0,x)\coloneqq \Delta_G(g_0)u(\phi(x))$. Then for every $g\in G$ we have $\Psi(g.(g_0,x))=\Delta_G(g)\Psi(g_0,x)$ for all $(g_0,x)\in G\times X$. For every $h\in H$, we also have
\begin{align*}
\Psi((g_0,x).h)
&=\Psi(g_0\beta(\phi(x),h),x.h) \\
&=\Delta_G(g_0)\Delta_G(\beta(\phi(x),h))u(\phi(x).h) \\
&=\Delta_G(g_0)u(\phi(x))\Delta_H(h) \\
&=\Psi(g_0,x)\Delta_H(h)
\end{align*}
for a.e.\ $(g_0,x)\in G\times X$. Thus $\Psi$ is $(G,H)$-biequivariant almost everywhere in the sense of \cref{lem:lift_ME_semidirect}, where $\RR$ is viewed as the group of dilations of $\mathbf{R}$.

By \cref{lem:lift_ME_semidirect}, the induced Boolean $(\tilde{G},\tilde{H})$-biaction on $\MAlg((G\times X)\times\mathbf{R},[\mu_G\times\mu_X\times\mu_{\mathbf{R}}])$ is an algebraic $(\tilde{G},\tilde{H})$-ME-coupling. Hence, by \cref{prop:algebraic_to_strict}, the groups $\tilde{G}$ and $\tilde{H}$ are measure equivalent.
\end{proof}

\subsection{Reduction of Theorem \ref{thm:1-GBS} to Theorem \ref{thm:nonunim_nonamen_treeable_ME}}

Let $\Gamma$ be a $1$-GBS group.
Assuming $\Gamma$ is nonamenable and does not contain a finite index subgroup isomorphic to $\mathbf{F}_r\times\mathbf{Z}$, \cite[Proposition 2.6]{Levitt:GBS} implies that the image of $\Gamma$ in $\mathbf{Q}^\times$ under the homomorphism $\delta$ induced by commensuration of a vertex group (see \cite[p.~483]{Levitt:GBS}) is infinite.
As described at the end of \cref{subsec:latticeEmbedding}, after moving to an index $2$ subgroup of $\Gamma$ if necessary, we may assume that this image is in the positive rationals (so that $\delta$ extends to the modular homomorphism of the closure $\overline \Gamma \leq \Aut_\Gamma(T_\Gamma)$), and hence $\Gamma$ is a lattice in the modular semidirect product $\overline \Gamma \ltimes \R$.  

Hence \cref{modular_semidirect_ME-reduction} reduces the measure equivalence classification of $\Gamma$ to that of $\overline \Gamma$, a nonunimodular nonamenable closed subgroup of the automorphism group of a locally finite tree, thus an lcsc treeable group by \cite[Proposition 3.4]{JKLcber} and \cite[B.3.(5)]{CGMTD}.
By \cref{thm:nonunim_nonamen_treeable_ME}, all such groups are measure equivalent to each other, yielding the measure equivalence of all such $1$-GBS groups.

\section{Actions of \texorpdfstring{$\Aut(T_{p,q})$}{Aut(Tpq)} and their cross-sections}

Recall the suspension construction (see \cite[3.3]{JKLcber} and \cite[Sections 2.2 and 2.3]{Gaboriau05}):
Given a free pmp Borel (left) action of $G \coloneq \Aut(T_{p,q})$ on a standard probability space $(\tilde{X},\tilde{\mu})$, consider the diagonal action of $G$ on $\tilde{X} \times T_{p,q}$ and take the quotient 
$G\backslash (\tilde{X}\times T_{p,q})$.
This delivers an acyclic Borel graph $\mathcal{T}_{\mathrm{susp}}$ on 
$G \backslash (\tilde{X}\times V_{p,q})\simeq K_{v_0} \backslash \tilde{X}$
with positive edges given by 
$G\backslash (\tilde{X}\times E^+_{p,q})\simeq K_{e_0}\backslash \tilde{X}$, where $K_{v_0}$ and $K_{e_0}$ are the stabilizers in $G$ of the base vertex $v_0 \coloneq \langle a \rangle$ and the edge $e_0 \coloneq \langle a^q \rangle$, respectively.
Since the action on $\~X$ is free, each $\tilde{x}\in\tilde{X}$ induces an isomorphism $T_{p,q}\simeq G\backslash\left(G.(\{\tilde{x}\} \times T_{p,q} )\right)$ of directed trees, given by $z\mapsto [(\tilde{x},z)]_G$.

The diagonal action of $G$ on $\tilde{X} \times V_{p,q}$ admits a Borel transversal, namely, $X \times \set{v_0}$, where $X \subseteq \tilde{X}$ is a Borel transversal for the action of $K_{v_0}$, and we identify $G\backslash (\tilde{X} \times V_{p,q})$ with $X$.
The graph $\mathcal{T}_{\mathrm{susp}}$ is then a graphing of the cross-section equivalence relation $E^G_X$ on $X$ induced by the action $G \acts \tilde{X}$, and we henceforth denote it by $\mathcal{T}_X$. 

We equip $X$ with the unique probability measure $\mu$ as in \cref{pmp_cross-section}, i.e., satisfying \cref{eq:cross-section_measure}.
This makes $\mathcal{T}_X$ into an mcp graph, and for each positive edge $e$ of $\mathcal{T}_X$ we have 
\[
\w_\mu(\orig_{\mathcal{T}_X}(e),\targ_{\mathcal{T}_X}(e)) = \Delta_G(t) = q/p,
\]
by \cref{eq:RN_through_modular}.
One can check, using \cite[Lemma 5.9]{Chen-Terlov-Tserunyan}, that $\mu$ is precisely the pushforward of $\tilde{\mu}$ via the map $\pi : \tilde{X} \to G\backslash(\tilde{X} \times V_{p,q})$ defined by $x \mapsto [(x,v_0)]_G$. 

\begin{definition}\label{def:TpqTreeing}
A \textbf{$\mu$-homogeneous $T_{p,q}$-treeing} is a directed mcp acyclic Borel graph $ \mathcal T$ on a standard probability space $(X,\mu)$ such that
\begin{itemize}
\item each vertex has $p$-many incoming positive edges of $\mathcal{T}$,
\item each vertex has $q$-many outgoing positive edges of $\mathcal{T}$,
\item $\w_\mu(\orig_ \mathcal T(e),\targ_ \mathcal T(e)) = \tfrac{q}{p}$ for every positive edge $e$ of $ \mathcal T$.
\end{itemize}
\end{definition}

The following shows that every treeing with these properties arises in the described fashion from a pmp $\Aut(T_{p,q})$ action. 
As above, $K$ denotes the stabilizer of $v_0\coloneqq \langle a\rangle\in V_{p,q}$ in $\Aut(T_{p,q})$.

\begin{proposition}[Reconstruction]\label{lem:cocycleTpqTreeingArisesFromAutTpqAction}
Let $E$ be an mcp countable Borel equivalence relation on a standard probability space $(X,\mu)$.
Suppose that $\mathcal{T}$ is a $\mu$-homogeneous $T_{p,q}$-treeing of $E$.

Then there is a free pmp left action of $\Aut(T_{p,q})$ on a standard probability space $(\tilde{X},\tilde{\mu})$ containing $X$ as a Borel subset such that
\begin{itemize}
\item $X$ is a transversal for the $K$-orbits in $\tilde{X}$;
\item $E=E^G_X$ and $\mathcal{T}$ coincides with the treeing $\mathcal{T}_X$; and
\item $\mu(B)=\tilde{\mu}(K.B)$ for all Borel $B\subseteq X$.
\end{itemize}
\end{proposition}

\begin{proof}
Using the Lusin--Novikov Uniformization Theorem, we may find a Borel map $\varphi:E\to V_{p,q}$ such that for each $x\in X$, the map $\varphi_x:[x]_E\to V_{p,q}$ defined by $\varphi_x(y)\coloneqq \varphi(x,y)$ is an isomorphism of the $\mathcal{T}$-component of $x$ with $T_{p,q}$ that sends $x$ to $v_0$.

The map $\alpha:E\to \Aut(T_{p,q})$, defined by $\alpha(y,x)\coloneqq \varphi_y\circ\varphi_x^{-1}$ for $(y,x)\in E$, is then a Borel cocycle for $E$.
We let $E_{\alpha}$ denote the associated skew product equivalence relation on $(X\times\Aut(T_{p,q}),\mu\times\mu_{p,q})$, defined by $((y,h),(x,g))\in E_\alpha$ if and only if $(y,x)\in E$ and $\alpha(y,x)g=h$. 
Here $\mu_{p,q}$ denotes the left Haar measure on $\Aut (T_{p,q})$ with $\mu_{p,q}(K)=1$.

The set $X\times K$ is a Borel transversal for $E_\alpha$:
given $(x,g)\in X\times\Aut(T_{p,q})$, the unique $y\in [x]_E$ such that $(y,\alpha(y,x)g)$ belongs to $X\times K$ is $y=\varphi_x^{-1}(gv_0)$.
Indeed, $(y,\alpha(y,x)g)\in X\times K$ if and only if $v_0=\alpha(y,x)gv_0$, which holds if and only if $y=\varphi_x^{-1}(gv_0)$.

Let $\Aut(T_{p,q})$ act on $X\times\Aut(T_{p,q})$ by $g.(x,h)\coloneqq (x,hg^{-1})$.
This action descends to a Borel action on the quotient space of $X\times\Aut(T_{p,q})$ by $E_\alpha$, which we identify with $\tilde{X}\coloneqq X\times K$.
Explicitly, the action of $\Aut(T_{p,q})$ on $\tilde{X}$ is given by
\[
g.(x,k)=\bigl(\varphi_x^{-1}(kg^{-1}v_0),\alpha(\varphi_x^{-1}(kg^{-1}v_0),x)kg^{-1}\bigr).
\]
Under this action, the set $X\times\{e\}$ is a Borel transversal for the $K$-orbits in $\tilde{X}$.
The action is free, since if $g.(x,k)=(x,k)$ then $\varphi _x^{-1}(kg^{-1}v_0)=x$ and $\alpha(x,x)kg^{-1}=k$, hence $g=e$.

\begin{claim}\label{claim:E=E_X_T=T_X}
Under the identification of $X$ with $X\times \{ e \}\subseteq \tilde{X}$, one has $E=E^G_X$ and $\mathcal{T}=\mathcal{T}_X$.
\end{claim}

\begin{proof}[Proof of \cref{claim:E=E_X_T=T_X}] For $x,y\in X$, we have $((x,e),(y,e))\in E^G_X$ if and only if $g.(x,e)\in K.(y,e)=\{ y \}\times K$ for some $g\in \Aut(T_{p,q})$, if and only if $\varphi_x^{-1}(g^{-1}v_0)=y$ for some $g\in \Aut(T_{p,q})$. 
Since $\varphi_x : [x]_E\to V_{p,q}$ is bijective, this is equivalent to saying that $y\in [x]_E$. 
Thus $E^G_X=E$.

Now let $x,y\in X$. 
By construction of the cross-section treeing $\mathcal{T}_X$, there is a positive edge from $x$ to $y$ in $\mathcal{T}_X$ if and only if there is some $g\in \Aut(T_{p,q})$ with $g^{-1}v_0\in K t v_0$ and $g.(x,e)\in K.(y,e)$. 
As above, this is equivalent to the existence of some $g\in\Aut(T_{p,q})$ with $g^{-1}v_0\in K t v_0$ and $\varphi_x^{-1}(g^{-1}v_0)=y$, and therefore to $\varphi_x(y)$ belonging to $Ktv_0$. 
Since $\varphi_x$ is an isomorphism from the $\mathcal{T}$-component of $x$ onto $T_{p,q}$ sending $x$ to $v_0$, this last condition is equivalent to there being a positive edge from $x$ to $y$ in $\mathcal{T}$. 
Hence $\mathcal{T}=\mathcal{T}_X$.\qedhere[\cref{claim:E=E_X_T=T_X}]
\end{proof}

Equipping $\tilde{X}$ with the measure $\tilde{\mu}\coloneqq \mu\times\mu_K$, where $\mu_K$ is normalized Haar measure on $K$ (which coincides with the restriction of $\mu_{p,q}$ to $K$), it remains to show that the action of $\Aut(T_{p,q})$ on $(\tilde{X},\tilde{\mu})$ is measure preserving.
The action of $K$ is measure preserving, since elements of $K$ act trivially on the first coordinate and by right translation on the second coordinate.
Since $t$ and $K$ generate all of $\Aut(T_{p,q})$, it is enough to show that the action of $t$ is measure preserving.

For $v\in V_{p,q}$, let $\varphi^v:X\to X$ be given by $\varphi^v(x)\coloneqq \varphi_x^{-1}(v)$, and let $G_v$ denote the stabilizer of $v$ in $\Aut(T_{p,q})$, so that $K=G_{v_0}$.
For each $v\in Ktv_0$ fix some $k_v\in K$ such that $k_vtv_0=v$, and let $K(v)\coloneqq k_v(K\cap G_{tv_0})$.
Then $\{k_v\}_{v\in Ktv_0}$ is a transversal for the left cosets of $K\cap G_{tv_0}$ in $K$.

For Borel sets $B\subseteq X$ and $C\subseteq K$, we have
\begin{align}
\nonumber (\mu\times\mu_K)\bigl(t.(B\times C)\bigr)
&= \int_{X\times K}  \mathbf{1}_{B\times C}(t^{-1}.(x,k))\,d(\mu\times\mu_K)(x,k) \\
\nonumber &= \int_X \int_K  \mathbf{1}_B(\varphi_x^{-1}(ktv_0)) \mathbf{1}_C(\alpha(\varphi_x^{-1}(ktv_0),x)kt)\,d\mu_K(k)\,d\mu(x) \\
\nonumber &= \int_X \sum_{v\in Ktv_0}\int_{k\in K(v)}  \mathbf{1}_B(\varphi_x^{-1}(v))\, \mathbf{1}_C(\alpha(\varphi_x^{-1}(v),x)kt)\,d\mu_K(k)\,d\mu(x) \\
\nonumber &= \sum_{v\in Ktv_0}\int_X  \mathbf{1}_B(\varphi^v(x))\mu_{p,q}(\alpha(\varphi^v(x),x)^{-1}Ct^{-1}\cap K(v))\,d\mu(x) .
\end{align}
Applying now mass transport and the fact that $\w(x,y)=q/p$ whenever there is an edge from $x$ to $y$, we continue the computation
\begin{align}
&= \nonumber \sum_{v\in Ktv_0}\frac{q}{p}\int_X  \mathbf{1}_B(y)\sum_{x\in (\varphi^v)^{-1}(y)} \mu_{p,q}(\alpha(y,x)^{-1}Ct^{-1}\cap K(v))\,d\mu(y) \\
\label{eq:aftertransport} &= \frac{q}{p}\int_{y\in B}\sum_{v\in Ktv_0}\sum_{x\in (\varphi^v)^{-1}(y)} \frac{p}{q}\mu_{p,q}(C\cap \alpha(y,x)K(v)t)\,d\mu(y),
\end{align}
where in the last line we used that $\Delta(t)=q/p$.

Note that for $v\in Ktv_0$ and $x\in (\varphi^v)^{-1}(y)$ we have $\alpha(y,x)v=v_0$, hence $\alpha(y,x)k_vt\in K$ and
\[
\alpha(y,x)K(v)t=\alpha(y,x)Kt\cap \alpha(y,x)k_vtK=\alpha(y,x)Kt\cap K.
\]
Therefore $C\cap \alpha(y,x)K(v)t=C\cap \alpha(y,x)Kt$, since $C\subseteq K$.
Thus \cref{eq:aftertransport} is equal to
\begin{equation}\label{eq:int_sum}
\int_{y\in B}\sum_{\{ x\in [y]_E:\varphi_x(y)\in Ktv_0\}} \mu_{p,q}(C\cap \alpha(y,x)Kt)\,d\mu(y).
\end{equation}
As $x$ varies over the set $\{x\in [y]_E:\varphi_x(y)\in Ktv_0\}$, the sets $\alpha(y,x)Kt$ are pairwise disjoint, since the vertices $\alpha(y,x)(v_0)=\varphi_y(x)$ are distinct for distinct $x$. 
Moreover, the union of all these sets $\alpha (y,x)Kt$ consists of automorphisms of the form $gkt$, where $k\in K$ and $g$ takes $v_0$ to one of the $p$-many vertices having an outgoing positive edge with terminus $v_0$.
Thus $K$ is contained in this union, since for $k_0\in K$ we have $k_0=gt$ where $g=k_0t^{-1}$.
It follows that the integrand in \cref{eq:int_sum} is equal to $\mu_{p,q}(C)$, and hence the entire integral in \cref{eq:int_sum} is equal to $(\mu\times\mu_K)(B\times C)$.
Thus, $t$ is measure preserving, and the action is pmp.
\end{proof}

\section{Scalar potentials and potential flows}\label{sec:scalar_potentials}

In this section, we fix an ergodic mcp countable Borel equivalence relation $E$ on a standard probability space $(X,\mu)$ and introduce a class of homomorphism-like functions to $\RR$-flows. 

To motivate, note that $E$ is type $\mathrm{I}$ or $\mathrm{II}$ exactly when the Radon--Nikodym cocycle $\w_\mu$ is a coboundary, i.e., the differential of some Borel function $\pi : X \to \RR$ satisfying the equivariance condition $\pi(x) = \w_\mu(x,y) .\pi(y)$ for a.e.\ $(x,y) \in E$.
When $E$ is type $\mathrm{III}$, the Radon--Nikodym cocycle $\w_\mu$ is not a coboundary, but there still are Borel maps from $X$ to \emph{some} $\RR$-flow satisfying the analogous equivariance condition. A trivial example is the map to the one-point $\RR$-flow. 
A less trivial example, in the case when the values of the Radon--Nikodym cocycle $\w_\mu$ are essentially contained in $\alpha^\Z$, is any constant map to the circle $\RR \acts \RR / \langle \alpha \rangle$.

Such an equivariant map $\pi$ to $\RR$-flows plays a similar role as the map whose differential is $\w_\mu$ in the type $\mathrm{II}$ case in the sense that $\pi$ assigns absolute objects to points as opposed to the purely relative weights given by the Radon--Nikodym cocycle.
The next definition formalizes this idea.

\begin{definition}[Scalar potential, potential flow]\label{def: scalar potential}
For an mcp flow $\RR \acts (\Omega, \eta)$, a Borel map $\pi : X \to \Omega$ is called a \textbf{scalar potential} for $(E, X, \mu)$ with respect to $\eta$ if
\begin{enumerate}[(i),leftmargin=*]
\item \label{scalar-equivariance} (scalar-equivariance) $\pi(x) = \w_\mu(x,y) . \pi(y)$ holds for a.e.\ $x \in X$ and for all $y \in [x]_E$; and 

\item \label{scalar mcp} $\pi_*\mu$ and $\eta$ agree on which $\RR$-invariant Borel sets are null.
\end{enumerate}

\noindent An mcp $\RR$-flow $(\Omega, \eta)$ is called a \textbf{potential flow} of $(E,X,\mu)$ if there is a scalar potential $X \to \Omega$ for $(E,X,\mu)$ with respect to $\eta$.

\end{definition}

\begin{remark}\label{scalar_potential_extension}
Identifying $X$ with $X \times \{1\} \subseteq X \times \RR$, every Borel map $\pi : X \to \Omega$ to an $\RR$-flow $\Omega$ uniquely extends to an $\RR$-equivariant (Borel) map $\overline\pi : X \times \RR \to \Omega$ (given by $\overline\pi(x,r) \coloneq r . \pi(x)$).
For every Borel probability measure $\mu$ on $X$, the pushforward of the measure class $[\overline\pi_* (\mu \times \mu_\RR)] = \overline\pi_* [\mu \times \mu_\RR]$ is preserved by the flow.
\end{remark}

\begin{remark}\label{unique measure-class on scalar potential}
Given a Borel $\RR$-flow $\Omega$ and a Borel map $\pi : X \to \Omega$ satisfying \labelcref{scalar-equivariance}, there is a unique measure-class $[\eta]$ on $\Omega$ preserved by the $\RR$-flow and satisfying condition \labelcref{scalar mcp}.
Indeed, the uniqueness follows from \cref{mcp class determined by invariant null sets}, and the existence follows from \cref{scalar_potential_extension} by taking $[\eta]\coloneqq [\overline{\pi}_* (\mu \times \mu_{\RR})]$.
Therefore, we often only define a Borel map $\pi : X \to \Omega$ and say that it is a scalar potential if it satisfies \labelcref{scalar-equivariance}. 
\end{remark}

In light of \cref{unique measure-class on scalar potential}, we omit writing the measure  $\eta$ on a potential flow unless it is necessary for the discussion. 

\begin{remark}\label{Krieger_flow_is_ergodic}
Since $E$ is ergodic and a scalar potential $\pi : X \to \Omega$ is a homomorphism from $E$ to the orbit equivalence relation of the flow $\RR \acts \Omega$, the flow is also ergodic with respect to the pushforward measure $\pi_* \mu$.
It then follows that the only flows which are not essentially free that can appear as a potential flow are the one-point flow and the multiplication flows on the circles $\RR / \langle \alpha \rangle$. 
\end{remark}

\begin{proposition}[Scalar potential permanence properties]
\label{scalar potential permanence properties}
Let $(\Omega, \eta)$ be an  mcp $\RR$-flow and $\pi : X \to \Omega$ be a scalar potential for $(E, X, \mu)$ with respect to $\eta$.

\begin{enumerate}[(a)]
\item\label{item:changing by isom flow} 

If $\Psi : \Omega \to \Omega'$ is an
$\RR$-equivariant mcp Borel isomorphism of mcp $\RR$-flows $(\Omega,\eta)$ and $(\Omega',\eta')$, then $\pi' \coloneq \Psi \circ \pi$ is a scalar potential for $(E,X,\mu)$ with respect to $\eta'$.

\item \label{item:equiv_measure} 
If $\nu$ is a sigma-finite measure on $X$ equivalent to $\mu$, then the map $\pi' : X \to \Omega$ given by $x \mapsto \tfrac{d \nu}{d \mu}(x) . \pi(x)$ is a scalar potential for $(E,X,\nu)$ with respect to $\eta$.

\item \label{item:whole_to_section} If $Y \subseteq X$ is a complete section, then $\pi' \coloneq \pi|_Y :Y \to \Omega$ is a scalar potential for $(E|_Y, Y, \mu|_Y)$ with respect to $\eta$.

\item \label{item:section_to_whole} Conversely, if $Y \subseteq X$ is a complete section and $\pi':Y \to \Omega$ is a scalar potential for $(E|_Y, Y, \mu|_Y)$ with respect to $\eta$, then there is a unique scalar potential $X \to \Omega$ for $(E,X,\mu)$ with respect to $\eta$, which agrees with $\pi'$ on $Y$.
\end{enumerate}

In particular, the category of potential flows of $(E,X,\mu)$ is a stable isomorphism invariant. 
\end{proposition}

\begin{proof}
Parts \labelcref{item:changing by isom flow,item:equiv_measure,item:whole_to_section} are straightforward, hence we only prove \labelcref{item:section_to_whole}. 
By the Feldman--Moore theorem \cite{FM77a}, there exists a Borel $E$-retract $f:X \to Y$, i.e., $f |_Y$ is the identity and $(x,f(x)) \in E$ for all $x \in X$, then, $\pi (x):=\w_\mu (x,f(x)).\pi '(f(x))$ is the (unique) desired scalar potential.
\end{proof}

For our arguments we will need the finest scalar potential, which we now define.

\begin{definition}[Krieger scalar potential]\label{def: Krieger=universal scalar potential}
A \textbf{Krieger scalar potential} for $(E,X,\mu)$, denoted $\pi_{E,\mu} : X \to \K_E$, is an initial scalar potential for $(E,X,\mu)$ with potential flow $(\K_E, \eta_E)$, that is, a scalar potential satisfying the following universal property: each scalar potential $\pi : X \to \Omega$ (uniquely) essentially factors through $\pi_{E,\mu}$, 
i.e., there is an essentially unique mcp Borel map $h : \K_E \to \Omega$ that is $\RR$-equivariant on an $\RR$-invariant $\eta_E$-conull set, and satisfies $\pi = h \circ \pi_{E,\mu}$ a.e. 
\begin{center}
\begin{tikzcd}
{}\arrow[dr, phantom, "\circlearrowleft", near end] & \K_E \ar[d, dashed, "\exists ! h"] \\
X \ar[r, swap, "\pi"] \ar[ur, "\pi_{E,\mu}"] & \Omega 
\end{tikzcd}
\end{center}
We sometime omit either $E$ or $\mu$ from the notation of $\pi_{E,\mu}$ and write $\pi_E$ or $\pi_\mu$ where $\mu$ or $E$ are clear from the context.
\end{definition}

We will postpone the proof of the existence of a Krieger scalar potential until \cref{section:constructing-Krieger-flow}. 
First note that the universal property implies the essential uniqueness of Krieger scalar potentials (up to Borel isomorphism of $\RR$-flows, as in \cref{scalar potential permanence properties}\labelcref{item:changing by isom flow}), hence we refer to it as \emph{the} Krieger scalar potential.
We will call the mcp action of $\RR$ on the codomain $(\K_E, \eta_E)$ of the Krieger scalar potential the \textbf{Krieger flow} of $(E,X,\mu)$. This is the name that appears in the literature, for example in \cite{HamachiOsikawa:ErgGroupsAndKrieger, KatznelsonWeiss:NonsingularTransformations}, and it is also referred to as Poincaré flow in \cite[p.~322]{FM77a}.

\begin{example}\label{w_values_in_p/q_Krieger_factors_onto_p/q-circle}
Let $(E,X,\mu)$ be an ergodic mcp countable Borel equivalence relation and let $p,q \ge 1$ are natural numbers.
If $\w_\mu([x]_E\times\{x \} ) \subseteq \langle\frac{p}{q}\rangle$, then the Krieger flow of $(E,X,\mu)$ factors onto the circle $\RR \acts \RR/\langle\frac{p}{q}\rangle$.
Indeed, this is simply because every constant map $\pi : X \to \RR/\langle\frac{p}{q}\rangle$ is a scalar potential.
\end{example}

\begin{remark}\label{krieger permanence properties}
It is straightforward to check that in \cref{scalar potential permanence properties}, if $\pi$ is the Krieger scalar potential then the resulting maps $\pi'$ in all parts \labelcref{item:changing by isom flow}--\labelcref{item:section_to_whole} are the Krieger scalar potential for the respective equivalence relations in \labelcref{item:changing by isom flow}--\labelcref{item:section_to_whole}.
Therefore, the Krieger flow is a stable isomorphism invariant of $(E,X,[\mu])$.
\end{remark}

\begin{proposition}\label{prop:subrelation_scalar_potentials}
Let $F \leq E$ be mcp equivalence relations on a standard probability space $(X,\mu)$. 
Then:

\begin{enumerate}[(a), leftmargin=*]
\item \label{item:subset_scalar_potentials} 
Every scalar potential for $(E,X,\mu)$ is also a scalar potential for $(F,X,\mu)$.
Thus, there is a natural ($\RR$-equivariant mcp Borel) factor map $\iota_{F,E} : (\K_F, \eta_F) \to (\K_E, \eta_E)$ from the Krieger flow of $(F,X,\mu)$ to the Krieger flow of $(E,X,\mu)$.

\item \label{item:same_scalar_potentials} If furthermore $\ker_F \w_\mu$ is ergodic relative to $\ker_E \w_\mu$ and 
$
\w_\mu ([x]_F\times\{x\}) = \w_\mu ([x]_E\times\{x\})
$
for a.e.\ $x \in X$, then conversely, every scalar potential for $(F,X,\mu)$ is also a scalar potential for $(E,X,\mu)$.
Moreover, $F$ is ergodic relative to $E$, and the natural factor map $\iota_{F,E}$ between the Krieger flows of $F$ and $E$ is an isomorphism.
\end{enumerate}
\end{proposition}

\begin{proof}
The main statement of part \labelcref{item:subset_scalar_potentials} is obvious since \cref{def: scalar potential} is less restrictive for $F$ than for $E$.
Furthermore, the statements about the Krieger flows in both parts \labelcref{item:subset_scalar_potentials,item:same_scalar_potentials} follow from the universal property of the Krieger scalar potentials.

For part \labelcref{item:same_scalar_potentials}, we let $\pi_F : (X, \mu) \to (\Omega, \eta)$ be a scalar potential for $(F,X,\mu)$.
By definition, $\pi_F$ is $\ker_F \w_\mu$-invariant {a.e.}, so by relative ergodicity we may assume that $\pi_F$ is $\ker_E \w_\mu$-invariant after discarding an $E$-invariant Borel null set.
Furthermore, the assumption on $\w_\mu$ implies that for a.e.\ $x \in X$ and for all $y \in [x]_E$, there exists $z \in [x]_F$ with $\w_\mu(z,x) = \w_\mu(y,x)$, so $\w_\mu(z,y) = 1$, {i.e.}, $z \in [y]_{\ker_E \w_\mu}$. 
Hence
$
\pi_F(y) = \pi_F(z) = \w_\mu(z,x).\pi_F(x) = \w_\mu(y,x).\pi_F(x).
$

As for the ergodicity of $F$ relative to $E$, let $Y$ be an $F$-invariant Borel set.
Let $\pi : X \to \K_F$ be a Krieger scalar potential of $(F,X,\mu)$.
Let $\{\omega\}$ be a trivial $\RR$-flow so $\Omega \coloneq \K_F \sqcup \{\omega\}$ is an mcp $\RR$-flow when equipped with the probability measure $\eta \coloneq \mu(Y) \cdot \eta_F + (1-\mu(Y)) \cdot \delta_\omega$.
Define $\tilde\pi : X \to \Omega$ by setting 
$\tilde\pi |_Y \coloneq \pi |_Y$ and $\tilde\pi |_{X \smallsetminus Y} :\equiv \omega$.
It is straightforward to check that $\tilde\pi$ is a scalar potential for $(F,X,\mu)$ with the potential flow $(\Omega, \eta)$.
Hence $\tilde\pi$ is a scalar potential for $(E,X,\mu)$ as well, so $Y = \tilde\pi^{-1}(\K_F)$ is $E$-invariant because $\K_F$ is $\RR$-invariant.
\end{proof}

\begin{proposition}[Krieger flow cross-section]
\label{canonical_potential_section}
Let $(E,X,\mu)$ be an ergodic mcp countable Borel equivalence relation and let $\pi_{E,\mu}:X\to \K_E$ be a Krieger scalar potential, with $[\eta_E]$ denoting the measure-class on $\K_E$.  
There exists an essentially unique (up to an $\RR$-invariant $\eta_E$-null set) Borel subset $Z_{E,\mu} \subseteq \K_E$ such that
\begin{equation}\label{eq:potential_section}
\{s \in  \RR: s.\pi_{E,\mu}(x) \in Z_{E,\mu} \} 
= 
\overline{\w_\mu ([x]_E\times \{ x \} )}.
\end{equation}
for $\mu$-a.e.\ $x \in X$. 
It also follows that $(\pi_{E,\mu})_* \mu$ is supported on $Z_{E,\mu}$.

Moreover, if $\w_\mu$ is dis-closed, then $Z_{E,\mu}$ is a cross-section 
and the measure $(\pi_{E,\mu})_\ast \mu$ is in the Ramsay class of $Z_{E,\mu}$ (see \cref{measure-class_on_cross-section}). 
\end{proposition}

\begin{proof}
Recall that the Effros Borel space $F(\RR)$ is the standard Borel space of all closed subsets of $\RR$.
Equipping $F(\RR)$ with the $\RR$-action $r . A \coloneq A r^{-1}$, for $A \in F(\RR)$ and $r \in \RR$, turns $F(\RR)$ into a Borel $\RR$-flow.
Then the Borel map $\pi : X \to F(\RR)$ defined by $x \mapsto \overline{\w_\mu([x]_E \times \{x\})}$ is a scalar potential, with the default measure class on $F(\RR)$ as in \cref{unique measure-class on scalar potential}; indeed: for a.e.\ $x \in X$ and $y \in [x]_E$,
\[
\pi(x) 
= \overline{\w_\mu([x]_E \times \{x\})} 
= \overline{\w_\mu([y]_E \times \{y\})} \w_\mu(y,x) 
= 
\pi(y) \w_\mu(x,y)^{-1} = \w_\mu(x,y) . \pi(y).
\]
By the universal property of the Krieger flow, $\pi$ factors through to an mcp Borel map $h : \K_E \to F(\RR)$.
Then the set 
\(
Z_{E,\mu} \coloneq \{z \in \K_E: 1 \in h(z)\}
\)
satisfies \labelcref{eq:potential_section} because for each $s \in \RR$, we have
\begin{align*}
s.\pi_{E,\mu}(x) \in Z_{E,\mu} 
&\iff 
1 \in 
h(s . \pi_{E,\mu}(x))
=
s . \pi(x) 
= 
\pi(x) s^{-1}
\\
&\iff 
s \in \pi(x) 
= 
\overline{\w_\mu([x]_E \times \{x\})}.
\end{align*}

To show essential uniqueness, let $Z_{E,\mu}'$ be another Borel set satisfying \labelcref{eq:potential_section}, and put $A \coloneq Z_{E,\mu} \mathbin{\triangle} Z'_{E,\mu}$. 
Then for $\mu$-a.e. $x \in X$, the set $\{s \in \RR  : s.\pi_{E,\mu}(x) \in A\}$ is empty, and hence $\RR.A$ is $(\pi_{E,\mu})_\ast \mu$-null.
But $(\pi_{E,\mu})_\ast \mu$ and $\eta_E$ have the same $\RR$-invariant null sets, so $\RR.A$ is also $\eta_E$-null, proving the essential uniqueness of $Z_{E,\mu}$.

That $(\pi_{E,\mu})_* \mu$ is supported on $Z_{E,\mu}$ follows from \labelcref{eq:potential_section} because for a.e.\ $x \in X$, we have 
$1 \in \overline{\w_\mu([x]_E \times x)}$, and hence $\pi_{E,\mu}(x) \in Z_{E,\mu}$.

Now suppose that $\w_\mu$ is dis-closed. 
It is then clear from \labelcref{eq:potential_section} that after discarding an $\RR$-invariant $\eta_E$-null set, $Z_{E,\mu}$ is a cross-section.
Since $(\pi_{E,\mu})_* \mu$ and $\eta_E$ have the same $\RR$-invariant null sets,
it remains to verify that the cross-section equivalence relation $F \coloneq E^\RR_{Z_{E,\mu}}$ on $Z_{E,\mu}$ induced by the $\RR$-flow preserves the measure $(\pi_{E,\mu})_* \mu$.
To this end, let $\phi \in [F]$ and let $A \subseteq Z_{E,\mu}$ be a $(\pi_{E,\mu})_\ast \mu$-nonnull set.
It then follows from \labelcref{eq:potential_section} that
for $\mu$-a.e.\ $x \in \pi_{E,\mu}^{-1}(A)$, we have $\phi(\pi_{E,\mu}(x)) = \pi_{E,\mu}(y)$ for some $y \in [x]_E$.
Luzin--Novikov uniformization (or the Feldman--Moore theorem) then yields a pseudogroup element $\psi \in [[E]]$ with $\mu$-nonnull domain $\dom(\psi) \subseteq \pi_{E,\mu}^{-1}(A)$ such that $\phi(\pi_{E,\mu}(x)) = \pi_{E,\mu}(\psi(x))$ for all $x \in \dom \psi$. 
Since $E$ is mcp, $\im (\psi)$ is also $\mu$-nonnull, and thus $\phi(A) \supseteq  \pi_{E, \mu}(\im (\psi))$ is $(\pi_{E,\mu})_\ast \mu$-nonnull, completing the proof.
\end{proof}

\subsection{Classification by types and Krieger flow}

The usual classification into types II, $\mathrm{III}_\alpha$, $\alpha \in [0,1]$ is equivalent to the following statements (see \cite[Section 8, last paragraph]{FM77a}, \cite[Theorem 19]{HamachiOsikawa:ErgGroupsAndKrieger} for the equivalence), which we will take as definition.

\begin{definition}[Classification by types]\label{defn:classbytype}
Let $ E$ be an ergodic mcp countable Borel equivalence relation on a standard measure space $(X,\mu)$ with $\mu$ atomless. 
Then $\K_E$ is ergodic by \cref{Krieger_flow_is_ergodic} and can be characterized by its stabilizers (closed subgroups of $\RR$) and whether or not it has a conull orbit. 
The classification of types then translates into the following:
\begin{itemize}
\item $E$ is type $\mathrm{II}$ 
if $\K_E$ is isomorphic to $\RR \acts \RR$. 
\item $E$ is type $\mathrm{III}_1$ 
if $\K_E$ is trivial, i.e. isomorphic to $\RR\acts \{\ast\}$.
\item $E$ is type $\mathrm{III}_\lambda$, $\lambda \in (0,1)$, 
if $\K_E$ isomorphic to $\RR \acts \RR/ \langle \lambda \rangle$.
\item $E$ is type $\mathrm{III}_0$ if $\K_E$ is nowhere smooth, equivalently $\K_E$ has no conull orbits and is essentially free. 
\end{itemize}
\end{definition}

The proof of the following proposition was also sketched in \cite[Section 4]{KatznelsonWeiss:NonsingularTransformations}.

\begin{proposition}\label{pi is ergodic decomposition}
Let $\pi_{E,\mu}:X\to \K_E$ be a Krieger scalar potential. If $\w_\mu$ is dis-closed, then $\pi_{E,\mu}$ is an ergodic decomposition of $\ker_E \w_\mu$.
\end{proposition}

\begin{proof}
It suffices to fix an ergodic decomposition $\pi : X \to Z$ of $\ker_E \w_\mu$ and show that $\pi$ is a scalar potential for $(E,X,\mu)$, where $Z$ is identified with a cross-section of some Borel $\RR$-flow.
Indeed, this would then imply that $\pi$ is a Krieger scalar potential by the universal properties of the ergodic decomposition and Krieger scalar potential, because all scalar potentials are $\ker_E \w_\mu$-invariant.

To this end, using the local finiteness of $\w_\mu$, we may define $\tau' : X \to (1, \infty)$ by
\[
\tau'(x) := \min\{r > 1 : \exists y\in[x]_E, \w_\mu(y,x) = r \}.
\]
One verifies that this function is $\ker_E \w_\mu$-invariant and thus, by the universal property of the ergodic decomposition, it factors through $\pi$, i.e., there is a Borel map $\tau : Z \to (1,\infty)$ such that $\tau' = \tau \circ \pi$ a.e. 

Furthermore, Luzin--Novikov uniformization \cite[Theorem 18.10]{Kechris:Classical_DST} yields a Borel map $T' : X \to X$ such that $\w_\mu(T'(x), x) = \tau'(x)$.
Then the composition $\pi \circ T'$ is $\ker_E \w_\mu$-invariant (although $T'$ may not be), so it also factors through $\pi$, i.e., there is a Borel map $T : Z \to Z$ such that $\pi \circ T' = T \circ \pi$ a.e. 
One can verify that $T$ is an invertible mcp transformation (for invertibility, take a Borel map $S': X \to X$ so that $\w_\mu(x, S'(x)) = \tau'(S(x))$ and verify that $x$, $S'(T'(x))$ and $T'(S'(x))$ are all in the same $\ker_E \w_\mu$-class for a.e. $x \in X$).

Now let $\RR \acts (\Omega, \eta)$ be the Poincaré suspension flow (also called the ``flow under a function'' in \cite{KatznelsonWeiss:NonsingularTransformations}) of the transformation $T$ on $(Z, \pi_\ast \mu)$ with return time map $\tau \colon Z \to (1,\infty)$.
By definition, $\RR \acts (\Omega, \eta)$ is an mcp flow for which $Z \subseteq \Omega$ is a Borel cross-section with return time function $\tau$ and Poincaré return map $T$ (defined in \cref{sect: real flows}), and $[\eta]$ is the Ramsay 
measure class
(cf.\ \cref{measure-class_on_cross-section}).

It remains to verify that the ergodic decomposition $\pi : X \to \Omega$ is a scalar potential. 
The condition on the measure class, \cref{def: scalar potential}\labelcref{scalar mcp}, is satisfied verbatim by the corresponding property of the Poincaré suspension flow said above.
It is enough to check the equivariance condition, \cref{def: scalar potential}\labelcref{scalar-equivariance}, for pairs $(T'x, x)$ for a.e.\ $x \in X$ because $\pi$ is $\ker_E \w_\mu$-invariant and $T'$ together with $\ker_E \w_\mu$ generate $E$. 
To this end, observe that for a.e.\ $x \in X$, we have
\[
\w_\mu(T'x,x).\pi(x)= \tau'(x).\pi(x) = \tau(\pi(x)).\pi(x) = T(\pi(x)) = \pi(T'(x)),
\]
where we respectively use the defining property of $T'$, the identity $\tau' = \tau \circ \pi$, the fact that $\tau$ is the return time of $T$ in the flow suspension flow $\RR \acts \Omega$, and the identity $T \circ \pi = \pi \circ T'$.
\end{proof}

\begin{remark}\label{erg_component=ker-class}
When $E$ is an ergodic type $\mathrm{III}_0$ equivalence relation and the Radon--Nikodym cocycle is dis-closed, it follows from \cref{def: scalar potential}\labelcref{scalar-equivariance} and the freeness of the Krieger flow that for each $z \in Z_{E, \mu}$, the preimage $\pi_{E,\mu}^{-1}(z)$ meets every $E$-class in at most one $\ker_E \w_\mu$-class a.e. 
Since $\pi_{E, \mu}$ is an ergodic decomposition of $\ker_E \w_\mu$ by \cref{pi is ergodic decomposition}, a.e.\ ergodic component of $\ker_E \w_\mu$ meets every $E$-class in at most one $\ker_E \w_\mu$-class.

It is worth emphasizing that in case $\mathrm{III}_1$ and in case $\mathrm{III}_\lambda$ with $\lambda \in (0,1)$ with dis-closed cocycle, the first property does not hold: the $\pi_{E,\mu}$-preimage of each $z \in Z_{E,\mu}$ meets infinitely many $\ker_E\w_\mu$-classes within a.e.\ $E$-class. 
\end{remark}

\begin{remark}\label{kw:alternate_con}
The proof of \cref{pi is ergodic decomposition} also yields an alternate construction of the Krieger flow for relations with dis-closed cocycles. 
In \cite[Section 4]{KatznelsonWeiss:NonsingularTransformations}, the authors say "It will be convenient to have an alternate, more concrete description of this flow" and proceed with this alternate construction.
\end{remark}

Assume now the Radon--Nikodym cocycle $\w_\mu$ is dis-closed (which can be ensured when $E$ is type $\mathrm{III}_0$ or $\mathrm{III}_\lambda$ for $\lambda \in (0,1)$, see  \cite[Propositions 2.2 and 2.3]{KatznelsonWeiss:NonsingularTransformations}). 
Let $\pi_{E,\mu}:X\to \K_E$ be a Krieger scalar potential for $(E,X,\mu )$ and let $Z_{E,\mu} \subseteq \K_E$ be the cross-section provided by \cref{canonical_potential_section}. 
We let $\PoincRetMap_{Z_{E, \mu}} : Z_{E, \mu} \to Z_{E, \mu}$ denote the associated Poincaré return map. 
By \cref{def: scalar potential}\labelcref{scalar-equivariance}, this map is explicitly defined by
\[
\PoincRetMap_{Z_{E, \mu}}(\pi_{E, \mu}(x)) \coloneqq \w_{\mu}(y,x) . \pi_{E, \mu}(x),
\]
where $x \in X$ and $y \in [x]_E$ is such that $\w(y,x)$ is the least value in $\w ([x]_E\times \{ x \} )$ greater than $1$, i.e.,\ $y$ belongs to the next $\ker_E \w_\mu$-class in $[x]_E$.
When $E$ is type $\mathrm{III}_0$, the Poincaré return map $\PoincRetMap_{Z_{E, \mu}}$ is aperiodic because such is the Krieger flow.

We say the Radon--Nikodym cocycle $\w_\mu$ is \textbf{balanced} if the Poincar\'e return map $\PoincRetMap_{Z_{E,\mu}}$ preserves the measure $(\pi_{E,\mu})_\ast \mu$.

\subsection{\texorpdfstring{Potential flows arising from pmp $G$-spaces}{Potential flows arising from pmp G-spaces}}

In this section we observe that all potential flows of cross-section equivalence relations of free \emph{pmp} actions of nonunimodular groups admit an $\RR$-invariant probability measure in its class.

\begin{proposition}\label{G-equiv_maps<=>scalar-potentials}
Let $G \acts (\Sigma,\nu)$ be a free ergodic pmp action of an lcsc group $G$.
Let $X \subseteq \Sigma$ be a uniformly discrete Borel cross-section and $\nu_X$ be the canonical cross-section probability measure (as in \cref{pmp_cross-section}).
Let $(\Omega, \eta)$ be any mcp $\RR$-flow. 
Then, the following are equivalent: 

\begin{enumerate}[(1)]
\item \label{item:exists_scalar_pot} There exists a scalar potential $\pi : X \to \Omega$ for $(E^G_X, X, \nu_X)$ with respect to $\eta$, where $E^G_X$ is the cross-section equivalence relation on $X$ induced by $G \acts \Sigma$.

\item \label{item:exists_G-equiv_map} There exists a $G$-invariant conull set $\Sigma' \subset \Sigma$ and a Borel $G$-equivariant map  $\hat\pi : \Sigma' \to \Omega$ (where $G$ acts on $\Omega$ through $\Delta_G:G \to \RR$) such that $\hat\pi_* \nu$ and $\eta$ have the same $\RR$-invariant null sets.
\end{enumerate}
In fact, given $\hat\pi$ as in \labelcref{item:exists_G-equiv_map}, then $\hat\pi |_X$ satisfies \labelcref{item:exists_scalar_pot}, and conversely, given $\pi$ as in \labelcref{item:exists_scalar_pot}, it admits an essentially unique extension to a $G$-equivariant map $\hat\pi$ satisfying \labelcref{item:exists_G-equiv_map}, defined by $\hat{\pi}(g . x) \coloneq \Delta_G(g) . \pi(x)$.
\end{proposition}

\begin{proof}
This is a straightforward verification relying on the fact that by \labelcref{eq:RN_through_modular}, we have $\w_{\nu_X}(g.x, x) = \Delta_G(g)$ for $\nu_X$-almost every $x \in X$ and all $g \in G$ such that $g.x \in X$.
\end{proof}

\begin{theorem}[pmp $G$-action $\Rightarrow$ pmp potential flows]
\label{th: pmp G-action implies potential flow}
Let $G \acts (\Sigma,\nu)$ be a free ergodic pmp action of a nonunimodular lcsc group $G$.
Let $X \subseteq \Sigma$ be a uniformly discrete Borel cross-section and $\nu_X$ be the canonical cross-section probability measure (as in \cref{pmp_cross-section}).
Let $E^G_X$ be the cross-section equivalence relation on $X$ induced by the action $G \acts \Sigma$.
Then every potential flow $(\Omega, \eta)$ for $(E^G_X, X, \nu_X)$ admits an $\RR$-invariant probability measure equivalent to $\eta$. 
\end{theorem}

\begin{proof}
Let $\pi : X \to \Omega$ be a scalar potential for $(E^G_X, X, \nu_X)$ with respect to $\eta$. 
Discarding a $G$-invariant null set from $\Sigma$, let $\hat\pi : \Sigma \to \Omega$ the $G$-equivariant map as in \cref{G-equiv_maps<=>scalar-potentials}, so $\hat{\pi}_* \nu$ is preserved by the modular image $\Delta_G(G)$, since $\nu$ is preserved by the action of $G$. 
Thus, $\hat{\pi}_* \nu$ is preserved by the closure $\overline{\Delta_G(G)}$ because, for the induced Borel $\RR$-flow on the space of Borel probability measures of $\Omega$, the stabilizers are closed by \cite[Theorem 9.17]{Kechris:Classical_DST}.

Since $G$ is nonunimodular, $\overline{\Delta_G(G)}$ is a non-trivial closed subgroup of $\RR$, which is hence either $\RR$ itself 
or $\langle \alpha \rangle$ for some element $\alpha \in (0,1)$.
In the first case, $\hat{\pi}_* \nu$ is the desired measure on $\Omega$ by the uniqueness statement of \cref{unique measure-class on scalar potential}.
In the latter case, we let $\mu$ be the convolution of $\hat{\pi}_* \nu$ by the normalized restriction of $\mu_\RR$ to $(\alpha, 1]$, namely, for a Borel $A \subseteq \Omega$,
\[
\mu(A) \coloneq \mu_\RR((\alpha, 1])^{-1} \int_{(\alpha, 1]} \hat{\pi}_* \nu(r^{-1} A) \;d \mu_\RR(r).
\]
It follows from the $\RR$-invariance of $\mu_\RR$ and $\langle \alpha \rangle$-invariance of $\hat{\pi}_* \nu$ that $\mu$ is an $\RR$-invariant probability measure.
It is also immediate from the definition of $\mu$ that $\mu$ and $\hat{\pi}_* \nu$ have the same $\RR$-invariant null sets.
Hence \cref{unique measure-class on scalar potential} implies that $\mu$ is as desired.
\end{proof}

\subsection{The Mackey--Maharam extension and the existence of the Krieger flow}\label{section:constructing-Krieger-flow}

Let $E$ be an ergodic mcp countable Borel equivalence relation on a standard probability space $(X,\mu)$.
We define an equivalence relation $E^\M_\mu$ on $(X \times \RR, \mu \times \mu_{\RR})$ by setting
\[
(x,{s}) \mathrel{E^\M_\mu} (y, {t}) \Longleftrightarrow x \mathrel{E}y \text{ and } \w_\mu(x,y) .{s} =  {t},
\]
and we call it the \textbf{Mackey--Maharam extension} of $(E, X, \mu)$.
Maharam defined the same equivalence relation but uses the (equivalent) Haar measure for the additive reals in order to make $E^\M_\mu$ measure preserving \cite[Theorem 1, 3.2]{Maharam:Extension1964}, which we do not need in this article.
On the other hand, Mackey defines such an extension more generally for any cocycle to an lcsc group, and we use the special case of his construction for the Radon--Nikodym cocycle $\w_\mu$ \cite[p.~198]{Mackey-1966-Erg-th-virt-gps}. 
In both Maharam's and Mackey's cases $E$ arises from a group action, while for mcp equivalence relations this construction is explicitly discussed in \cite[Section 8]{FM77a}.

Consider the flow $\RR \acts X \times \RR$ given by multiplication on the second coordinate. Notice that $E^M_\mu$ is preserved by this flow.

\begin{proposition}
[Existence of the Krieger scalar potential]
Every ergodic mcp countable Borel equivalence relation $(E,X,\mu)$ admits a Krieger scalar potential $\pi_{E,\mu} : (X, \mu) \to (\K_E, \eta_E)$.
\end{proposition}
\begin{proof}
Applying \cite[Theorems 1 and 2]{Mackey:PointRealizations} to the sigma-algebra of $\mathrel{E^\M_\mu}$-invariant sets, we obtain an mcp Borel flow $\RR \acts (\Omega, \eta)$ and an $\RR$-equivariant mcp Borel map $\Pi: X \times \RR \to \Omega$.

Then $\pi:X \to \Omega$ defined by $\pi(x) = \Pi(x,1)$ satisfies that for $\mu$-almost every $x\in X$ and for every $y \in [x]_E$,
\[
\pi(y) = \Pi(y,1) = \Pi(x,\w_\mu(y,x))= \w_\mu(y,x).\Pi(x,1) = \w_\mu(y,x).\pi(x).
\]
Furthermore, if $A \subseteq \Omega$ is $\RR$-invariant, then $\Pi^{-1}(A) = \pi^{-1}(A) \times \RR$, so $A$ is $\eta$-null exactly when $\pi^{-1}(A)$ is $\mu$-null, as required by the Definition of scalar potential (\cref{def: scalar potential}\labelcref{scalar mcp}).
Thus, $\pi: X \to (\Omega, \eta)$ is a scalar potential. 
We now proceed to showing that it satisfies the universal property (\cref{def: Krieger=universal scalar potential}).

Let $\pi_0 : X \to (\Omega_0, \eta_0)$ be an arbitrary scalar potential. 
By \cref{scalar_potential_extension}, the measure algebra $\mathrm{MALG}_{\eta_0}(\mathcal{B}(\Omega_0))$ is isomorphic to an $\RR$-invariant sub-algebra of the measure algebra of $E^\M_\mu$-invariants sets. Since $\Omega$ is a point realization of the measure algebra of \emph{all} $E^\M_\mu$-invariants sets, $\mathrm{MALG}_{\eta_0}(\mathcal{B}(\Omega_0))$ is isomorphic to an $\RR$-invariant subalgebra of $\mathrm{MALG}_{\eta}(\mathcal{B}(\Omega))$. 
Applying \cite[Theorem 2]{Mackey:PointRealizations} yields an mcp $\RR$-equivariant Borel map $\Omega \to \Omega_0$, as required for the universal property.
\end{proof}

The following is immediate from the above proof:

\begin{corollary}
For an ergodic mcp countable Borel equivalence relation $(E,X,\mu)$, the unique extension (as in \cref{scalar_potential_extension}) of the Krieger scalar potential $\pi_E : X \to \K_E$ to an $\RR$-equivariant map $\overline\pi_E : X \times \RR \to \K_E$ is an ergodic decomposition map for the Mackey--Maharam extension $(E^\M_\mu, X \times \RR, \mu \times \mu_\RR)$.
\end{corollary}

\subsection{Changing target measures of scalar potentials}\label{subsec:changing_measures}

\begin{theorem}
\label{changing measure to reach a pmp in Krieger flow}
Let $E$ be an ergodic mcp type $\mathrm{III}_0$ countable Borel equivalence relation on a standard probability space $(X,\mu)$. 
Suppose that the measure class $[\eta_E]$ of its Krieger flow $\K_E$ contains an invariant probability measure, and let $Z$ be a uniformly discrete Borel cross-section of $\RR\curvearrowright\K_E$.
Then there exists an equivalent probability measure $\mu' \sim \mu$ and a Krieger scalar potential $\pi' : X \to \K_E$ for $(E,X,\mu')$
such that the pushforward measure $\pi'_* \mu'$ is in the Ramsay measure class on $Z$ induced by $[\eta_E]$ (see \cref{pmp_cross-section}) and is preserved by the cross-section equivalence relation $E^\RR_Z$ induced on $Z$ by the $\RR$-action.
\end{theorem}

\begin{proof}
The proof of this theorem follows from \cref{change target cross-section,change target measure} proved below together with \cref{pmp_cross-section} as follows.
Let $\pi : X \to \K_E$ be a Krieger scalar potential for $(E,X,\mu)$, and note that the Krieger flow $(\K_E, \eta_E)$ is ergodic (by \cref{Krieger_flow_is_ergodic}) and nowhere smooth because $E$ is type $\mathrm{III}_0$, see \cref{defn:classbytype}.
Then \cref{change target cross-section} gives an equivalent measure $\tilde \mu \sim \mu$ on $X$ and a Krieger scalar potential $\tilde \pi : X \to \K_E$ for $(E, X, \tilde\mu)$ such that the pushforward measure $\tilde\pi_* \tilde \mu$ is in the Ramsay measure class on $Z$ induced by $\eta_E$.
By \cref{pmp_cross-section}, there is a probability measure $\xi$ on $Z$ in the Ramsay class induced by $[\eta_E]$ which is preserved by the equivalence relation $E^\RR_Z$ induced on $Z$ by the $\RR$-action.
Finally, applying \cref{change target measure} to $\tilde \pi$ on $(E,X,\tilde \mu)$, we get an equivalent measure $\mu' \sim \mu$ and a Krieger scalar potential $\pi'$ for $(E,X,\mu')$ with $(\pi')_* \mu' = \xi$, as desired.
\end{proof}

The following lemma can be viewed as a generalization and strengthening (in terms of scalar potentials and for arbitrary equivalence relations) of a statement for $\Z$-actions about discretization of the Radon--Nikodym cocycle, which goes back to Krieger \cite[Lemma 6.5]{Krieger:Flow} and revisited in \cite[Proposition 2.3]{KatznelsonWeiss:NonsingularTransformations}.

\begin{lemma}[Changing the target measure class]\label{change target cross-section}
Let $E$ be an ergodic mcp type $\mathrm{III}_0$ countable Borel equivalence relation on a standard probability space $(X,\mu)$.
Let $\pi : (X,\mu) \to (\Omega, \eta)$ be a scalar potential for $(E,X,\mu)$, where the potential flow $\RR \acts (\Omega,\eta)$ is nowhere smooth.
Let $Z$ be a Borel cross-section of $\Omega$.

Then there exists an equivalent probability measure $\tilde \mu \sim \mu$ such that the scalar potential $\tilde\pi : X \to \Omega$ given by $x \mapsto \frac{d\tilde\mu}{d\mu}(x) . \pi(x)$ for $(E,X,\tilde\mu)$ satisfies that $\tilde\pi_* \tilde \mu$ is in the Ramsay measure class on $Z$ induced by $\eta$.
In particular, if $(\Omega, \eta)$ is a Krieger flow for $(E,X,\mu)$ then $\tilde\pi$ is a Krieger scalar potential for $(E,X,\tilde\mu)$.
\end{lemma}

\begin{proof}
We proceed in three steps.

\noindent \textit{Step 1.}
For every $x\in X$, let 
\[
\tau_1(x) \coloneq \text{ the largest $r < 1$ in $\RR$ such that $r . \pi(x) \in Z$}. 
\]
Let $\pi_1 : X \to \Omega$ be defined by $x \mapsto \tau_1(x).\pi(x)$. 
This map lands inside $Z$.
Define the measure $\mu_1$ by $d \mu_1 \coloneq \tau_1 d\mu$, i.e.\ $\frac{d \mu_1}{d \mu} \coloneqq \tau_1$, and note that $\mu_1 \sim \mu$ and $\mu_1(X) < 1$ since $0 < \tau_1 < 1$.
Now $\pi_1$ is a scalar potential for $(E,X,\mu_1)$ with respect to $\eta$, by \cref{scalar potential permanence properties}\ref{item:equiv_measure}, and it is a Krieger scalar potential if $\pi$ is, by \cref{krieger permanence properties}.
Since the image $\pi_1(X)$ lands inside the cross-section $Z$, it follows that the Radon--Nikodym cocycle $w_{\mu_1}$ is dis-closed.

\smallskip

\noindent \textit{Step 2.}
By \cref{prop:kernelIsNowhereSmooth}, $\ker_E w_{\mu_1}$ is $\mu_1$-nowhere smooth.
Since $\pi_1$ is $\ker_E w_{\mu_1}$-invariant, the measure disintegration $z \mapsto (\mu_1)_z$ of $\mu_1$ through $\pi_1$ satisfies that $(\mu_1)_z$ is atomless for $(\pi_1)_* \mu_1$-a.e.\ $z \in Z$.
Thus \cref{sets_with_prescribed_fiber-measures} applies to the measure disintegration $z \mapsto (\mu_1)_z$, yielding a partition of $X = \sqcup_{n \ge 1} X_n$ into Borel sets such that $(\mu_1)_z(X_n) = 2^{-n}$ for each $n \ge 1$ and for $(\pi_1)_* \mu_1$-a.e.\ $z \in Z$.

We now define $\tau_2 : X \to \RR$ by setting, for each $n \in \N$ and $x \in X_n$,
\[
\tau_2(x) \coloneqq \text{ the $n^\text{th}$ largest $r < 1$ in $\RR$ such that $r . \pi_1(x) \in Z$}. 
\]

\noindent
Let $\pi_2 : X \to \Omega$ be defined by $x \mapsto \tau_2(x).\pi_1(x)$, and note that $\pi_2 |_{X_n} = \PoincRetMap_Z^{-n} \circ (\pi_1 |_{X_n})$ for each $n \ge 1$, where $\PoincRetMap_Z$ is the Poincaré return map associated to $Z$.
Define the measure $\mu_2$ by setting $d \mu_2 \coloneq \tau_2 d\mu_1$, i.e.\ $\frac{d \mu_2}{d \mu_1} \coloneqq \tau_2$, and note that $\mu_2 \sim \mu$ and $\mu_2(X) < 1$ since $0 < \tau_2 < 1$.
Again, $\pi_2$ is a scalar potential for $(E,X,\mu_2)$ with respect to $\eta$, by \cref{scalar potential permanence properties}\ref{item:equiv_measure}, and it is a Krieger scalar potential if $\pi$ (hence $\pi_1$) is, by \cref{krieger permanence properties}.

\smallskip

\noindent \textit{Step 3.}
We now prove that $(\pi_2)_* \mu_2$ is in the Ramsay measure class on $Z$ induced by $[\eta]$.
Since $(\pi_2)_* \mu_2$ and $\eta$ have the same $\RR$-invariant null sets, it suffices to check that the measure class of $(\pi_2)_* \mu_2$ is preserved by $E_Z^{\RR}$, equivalently, by the Poincaré return map $\PoincRetMap_Z$.

Put $\nu \coloneq (\pi_1)_* (\mu_1)$ and observe that $(\pi_1)_* (\mu_1 |_{X_n}) = 2^{-n} \nu$ by the definition of the sets $X_n$ and measure disintegration.
Because $\pi_2 |_{X_n} = \PoincRetMap_Z^{-n} \circ (\pi_1 |_{X_n})$ for all $n \ge 1$, we have: 
\[
(\pi_2)_* \mu_2 \sim (\pi_2)_* \mu_1 = \sum_{n \ge 1} (\pi_2)_* (\mu_1 |_{X_n}) = \sum_{n \ge 1} (\PoincRetMap_Z^{-n})_* \bigl((\pi_1)_* (\mu_1 |_{X_n})\bigr) = \sum_{n \ge 1} 2^{-n}(\PoincRetMap_Z^{-n})_* \nu.
\]
Thus, 
\[
(\PoincRetMap_Z)_* \big((\pi_2)_* \mu_2\big) \sim (\PoincRetMap_Z)_* \left(\sum_{n \ge 1} 2^{-n}(\PoincRetMap_Z^{-n})_* \nu \right) = \sum_{n \ge 1} 2^{-n}(\PoincRetMap_Z^{-n+1})_* \nu = 2^{-1} \sum_{n \ge 1} 2^{-n}(\PoincRetMap_Z^{-n})_* \nu + 2^{-1} \nu,
\]
so it is enough to check that if a set $B$ is null with respect to $\sum_{n \ge 1} 2^{-n}(\PoincRetMap_Z^{-n})_* \nu$, then it is $\nu$-null.
By definition, $\bigcup_{n \ge 1} \PoincRetMap_Z^n (B)$ is $\nu$-null.
But $E^\RR_Z$ is $\nu$-nowhere smooth because the flow $\Omega$ is $\eta$-nowhere smooth and $\nu$ has the same $\RR$-invariant Borel null sets as $\eta$ (being a push-forward under a scalar potential with respect to $\eta$).
Hence $\bigcup_{n \ge 1} \PoincRetMap_Z^n (B)$ is equal to the full $\PoincRetMap_Z$-saturation $\bigcup_{n \in \Z} \PoincRetMap_Z^n (B)$ up to a $\nu$-null set. 
In particular, $B=\PoincRetMap_Z^0 (B)$ is $\nu$-null. 

Thus, $\tilde\mu \coloneqq \mu_2$ and $\tilde\pi \coloneqq \pi_2$ are as desired.
\end{proof}

\begin{lemma}[Changing the target measure to an equivalent one]
\label{change target measure}
Let $E$ be an ergodic mcp type $\mathrm{III}$ countable Borel equivalence relation on a standard probability space $(X,\mu)$.
Let $\pi : X \to \Omega$ be a scalar potential for $(E,X,\mu)$ and let $\xi $ be a probability measure equivalent to $\pi_* \mu$.
There is a probability measure $\mu'$ on $X$ equivalent to $\mu$ and a scalar potential $\pi' : X \to \Omega$ for $(E, X, \mu')$ such that $\pi'_* \mu' = \xi$.
Moreover, if $\pi$ is a Krieger scalar potential for $(E,X,\mu)$ then $\pi'$ is a Krieger scalar potential for $(E,X,\mu')$.
\end{lemma}

\begin{proof}
Let $k : \Omega \to \N; \ z \mapsto  \lceil \frac{d \xi }{d (\pi_* \mu)}(z) \rceil - 1$ and  $f : \Omega \to (0,1]; \ z \mapsto \frac{d \xi }{d (\pi_* \mu)}(z) - k(z)$,
so $f$ is the fractional part of the Radon--Nikodym derivative except when the latter is an integer, in which case $f(z)=1$.

Applying the ``Moreover'' part of \cref{sets_with_prescribed_fiber-measures} to $X$, $\pi$, and $f$, we obtain a Borel set $Y \subseteq X$ such that
the disintegration $\{ \mu_z\}_{z \in Z}$ of the measure $\mu$ with respect to $\pi$ satisfies 
$\mu_z(Y) = f(z)$  for $(\pi_* \mu)$-a.e.\ $z \in \Omega$. Let  
\[
\tilde{X} \coloneq \left(Y \times \{0\}\right) \sqcup  \left\{(x,n) \in X \times \NN: n \le k(\pi(x)) \right\},
\]
and let $\tilde \pi : \tilde X \to \Omega$ be defined as the composition $\pi \circ \proj_X$ of the projection on the first coordinate followed by $\pi$
Let $\chi$ be the counting measure on $\N$ and let $\tilde\mu$ denote the restriction of $\mu \times \chi$ to $\tilde X$.
We show that pushforward $\tilde \pi_* \tilde \mu$ is exactly $\xi $.

For a Borel subset $A \subseteq \Omega$, we compute using Fubini's theorem on $\tilde X \subseteq X\times \N$: 
\begin{eqnarray*}
\tilde\pi_*\tilde\mu(A) &=& \tilde\mu (\proj_X^{-1} \pi^{-1} A) 
\\
&=&
\int_{\pi^{-1}(A)} \sum_{n \in \N} \left( 
\mathbf{1}_{Y \times \{0\}}(x,n)
+
\mathbf{1}_{
0 < n \leq k(\pi(x))
}
\right)d\mu(x)
\\
&=&
\int_{\pi^{-1}(A)} \left( \mathbf{1}_Y(x) + k(\pi(x))\right) d\mu(x)
\\
&=&
\int_A \int_{\pi^{-1}(z)} \left ( \mathbf{1}_Y(x) + k(\pi(x))\right) d\mu_z(x) d(\pi_* \mu)(z)
\\
&=&
\int_A \left(\mu_z(Y) + k(z) \right)d(\pi_* \mu)(z)
\\
&=&
\int_A \frac{d \xi }{d (\pi_* \mu)}(z) d(\pi_* \mu)(z)
=
\xi (A).
\end{eqnarray*}

We now define $\tilde E$ on $\tilde X$ by setting $(x,n) \tilde E (y,m)$ if $x E y$, so $\tilde E$ coincides with $E$ on $Y \times \{0\}\simeq Y$, which is a complete section for both $E$ and $\tilde E$.
Thus, $(E,X,\mu)$ and $(\tilde E,\tilde X,\tilde \mu)$ are stably isomorphic, and hence isomorphic by \cref{typeIII elasticity}.
Let $\varphi : \tilde{X} \to X$ be an isomorphism from $\tilde E$ to $E$.

Also note that $\tilde \pi = \pi \circ \proj_X$ is a scalar potential for $\tilde E$ because $\pi$ is a scalar potential for $E$ and $\w_{\tilde \mu}((x,n), (y,m)) = \w_\mu(x,y)$.
Therefore, setting $\mu' \coloneq \varphi_* \tilde\mu$ and $\pi' \coloneq \tilde\pi \circ \varphi^{-1}$, we see that $\pi'$ is a scalar potential for $(E,X,\mu')$ and $(\pi')_* \mu' = \xi $ as desired.

Finally, the property of being a Krieger scalar potential transfers from 
$\pi$ for $(E,X,\mu)$ to $\pi'$ for  $(E,X,\mu')$ (through 
$\tilde \pi$ for $(\tilde E, \tilde X, \tilde \mu)$)
by \labelcref{item:whole_to_section,item:section_to_whole} of \cref{scalar potential permanence properties} and \cref{krieger permanence properties}.
\end{proof}

\section{Constructing hyperfinite subgraphs with the same Krieger flow}

We say that a Borel graph $\G$ on a standard measure space $(X,\mu)$ is \textbf{mcp} if its connectedness relation $E_\G$ is mcp.
To simplify notation we also denote:
\begin{itemize}
\item $[x]_\G\coloneqq [x]_{E_\G}$;
\item $\ker_\G\w_\mu \coloneqq \ker_{E_\G}\w_\mu$;
\item $\pi_\G \coloneqq \pi_{E_\G}$ and $(\K_\G, \eta_\G )\coloneqq(\K_{E_\G},\eta_{E_\G})$.
\end{itemize} 

For an mcp countable Borel equivalence relation $E$ on a standard measure space $(X,\mu)$ and a Borel function $f : X \to \R$ that is either integrable or non-negative, we denote by $\E[f \mid E]$ the conditional expectation of $f$ with respect to the sigma-algebra of $E$-invariant Borel sets.

\begin{theorem}\label{thm:sameTypeHyperfiniteSubgraph}
Let $\G$ be an ergodic mcp type $\mathrm{III}_0$ locally countable Borel graph on a standard probability space $(X,\mu)$ with dis-closed Radon--Nikodym cocycle $\w_\mu$. There exists an ergodic hyperfinite subgraph $\F \subseteq \G$ such that
\begin{enumerate}[(i)]
\item\label{item:kerrelerg} $\ker_{\F} \w_\mu$ is ergodic relative to $\ker_{\G} \w_\mu$,

\item\label{item:fullimgcocy} $\w_\mu([x]_{\F}\times \{ x\} )=\w_\mu([x]_{\G}\times \{ x \} )$ for a.e.\ $x\in X$.
\end{enumerate}
Moreover, $(E_\F, X,\mu)$ and $(E_\G,X,\mu)$ have the same category of scalar potentials, in particular 
they have the same Krieger scalar potential.

\end{theorem}

\begin{proof}
The graph $\F$ will be constructed as an increasing union of component-finite graphs by an analogous iteration argument to the proof of \cite[Theorem 1.1]{Tserunyan:hyperfin_erg_subgraph}, using a more elaborate version (\cref{fin_almost_erg_subgraph} below) of \cite[Main Lemma 1.6]{Tserunyan:hyperfin_erg_subgraph} for the induction step.
We will use a Krieger scalar potential $\pi_\G : X \to \KKK_\G$ of $(E_\G, X, \mu)$ to guide our construction of $\F$, and we fix a canonical potential cross-section $Z \subseteq \KKK_G$ for $\pi_\G$ as in \cref{canonical_potential_section}.

More precisely, let $D$ be a countable set of bounded functions that is dense in $L^1(X,\mu)$, and let $(f_n)_{n \ge 0}$ be an enumeration of $D$ such that every element appears infinitely often.
We inductively build a sequence $\F_0 \subseteq \F_1 \subseteq \hdots \subseteq \F_n\subseteq \hdots \subseteq \G$ of component-finite Borel subgraphs, starting with $\F_0 \coloneq \emptyset$ and obtaining $\F_{n+1}$ from \cref{fin_almost_erg_subgraph} applied to $\pi_\G$, $Z$, $\H_0 \coloneq \F_n$, $f \coloneq f_n$, and $\varepsilon \coloneq 2^{-n}$.
Finally, let $\F \coloneq \bigcup_{n \in \N} \nearrow \F_n$.

Then the conditional expectations with respect to $\ker_\F \w_\mu$ and $\ker_\G \w_\mu$ agree on $D$, and hence on all of $L^1(X,\mu)$ by continuity.
Furthermore, it follows from \cref{fin_almost_erg_subgraph}\labelcref{cond:seeAllValues} and the equivariance of $\pi_\G$ (condition \cref{def: scalar potential}\labelcref{scalar-equivariance}) that condition \labelcref{item:fullimgcocy} holds for $\F$ as well.

The ergodicity of $\F$ and the ``moreover'' part now follow from \cref{prop:subrelation_scalar_potentials}\labelcref{item:same_scalar_potentials}.
\end{proof}

\begin{corollary}\label{cor:III0_hf_subgraph}
Let $\G$ be an ergodic mcp type $\mathrm{III}_0$ locally countable Borel graph on a standard probability space $(X,\mu)$. 
There exists an ergodic hyperfinite subgraph $\F \subseteq \G$ such that $(E_\F , X,\mu )$ and $(E_\G , X, \mu )$ have the same Krieger scalar potentials, and hence isomorphic Krieger flows.
\end{corollary}

\begin{proof}
Let $\pi : X\to \K_{\mathcal{G}}$ be a Krieger scalar potential for $(E_\mathcal{G},X,\mu )$.
Because $(E,X,\mu)$ is type $\mathrm{III}_0$, the $\RR$-flow $\K_E$ is free.
Letting $Z \subseteq \K_{\mathcal{G}}$ be a Borel cross-section for this flow, \cref{change target cross-section} gives a measure $\tilde \mu \sim \mu$ such that $\tilde{\pi}_* \tilde \mu$ in the Ramsay measure class of $Z$, where $\tilde{\pi}(x)=\tfrac{d\tilde{\mu}}{d\mu}(x).\pi (x)$;
the associated cocycle $\w_{\tilde{\mu}}$ is then dis-closed.
We may therefore apply \cref{thm:sameTypeHyperfiniteSubgraph} to obtain an ergodic hyperfinite subgraph $\mathcal{F}$ of $\mathcal{G}$ with the property that $\tilde{\pi}$ is also a Krieger scalar potential for $(E_\F , X, \tilde{\mu})$.
Using \cref{prop:subrelation_scalar_potentials}\labelcref{item:same_scalar_potentials} we untwist $\tilde\mu$ back to $\mu$ and $\tilde{\pi}$ back to $\pi$ to see that $\pi$ is a Krieger scalar potential for $(E_\F , X, \mu )$. 
\end{proof}

We now restate the main lemma underlying the proof of the existence of hyperfinite ergodic subgraphs \cite[Theorem 1.3]{Tserunyan:hyperfin_erg_subgraph} in an equivalent form in terms of $L^1$-norms.

\begin{lemma}[{\cite[Main Lemma 1.6]{Tserunyan:hyperfin_erg_subgraph}}]\label{lem:Anush1.6}
Let $\G$ be an mcp locally countable Borel graph on a standard probability space $(X,\mu)$. 
Let $\mathcal{L}_0 \leq \G$ be a component-finite Borel subgraph, $f \in L^\infty(X,\mu)$, and $\epsilon>0$.
Then there exists an intermediate component-finite Borel subgraph $\mathcal{L}_0 \leq \mathcal{L}_1 \leq \G$ such that 
\[
\|\E[f \mid E_{\mathcal{L}_1}]-\E[f \mid E_{\G}]\|_1 \leq \epsilon.
\]
\end{lemma}

\begin{lemma}\label{increasing_union_with_finite_cocycle_values}
Let $E$ be an ergodic mcp type $\mathrm{III}_0$ equivalence relation on a standard probability space $(X,\mu)$
Let $\pi : (X, \mu) \to (\Omega,\eta)$ be a scalar potential for $(E,X,\mu)$ such that the $\RR$-flow $(\Omega,\eta)$ is aperiodic and $\pi_* \mu$ is in the Ramsay measure class of some Borel cross-section $Z \subseteq \Omega$ of the $\RR$-flow.
Then, after discarding an $E$-invariant null set, we may express $E$ as an increasing union $E = \bigcup_{n \in \N}\nearrow E_n$ of Borel subequivalence relations such that for each $n\in \N$ the cardinality of $\w_\mu([x]_{E_n} \times \{x\})$ is uniformly bounded in $x$.
\end{lemma}

\begin{proof}
By the aperiodicity of the $\RR$-flow $\Omega$, the Poincaré return map $\PoincRetMap_Z$ on $Z$ is aperiodic, so the cross-section equivalence relation $\EE_Z$ on $Z$ is an increasing union $\bigcup_{n \in \N} \nearrow \EE_n$, where each $\EE_n$ is a Borel equivalence relation, whose classes have  cardinality bounded by some (finite) $r_n$. 
Let $E_n \coloneq E \cap (\pi \times \pi)^{-1}(\EE_n)$.
It now follows from the scalar-equivariance of $\pi$ (\cref{def: scalar potential}\labelcref{scalar-equivariance}) that $E_n$ is as desired.
\end{proof}

\begin{corollary}
\label{cor: III0 -> hyper II1}
Every ergodic mcp type $\mathrm{III}_0$ equivalence relation $E$ is \dfn{hyper type $\mathrm{II}_1$}, i.e., $E = \bigcup_{n \in \N} \nearrow E_n$, where each $E_n$ is a Borel subequivalence relation (potentially nonergodic), admitting an invariant Borel probability measure $\mu_n \sim \mu$.
\end{corollary}
\begin{proof}
Because $(E,X,\mu)$ is type $\mathrm{III}_0$, it admits a free potential flow $(\Omega,\eta)$.
Letting $Z \subseteq \Omega$ be a Borel cross-section for this flow, \cref{change target cross-section} gives a measure $\tilde \mu \sim \mu$ and a scalar potential $\tilde\pi : (X,\tilde\mu) \to (\Omega,\eta)$ for $(E,X,\tilde\mu)$ satisfying that $\tilde\pi_* \tilde\mu$ in the Ramsay measure class of $Z$.
\cref{increasing_union_with_finite_cocycle_values} now yields $E = \bigcup_{n \in \N} \nearrow E_n$, where for each $n \in \N$, the cardinality of $\w_{\tilde\mu}([x]_{E_n} \times \{x\})$ is finite for a.e.\ $x \in X$.
Thus, \cref{bounded-unbounded_cocycle_types}\labelcref{item:inf>0} yields an $E_n$-invariant finite measure $\~\mu_n \sim \~\mu$, so its normalization is a desired probability measure.
\end{proof}

\begin{lemma}\label{fin_almost_erg_subgraph}
Let $E$ be an ergodic mcp type $\mathrm{III}_0$ equivalence relation on a standard probability space $(X,\mu)$
Let $\pi : (X, \mu) \to (\Omega,\eta)$ be a scalar potential for $(E,X,\mu)$ such that the $\RR$-flow $(\Omega,\eta)$ is aperiodic and $\pi_* \mu$ is in the Ramsay measure class of some Borel cross-section $Z \subseteq \Omega$ of the $\RR$-flow.

Let $\H_0 \subseteq \G$ be a component-finite Borel subgraph such that the image graph $(\pi\times \pi)(\H_0)$ is component-finite. 
Let $\epsilon > 0$ and $f \in L^\infty(X,\mu)$. 
Then there exists an intermediate component-finite Borel subgraph $\H_0 \subseteq \H_1 \subseteq \G$ such that
\begin{enumerate}[(i)]
\item\label{cond:finRel} the image graph $(\pi\times \pi)(\H_1)$ is component-finite;

\item\label{cond:seeAllValues} $\mu(\{x \in X : \PoincRetMap_Z(\pi x) \text{ and } \PoincRetMap_Z^{-1}(\pi x) \text{ are both in } \pi ([x]_{\H_1})\}) > 1-\epsilon$;

\item\label{cond:relErgodic} $\|\E[f \mid \ker_{\H_1} \w_\mu] - \E[f \mid \ker_{\G} \w_\mu]\|_1 \leq \epsilon$.
\end{enumerate}
\end{lemma}

\begin{proof}
Denote $\w \coloneq \w_\mu$.
By replacing $f$ with $f - \E[f \mid \ker_{\G} \w]$, we may assume without loss of generality that  $\E[f \mid \ker_{\G} \w] = 0$.

Because the connectedness relation of $(\pi \times \pi) (\H_0)$ is finite and the cross-section equivalence relation $R_Z = E_{\PoincRetMap_Z}$ is hyperfinite, there is an increasing exhaustion $R_Z = \bigcup_{n \in \N} \nearrow R_n$ with $R_0=E_{(\pi \times \pi) (\H_0)}$ where each $R_n$ is a finite Borel equivalence relation.
Define a sequence of subgraphings $\G_n\subseteq \G$ by letting $\G_n = \G \cap (\pi\times\pi)^{-1}(R_n)$.
These subgraphings satisfy that $\H_0\subseteq \G_n$ for every $n$ and $\G_n\nearrow \G$ and $\ker_{\G_n} \w \nearrow \ker_\G \w$.
Notice that the $\G_n$ are not necessarily component-finite.

By the reverse martingale convergence theorem, we have $\lim_n \E[f \mid \ker_{\G_n} \w] = \E[f \mid \ker_{\G} \w] = 0$ both a.e.\ and in $L^1$-norm.
Then there exists a large enough $k$ such that $\Kc \coloneq \G_k$ satisfies:
\begin{itemize}
\item $\mu(\{x \in X : \PoincRetMap_Z(\pi x) \text{ and } \PoincRetMap_Z^{-1}(\pi x) \text{ are both in } \pi ([x]_{\Kc})\}) > 1-\epsilon$, and

\item $\left \|\E[f\mid \ker_{\Kc} \w]\right\|_1 \leq \frac{\epsilon}{2}$.
\end{itemize}

Thus, the only desired property that $\Kc$ fails to have is component-finiteness.
Note that by construction, the set $\w([x]_{\Kc} \times \{x\})$ is finite for every $x \in X$.

We will define $\H_1$ in two steps. 
First, we will define an intermediate finite graphing $\H_0 \subseteq \H \subseteq \Kc$ such that $\pi([x]_{\H})=\pi([x]_{\Kc})$ for almost every $x\in X$. 
Second, we will apply Tserunyan's ergodic subgraph lemma countably many times to obtain the desired $\H_1$.

We begin with the first step of defining $\H$.
Let $[X]^{< \N}$ denote the standard Borel space of all finite subsets of $X$ and consider the (Borel) collection $\Theta \subseteq [X]^{< \N}$ of all nonempty $E_{\H_0}$-invariant finite subsets $C$ of $X$ which are $\Kc$-connected (i.e., the induced subgraph $\Kc |_C$ is connected) and satisfy $\pi(C)=\pi([C]_{\Kc})$.
By \cite[Lemma 7.3]{Kechris-Miller}, there is a maximal collection $\Theta_0 \subseteq \Theta$ of pairwise disjoint sets with $\Theta_0$ Borel.
Then the union $V_0 = \bigcup_{C \in \Theta_0} C$ is Borel and by maximality it meets every $\Kc$-component.
Now Luzin--Novikov uniformization together with a standard ``lexicographically least path'' argument yields a Borel retraction $r : X \to V_0$ to $V_0$ \textbf{along $\Kc$}, i.e., $r^{-1}(x)$ is $\Kc$-connected for each $x \in V_0$.
Applying mass transport \cref{mass_transport_formula} to the indicator function of the graph of $r$ yields that for a.e.\ $x \in V_0$, the sum $\sum_{y \in r^{-1}(x)} \w(y,x)$ is finite, and hence $r^{-1}(x)$ is finite because the set $\w([x]_{\Kc} \times \{x\})$ of possible values is bounded away from $0$.
We define $\H$ to be the subgraph of $\Kc$ where two $\Kc$-adjacent vertices $x$ and $y$ are $\H$-adjacent iff $r(x)$ and $r(y)$  belong to the same $C \in \Theta_0$.
Thus, the $\H$-components are exactly the sets $r^{-1}(C)$ for $C \in \Theta_0$, and in particular they are finite.

We now move on to step 2. 
For $j \geq 1$, we define $X_j \subseteq X$ to be the set of points $x\in X$ having $j^\text{th}$ least $\w$-weight within their $\Kc$-component, i.e., we inductively define $X_j$ to consist of all $x\in X\smallsetminus \bigcup_{1 \le i < j} X_i$ with $\w(y,x) \geq 1$ for all $y\in [x]_{\Kc}\smallsetminus\bigcup_{1 \le i < j} X_i$.
The $X_j$ form a partition of $X$ into $\ker_{\Kc}\w$-invariant Borel sets, and each $\Kc$-component meets $X_j$ for only finitely many $j \geq 1$.

We now aim to construct an increasing sequence $\F_0 \subseteq \F_1 \subseteq \dots$ of component-finite Borel subgraphs of $\Kc$ satisfying 
\begin{equation}\label{eq:ind_cond_expectation}
\|\E[\mathbf{1}_{X_j}f\mid\ker_{ \F_j}\w]\|_1\leq \frac{\epsilon}{2^{j + 1}}+\|\E[\mathbf{1}_{X_j}f\mid\ker_{\Kc}\w]\|_1    
\end{equation}
for every $j \geq 1$.  
Set $\F_0 \coloneqq \H$.
Let $j\geq 1$ and assume that $\F_{j-1}$ has already been constructed. 
Informally, to obtain $\F_j$ we will now edge-slide both $\F_{j-1}$ and $\Kc$ to $X_j$ along $\F_0 = \H$, apply \cref{lem:Anush1.6}, and then slide the resulting graph back, see \cref{fig:retractions}.

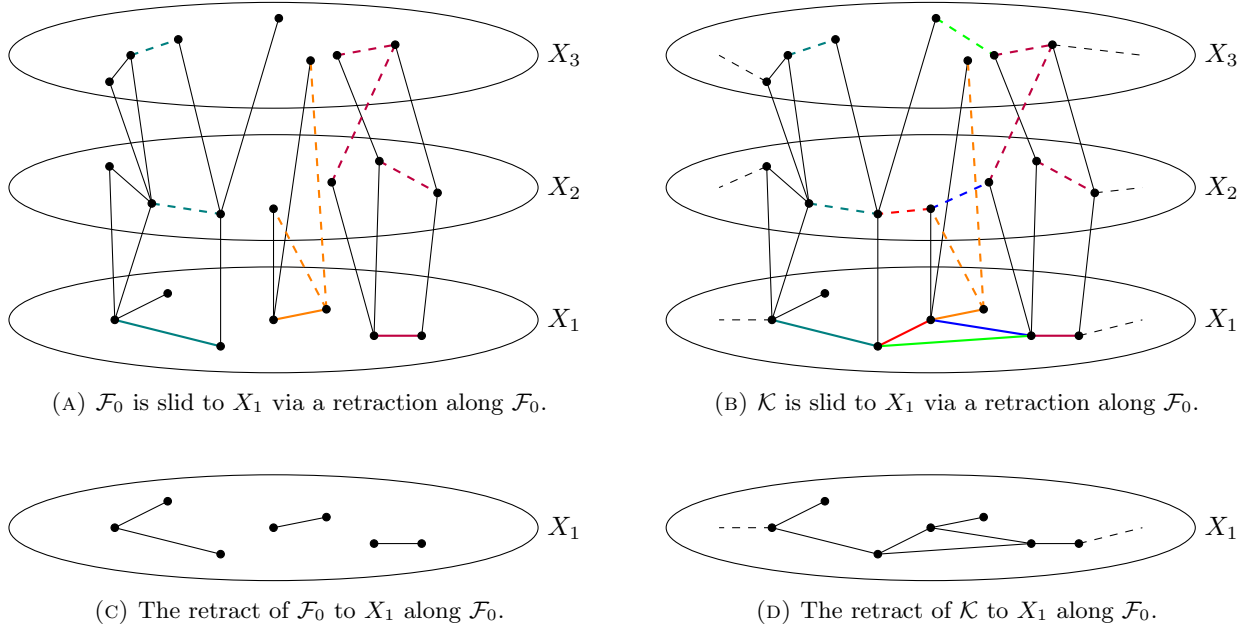
\begin{figure*}[t!]
\subfloat[$\F_0$ is slid to $X_1$ via a retraction along $\F_0$.]{%
\input{Tikz-folders/tikzF-1}%
\label{subfig:a}%
}\hfill
\subfloat[$\Kc$ is slid to $X_1$ via a retraction along $\F_0$.]{%
\input{Tikz-folders/tikzK-1}%
\label{subfig:b}%
}\vspace{2em}\\
\subfloat[The retract of $\F_0$ to $X_1$ along $\F_0$.]{%
\input{Tikz-folders/tikzF-2}%
\label{subfig:c}%
}\hfill
\subfloat[The retract of $\Kc$ to $X_1$ along $\F_0$.]{%
\input{Tikz-folders/tikzK-2}%
\label{subfig:d}%
}
\caption{
The retractions of $\F_0$ and $\Kc$ to $X_1$ along $\F_0$.
The loops in the retracts are not drawn.
}
\label{fig:retractions}
\end{figure*}

Let $[X_j]_{\Kc}$ denote the $E_{\Kc}$-saturation of $X_j$.
Since $\pi ([x]_\H) = \pi ([x]_\Kc)$ for each $x \in X$, it follows that $X_j$ meets each $\H$-component in $[X_j]_\Kc$.
Therefore, we may find a Borel retract $r_j : [X_j]_\Kc \to X_j$ along the graph $\H$.
In particular, $r_j(x) \in [x]_\H \subseteq [x]_{\F_j}$ for every $x \in X$ and $r_j$ is finite-to-one.

Note that $E_\Kc |_{X_j} = (\ker_{\Kc} \w) |_{X_j}$ and this relation is graphed by $(r_j \times r_j)(\Kc)$. 
Furthermore, the subgraph $(r_j \times r_j)(\F_{j-1})$ of $(r_j \times r_j)(\Kc)$ is component-finite since it is contained in $E_{\F_{j-1}} |_{X_j}$.
We can now apply \cref{lem:Anush1.6} to get a component-finite Borel subgraph $\mathcal{L}_j\subseteq (r_j\times r_j)(\Kc)$ containing $(r_j\times r_j)(\F_{j-1})$ with 
\[
\|\E \big[ f|_{X_j} \bigmid E_{\mathcal{L}_j} |_{X_j} \big] - \E\big[f|_{X_j} \bigmid (\ker_{\Kc} \w) |_{X_j}\big] \|_1\leq  \frac{\epsilon}{2^{j + 1}},
\]
which, by $\ker_{\Kc}\w$-invariance of $X_j$, is equivalent to
\[
\|\E [  \mathbf{1}_{X_j}f \mid E_{\mathcal{L}_j} ] - \E [ \mathbf{1}_{X_j}f \mid \ker_{\Kc} \w  ] \|_1\leq  \frac{\epsilon}{2^{j + 1}}.
\]
Define $\F_j\coloneqq (r_j \times r_j)^{-1}(E_{\mathcal{L}_j})\cap \Kc$, and note that $E_{\F_j}=(r_j\times r_j)^{-1}(E_{\mathcal{L}_j})$. 
Observe that $\F_j$ contains $\F_{j-1}$, and $\F_j$ is component-finite since $\F_{j-1}$ is component-finite and $r_j$ is finite-to-one.
Furthermore, $E_{\mathcal{L}_j} |_{X_j} = E_{\F_j} |_{X_j} = (\ker_{ \F_j}\w) |_{X_j}$, so $\F_j$ satisfies \cref{eq:ind_cond_expectation}.

Define $\H_1\coloneqq \bigcup_j  \F_j$. 
Then $\H_1$ is component-finite since each $\Kc$-component meets only finitely many $X_j$'s, hence for each $x\in X$ the sequence $[x]_{\F _0}\subseteq [x]_{\F_1}\subseteq \cdots\subseteq [x]_{\F_j}\subseteq \cdots$ eventually stabilizes.

Since $\ker_{\F_j}\w \subseteq \ker_{\H_1}\w$, we have
\[ 
\|\E[\mathbf{1}_{X_j} f \mid\ker_{\H_1} \w]\|_1 \leq \|\E[\mathbf{1}_{X_j} f \mid\ker_{ \F_{j}} \w]\|_1
\]
for every $j \ge 1$.
Since the $X_j$ are $(\ker_\Kc \w)$-invariant and partition $X$, this gives 
\[
\|\E[f\mid \ker_{\H_1} \w]\|_1 \leq \sum_{j \geq 1}\|\E[\mathbf{1}_{X_j} f \mid \ker_{ \F_{j}} \w]\|_1 \leq  \sum_{j \geq 1}  \frac{\epsilon}{2^{j + 1}}  +\underbrace{   \sum_{j \geq 1}\|\E[\mathbf{1}_{X_j} f \mid \ker_{\Kc} \w]\|_1}_{\leq \epsilon/2}  \leq \epsilon ,
\]
showing that $\H_1$ satisfies item \ref{cond:relErgodic}. Items \ref{cond:finRel} and \ref{cond:seeAllValues} follow since $\H\subseteq \H_1\subseteq \Kc$.
\end{proof}

\section{Balanced hyperfinite graphings with minimal cocycle values}

Let $E$ be a countable Borel equivalence relation on $X$ and let $A \subseteq B \subseteq X$ be a pair of nested sets. 
An \textbf{$E$-compression of $B$ starting from $A$} is a map $c \in [[E]]$ with domain $B$ and image contained in $B\smallsetminus A$ such that the sets $A,c(A),c^2(A),\dots$ form a partition of $B$.

\begin{lemma}\label{compressions}
Let $(E, X, \mu)$ be an ergodic mcp type $\mathrm{III}_0$ countable Borel equivalence relation with dis-closed balanced Radon--Nikodym cocycle $\w_\mu$. 
Suppose further that there is an $\alpha \in (0,1)$ such that $\w ([x]_E\times \{x\})=\alpha^\mathbf{Z}$ for almost every $x\in X$. 
Let $\pi: X\to Z$ be an ergodic decomposition map for $(\ker_E \w_\mu, X,\mu)$. 
Let $A \subset B \subset X$ be a pair of non-null measurable sets, and assume that $\mu (A)=\mu _z(A)=(1-\alpha )\mu_z(B)$ for a.e.\ $z \in Z$.

Then, after discarding a null set, there exists an $E$-compression $c$ of $B$ starting from $A$ satisfying $\w_\mu(c(x),x) = \alpha$ for a.e.\ $x \in B$, and $\mu_z(c^i(A)) = \mu (c^i(A))=\alpha^i \mu(A)$ for each $i \in \N$ and a.e.\ $z \in Z$.
\end{lemma}
\begin{proof}
Set $A_0 \coloneqq A$.
By \cref{prop:kernelIsNowhereSmooth}, the $\mu_z$ is atomless for a.e.\ $z \in Z$.
\cref{sets_with_prescribed_fiber-measures} yields a sequence $(A_i)_{i \ge 1}$ of pairwise disjoint measurable subsets of $B$ satisfying $\mu_z(A_i) = \alpha^i \mu_z(A)$ for a.e.\ $z \in Z$.
Note that $\mu_z(B) = \frac{\mu_z(A)}{1 - \alpha} = \sum_{i \geq 0} \mu_z(A_i)$, so after discarding an $E$-invariant Borel null set, we may assume that the $A_i$ partition $B$. 
We may also assume that each $A_i$ is a complete section for $\ker_E \w_\mu$, and that $\w_\mu^x([x]_E)=\alpha^\mathbf{Z}$ for every $x \in X$.

It then follows that for each $x \in A_i$ there is a $y \in [x]_E$ with $\w_\mu(y,x) = \alpha$ and moreover, we may take $y$ to be in $A_{i+1}$ because the latter is a complete section for $\ker_E \w_\mu$.
Hence, by Luzin--Novikov uniformization, there exist Borel maps $\tilde{c}_i : A_i \to A_{i+1}$ such that $\tilde{c}_i(x) \mathrel{E} x$ and $\w_\mu(\tilde{c}_i(x), x) = \alpha$.

By \cref{pi is ergodic decomposition} and essential uniqueness of ergodic decomposition, we may assume that $\pi$ is a Krieger scalar potential for $(E,X,\mu )$ and that $Z$ is the cross-section $Z_{E,\mu}$ associated to $\pi$ provided by \cref{canonical_potential_section}.

Since $\w_\mu$ is balanced, the measure $\pi_*\mu$ is invariant for the cross-section equivalence relation on $Z$. 
Hence, applying mass transport across Borel homomorphisms (\cref{lem:mass_transport_homomorphism}) with $f_i(x,y) \coloneqq 1$ if $y = \tilde{c}_i(x)$ and $f_i(x,y)\coloneqq 0$ otherwise, we get that for a.e.\ $z \in Z$:
\begin{align*}
\mu_z(A_i) = \int_{X} \sum_{y\in [x]_E}f_i(x,y)\,\deriv \mu_z(x) 
&=\sum_{z'\in \langle \PoincRetMap_Z\rangle (z)} \underbrace{\w_{\pi_*\mu}(z',z)}_{=1} \int_{X}\sum_{\substack{y\in [x]_E\\
\pi y = z}}f_i(y,x)\w_\mu(y,x) \, \deriv \mu_{z'}(x)\\
&= \int_{A_{i+1}}\sum_{\substack{y\in \pi^{-1}(z) \\
\tilde c_i(y)=x}}f_i(y,x)\underbrace{\w_\mu(y,\tilde c_i(y))}_{=\alpha^{-1}} \, \deriv \mu_{\alpha.z}(x)
\\
&= \alpha  ^{-1}\int_{A_{i+1}} |\tilde{c}_i^{-1}(x)| \, \deriv \mu_{\alpha.z},
\end{align*}
where the sum in the integrand in the righthand side of the first line is nonzero only when $x\in A_{i+1}$ and $x=\tilde{c}_i(y)$, and hence the only $z'$ term in the sum that contributes is $z'=\pi(\tilde c_i(y))=\alpha.\pi(y)=\alpha .z$.

Thus by assumption, we have $\mu_{\alpha.z}(A_{i+1})=\alpha\mu_{\alpha.z}(A_i)=\alpha\mu_z(A_i)=\mu_{\alpha.z}(x\mapsto |\tilde{c}_i^{-1}(x)|) $ for a.e.\ $z \in Z$.

For each $i \in \N$, applying \cref{typeIISlide} with 
$R \coloneqq \ker_E w_\mu$, 
$g \coloneqq \mathbf{1}_{A_{i+1}}$, 
$Y \coloneqq A_i$,
and $\phi \coloneqq \tilde{c}_i$
gives a Borel map $c_i : A_i \to A_{i+1}$ with $|c_i^{-1}(x)| = 1$ for all $x \in A^{i+1}$ and $c_i(y) \mathrel{\ker_E \w_\mu} \tilde{c}_i(y)$ for $y \in A_i$. 
The union of these maps $c \coloneqq \bigcup c_i$ satisfies the conclusion of the lemma.
\end{proof}

\begin{remark}\label{rem:typeIIILambdaCompressions}
\cref{compressions} also holds when $(E,X,\mu)$ is instead assumed to be type $\mathrm{III}_\lambda$ and $\w_\mu ([x]_E\times \{ x \} ) = \lambda^\mathbf{Z}$ for some $\lambda \in (0,1)$. 
In this case, $\lambda$ is an essential value of the Radon--Nikodym cocycle, which implies via measure exhaustion that for every pair of positive measure sets $C, D\subseteq X$ satisfying $\mu(D) = \lambda \,\mu(C)$, there exists a bijection $\phi \in [[E]]$ with domain $C$ and image $D$ with $\w_\mu(\phi x,x)=\lambda$ for almost every $x\in C$. 
Applying this with $C \coloneqq A_i$ and $D \coloneqq A_{i+1}$ as in the proof of \cref{compressions}, yields the conclusion of \cref{compressions}.
\end{remark}
\begin{proposition}\label{prop:graphWithMinCocycleValues}
Let $(R, X,\mu)$ be an ergodic mcp type $\mathrm{III}_0$ hyperfinite Borel equivalence relation with dis-closed balanced Radon--Nikodym cocycle $\w_\mu$. 
Let $\alpha \in (0,1)$ and suppose that $\w_\mu ([x]_R\times \{x\}) = \alpha^\mathbf{Z}$ for a.e.\ $x \in X$. 
Then there exists $S \in [R]$ such that $E_S  = R$ and $\w_\mu(S x, x) = \alpha^{\pm 1}$ for a.e.\ $x \in X$.
\end{proposition}

\begin{proof}
Let $\pi : X \to \K_R$ be a Krieger scalar potential for $(R,X,\mu)$, and let $Z=Z_{R,\mu}$ be the cross-section associated to $\pi$ provided by \cref{canonical_potential_section}.
Then $\pi$ is an ergodic decomposition map for $\ker_R\w_\mu$ by \cref{pi is ergodic decomposition}.
By \cref{prop:kernelIsNowhereSmooth}, $\ker_E w_\mu$ is $\mu$-nowhere smooth.
Since $\pi$ is $\ker_E w_{\mu}$-invariant, the measure disintegration $z \mapsto \mu_z$ of $\mu$ through $\pi$ satisfies that $\mu_z$ is atomless for $\pi_* \mu$-a.e.\ $z \in Z$.
By \cref{sets_with_prescribed_fiber-measures}, we may therefore find a partition $A\sqcup B_-\sqcup B_+ = X$ of $X$ into Borel sets such that $\mu_z(A) = \frac{1 - \alpha}{ 1 + \alpha}$ and $\mu_z(B_-) = \mu_z(B_+) = \frac{\alpha}{1+\alpha}$ for a.e.\ $z \in Z$.
For each $\epsilon \in \{-,+\}$, we apply \cref{compressions} and discard an $R$-invariant null set to get an $R$-compression $c_\epsilon$ of $A \sqcup B_\epsilon$ starting from $A$ with $\w_\mu(c_\epsilon(x), x) = \alpha$ for all $x \in A \sqcup B_\epsilon$, and such that for each $\epsilon \in \{-,+\}$, the iterates $c_\epsilon^i(A)$ with $i \geq 1$ partition $B_\epsilon$. For a large enough $k \in \N$, we have
\[
\mu_z(\sqcup_{i\geq k} c_{\epsilon}^i(A)) = \mu_z(A) \frac{\alpha^k}{1-\alpha} \le \frac{1}{2} \mu_z(A) 
\]
for all $z \in Z$.
We fix such a $k \ge 1$.

Since $R|_A$ is hyperfinite we may find a Borel bijection $S_0 : A \bij A$ in $[R|_A]$ generating $R|_A$ \cite[Theorem~6.6]{Kechris-Miller}. 
Let $k_+ : A \to \N \cap [k, \infty)$ be a measurable function such that $\w_\mu(c_+^{k_+(x)}(x), S_0(x)) < \alpha^k$ for every $x \in A$. 
Then there is a unique (measurable) function $k_- : A \to \N$ such that $\w_\mu(c_+^{k_+(x)} (x), c_-^{k_-(S_0(x))}(S_0(x))) = \alpha$ for every $x \in A$, see the left side of \cref{fig:detour}.
Observe that $k_-(y) \geq k$ for all $y \in A$.

We now modify $S_0$ in order to extend its domain to the following set:
\[
C \coloneqq\{c_\epsilon^{i}(x) \mid x \in A \text{ and } 0 \leq i \leq k_\epsilon(x) \text{ and } \epsilon \in \set{-,+}\}.
\]
We define a Borel bijection $S_1 : C \bij C$ by replacing each edge $(x,y) \coloneqq (x,S_0 x)$ of $S_0$ with the path $x, c_+(x), c_+^2(x), \dots, c_+^{k_+(x)}(x)$, then a jump from $c_+^{k_+(x)}(x)$ to $c_-^{k_-(y)}(y)$, then the path $c_-^{k_-(y)}(y), \dots, c_-^2(y), c_-(y), y$, as in \cref{fig:detour}.
More precisely:
\[   
S_1(u) \coloneqq 
\begin{cases}
c_+(u) &\quad\text{if } u = c_+^i(x)\text{ for some } x \in A \text{ and } 0 \leq i < k_+(x),
\\
c_-^{-1}(u) &\quad\text{if } u = c_-^i(y) \text{ for some } y \in A \text{ and } 0 < i \leq k_-(y),
\\
c_-^{k_-(y)}(y) &\quad\text{if } u = c_+^{k_+(x)}(x) \text{ for some } x \in A \text{ and } y = S_0(x).
\end{cases}
\]
The map $S_1$ now satisfies $\w_\mu( S_1(x), x) = \alpha^{\pm 1}$ for every $x \in C$.

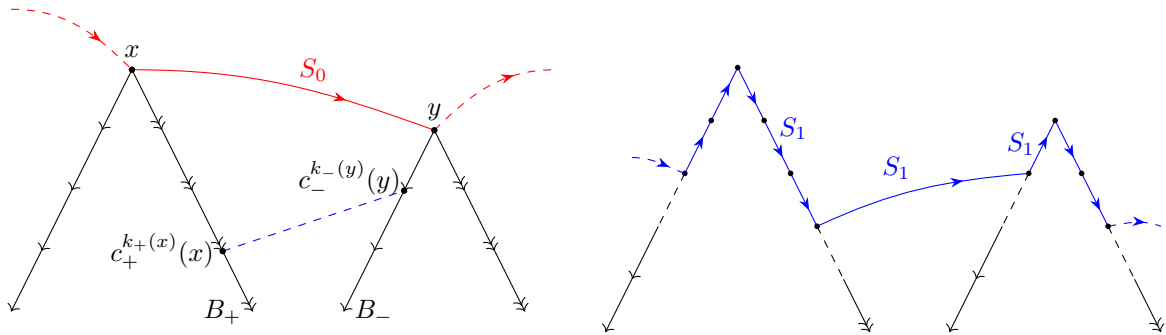
\begin{figure}[ht]
\begin{subfigure}[b]{.45\textwidth}
\centering
\begin{tikzpicture}[scale = 0.8]
\foreach \i in {3, 2, 1, 0}{
\draw[->>] ( -2 + 0.5 * \i, 4 - \i ) -- (-1.5 + 0.5 * \i, 3 - \i);

\draw[->] ( -2 - 0.5 * \i, 4 - \i ) -- (-2.5 - 0.5 * \i, 3 - \i);
}

\foreach \i in {2, 1, 0}{
\draw[->] (3 - 0.5 * \i, 3 - \i ) -- (2.5 - 0.5 * \i, 2 - \i);
\draw[->>] (3 + 0.5 * \i, 3 - \i ) -- (3.5 + 0.5 * \i, 2 - \i);
}

\draw[dashed, blue] (-0.5,1) -- (2.5,2);
\draw[red, ->-] (-2, 4) to[out=0, in=160] (3,3);

\draw[red, ->-, dashed] (-4, 5) to [out=0, in=135] (-2, 4);
\draw[red, ->-, dashed] (3, 3) to [out=45, in=180] (5, 4);

\node at (-2, 4.3) {$x$};
\node at (3, 3.3) {$y$};
\node[red] at (1,4) {$S_0$};
\node at (-.5, 0) {$B_+$};
\node at (2, 0) {$B_-$};
\node at (-1.5, 1) {$c_+^{k_+(x)}(x)$};
\node at (1.6, 2.2) {$c_-^{k_-(y)}(y)$};

\fill (-2,4) circle (1.5pt);
\fill (3,3) circle (1.5pt);
\fill (-0.5,1) circle (1.5pt);
\fill (2.5,2) circle (1.5pt);
\end{tikzpicture}
\end{subfigure}
\begin{subfigure}[b]{.5\textwidth}
\centering    
\begin{tikzpicture}[scale = 0.7]
\foreach \i in {4, 3, 2, 1, 0}{
\ifthenelse{\i < 3}{
\draw[->-, blue] ( -3 + 0.5 * \i, 4 - \i) -- (-2.5 + 0.5 * \i, 3 - \i);
\fill (-2.5 + 0.5 * \i, 3 - \i) circle (1.5pt);
}{}
\ifthenelse{\i = 3}{
\draw[dashed] ( -3 + 0.5 * \i, 4 - \i ) -- (-2.5 + 0.5 * \i, 3 - \i);  
}{}
\ifthenelse{\i > 3}{
\draw[->>] ( -3 + 0.5 * \i, 4 - \i ) -- (-2.5 + 0.5 * \i, 3 - \i);}{}

\ifthenelse{\i < 2}{
\draw[->-, blue]  (-3.5 - 0.5 * \i, 3 - \i) -- ( -3 - 0.5 * \i, 4 - \i) ;
\fill ( -3 - 0.5 * \i, 4 - \i) circle (1.5pt);
}{}
\ifthenelse{\i = 2}{
\draw[dashed] ( -3 - 0.5 * \i, 4 - \i ) -- (-3.5 - 0.5 * \i, 3 - \i);
\fill ( -3 - 0.5 * \i, 4 - \i ) circle (1.5pt);
\draw[->-, blue, dashed] (-4 - 0.5 * \i, 4.3 - \i) to[out=0, in=160]  ( -3 - 0.5 * \i, 4 - \i );
}{}
\ifthenelse{\i > 2}{
\draw[->] ( -3 - 0.5 * \i, 4 - \i ) -- (-3.5 - 0.5 * \i, 3 - \i);}{}
}


\foreach \i in {3, 2, 1, 0}{
\ifthenelse{\i < 1}{\draw[->-, blue] (2.5 - 0.5 * \i, 2 - \i) -- (3 - 0.5 * \i, 3 - \i );
\fill (3 - 0.5 * \i, 3 - \i ) circle (1.5pt);
}{}

\ifthenelse{\i = 1}{\draw[dashed] (3 - 0.5 * \i, 3 - \i ) -- (2.5 - 0.5 * \i, 2 - \i);
}{
}
\ifthenelse{\i > 1}{\draw[->] (3 - 0.5 * \i, 3 - \i ) -- (2.5 - 0.5 * \i, 2 - \i);}{}

\ifthenelse{\i < 2}{\draw[->-, blue] (3 + 0.5 * \i, 3 - \i ) -- (3.5 + 0.5 * \i, 2 - \i);
\fill (3.5 + 0.5 * \i, 2 - \i) circle (1.5pt);
}{}

\ifthenelse{\i = 2}{\draw[dashed] (3 + 0.5 * \i, 3 - \i ) -- (3.5 + 0.5 * \i, 2 - \i);
\draw[->-, blue, dashed] (3 + 0.5 * \i, 3 - \i ) to[out=10, in=160]  (4 + 0.5 * \i, 3 - \i);}{
}
\ifthenelse{\i > 2}{\draw[->>] (3 + 0.5 * \i, 3 - \i ) -- (3.5 + 0.5 * \i, 2 - \i);}{}
}

\draw[blue, ->-] (-1.5,1) to[out=25, in=185]  (2.5,2);
\fill (2.5,2) circle (1.5pt);
\node[blue] at (0,2.1) {$S_1$};
\node[blue] at (-1.95,2.85) {$S_1$}; 
\node[blue] at (2.35,2.55) {$S_1$};

\end{tikzpicture}
\end{subfigure}
\caption{
Detouring an edge of $S_0$ along compressions $c_+$ and $c_-$.}
\label{fig:detour}
\end{figure}

Let $D_0$ be the set of the points in $X\smallsetminus C$ that belong to $c_+ (C) \cup c_- (C)$, the "first ones" along the downwards compressions that do not belong to $C = \dom(S_1)$, i.e.,
\[
D_0 =  \left\{v \in X: \exists u \in A, v = c_+^{k_+(u) + 1}(u) \text{ or } v = c_-^{k_-(u) + 1}(u)\right\}.
\]
Then $c \coloneqq c_+|_{X \smallsetminus  C} \sqcup c_-|_{X \smallsetminus  C}$ is a compression of $X \smallsetminus C$ starting from $D_0$.
Let 
\[
D' \coloneqq \bigsqcup_{n \geq 0} c^{2n} D_0 
\text{ and } 
D'' \coloneqq c D'= \bigsqcup_{n \geq 0} c^{2n+1} D_0.
\]

Notice that $D'\subset (\sqcup_{i\geq k} c_{1}^i(A)) \sqcup (\sqcup_{j\geq k} c_{0}^j(A))$, so $\mu_z(D') \le \mu_z(A) \le \frac{1}{2} \mu_z(C)$ for all $z\in Z$ due to our choice of $k$.
Thus, standard measure exhaustion yields a bijection $t \in [[\mathrel{\ker_R \w_\mu}]]$ with domain $D'$ and image $C' \subseteq C$ between $D'$ and a subset $C'$ of $C$. 
Let $r : C' \bij D'$ be the inverse of $t$, so it also belongs to $[[\mathrel{\ker_R \w_\mu}]]$.
We now extend the domain of $S_1$ to the entire $X$ by replacing every edge $(x, S_1(x))$ where $x \in C'$, with the path $x, c(r(x)), r(x), S_1(x)$, see \cref{fig:zigzag}.

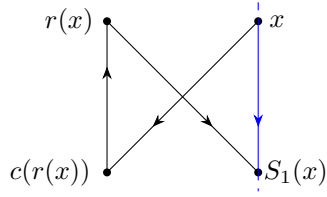
\begin{figure}[ht]
\centering
\begin{tikzpicture}
\node at (-0.75,0) {$c(r(x))$};
\node at (2.5, 0) {$S_1(x)$};
\node at (-0.5, 2) {$r(x)$};
\node at (2.25, 2) {$x$};

\fill (0,0) circle (1.5pt);
\fill (0,2) circle (1.5pt);
\fill (2,0) circle (1.5pt);
\fill (2,2) circle (1.5pt);

\draw[->-, blue] (2,2) -- (2,0);
\draw[dashed, blue] (2,2.25) -- (2,2);
\draw[dashed, blue] (2,0) -- (2,-0.25);

\draw[->-] (2,2) -- (0,0);
\draw[->-] (0,0) -- (0,2 );
\draw[->-] (0,2) -- (2,0);

\end{tikzpicture}
\caption{Replacing an edge of $S_1$ with a path.}
\label{fig:zigzag}
\end{figure}

More precisely, we define $S  : X \to X$ by
\[   
S (x) = 
\begin{cases}
c(r(x)) &\quad\text{if } x \in C', \text{ this sends  bijectively } C' \text{ to } D''
\\
c^{-1}(x) &\quad\text{if } x \in D'' , \text{ this sends  bijectively } D'' \text{ to } D'
\\
S_1(t(x)) &\quad\text{if } x\in D', \text{ this sends  bijectively } D' \text{ to } S_1(C')
\\
S_1(x) &\quad\text{if } x \in C\smallsetminus C', \text{ this sends bijectively } C \smallsetminus C' \text{ to } S_1(C\smallsetminus C').
\end{cases}
\]
$S $ is surjective because $X = D'' \sqcup D' \sqcup S_1(C') \cup S_1(C \smallsetminus C')$, and it is injective because in each case of its definition a bijection is applied whose image is disjoint from the other cases.
Also, in all cases we have $\w_\mu(S (x), x) = \alpha^{\pm 1}$ as required, since $S (x)$ is defined as a map that satisfies this cocycle property, possibly composed with another map along which the Radon--Nikodym cocycle is trivial.
For instance, if $x \in D'$, then $\w_\mu(S (x), x) = \w_\mu(S_1(t(x)), x) = \w_\mu(S_1(t(x)), t(x)) \w_\mu(t(x), x) = \alpha^{\pm 1}$ since $\w_\mu(t(x), x)=1$.

Lastly, because $S_0$ generates $R|_A$, it follows from the construction that $S $ generates $R$ since the graph of $S $ is contained in $R$ and each $S_0$-orbit is contained in a $S $-orbit.
\end{proof}

\begin{remark}
In the context of \cref{rem:typeIIILambdaCompressions}, \cref{prop:graphWithMinCocycleValues} holds with the same proof if the countable Borel equivalence relation $(R,X,\mu)$ is instead assumed to be type $\mathrm{III}_\alpha$ and $\w_\mu ([x]_R\times\{x\})=\alpha^\mathbf{Z}$ for some $\alpha\in(0,1)$.
\end{remark}

\section{\texorpdfstring{Edge sliding into a $\mu$-homogeneous $T_{p,q}$-treeing}{Edge sliding into a mu-homogeneous {Tp,q}-treeing}}

In the following, given a Borel automorphism $S$ of the standard Borel space $X$, we denote by $\mathcal{G}_S$ the Borel graph consisting of the edges $(x,Sx)$ and $(Sx,x)$ for every $x\in X$.

\begin{theorem}\label{prop:edgeSlideTpqTreeing}
Let $2\leq p< q$ and let $\G$ be an ergodic mcp type $\mathrm{III}_0$ nonamenable Borel graph on $(X,\mu)$ that is acyclic, and let $E\coloneqq E_{\G}$. 
Assume that the Radon--Nikodym cocycle of $E$ is dis-closed and balanced, with $\w_\mu ([x]_E\times \{ x \} ) = ( \frac{p}{q} )^{\mathbf{Z}}$ for almost every $x \in X$. 
Let $S \in [E]$ be such that
\begin{enumerate}[(i)]
\item $\mathcal{G}_S$ is a subgraph of $\G$;
\item \label{it: rel. ergodicity for ker S w} $\ker_{E_S}\w_\mu\leq \ker_E \w_\mu$ is relatively ergodic;
\item \label{it:one-step-p/q} $\w_\mu(Sx, x) = (\frac{p}{q})^{\pm 1}$ for almost every $x\in X$;
\item \label{it: same values for w} $\w_\mu ([x]_S\times \{ x \} )=( \frac{p}{q} )^{\mathbf{Z}}$ for almost every $x\in X$.
\end{enumerate}
Then there exists a $\mu$-homogeneous $T_{p,q}$-treeing of $E$.
\end{theorem}
\begin{proof}
Throughout the proof, in order to find a $\mu$-homogeneous $T_{p,q}$-treeing we will slide edges of $\G\smallsetminus \G_S$, using edges of the graph $\G_S$ as rails. This will ensure that the resulting graph is acyclic, and we will make sure that its connected components coincide with those of $\G$. 

Let $\pi : X \to \K_{E}$ be a Krieger scalar potential for $(E,X,\mu)$, and let $Z=Z_{E,\mu}$ be the cross-section associated to $\pi$ provided by \cref{canonical_potential_section}. 
By \cref{prop:subrelation_scalar_potentials}\labelcref{item:same_scalar_potentials} and assumptions \labelcref{it: same values for w,it: rel. ergodicity for ker S w}, $E_S$ is ergodic and $\pi$ is also a Krieger scalar potential for $(E_S, X, \mu)$. 
In particular, $E_S$ is type $\mathrm{III}_0$.
By \labelcref{it:one-step-p/q} we may direct the edges of $\G_S$ to get a directed Borel graph $\vec{\G}_S$ satisfying $\w_\mu(\targ (e),\orig (e)) = \tfrac{p}{q}$ for every positive edge $e$ of $\vec{\G}_S$.
Direct the edges of $\G \smallsetminus \G_S$ in an arbitrary Borel fashion to extend $\vec{\G}_S$ to a directing of the entire graph we call $\vec{\G}$. 

Notice that $\deg^{\mathrm{in}}_{\vec{\G_S}}(x) \leq 2 \leq p$ and that  $\deg^{\mathrm{out}}_{\vec{\G_S}}(x) \leq \deg^{\mathrm{out}}_{\vec{\G}}(x)$, with both inequalities being strict somewhere; since the non-oriented degree of $\G_S$ is $2$ and since
$E_S$ is amenable contrarily to $\G$, respectively.
We may therefore apply \cite[Theorem 3.1]{Poulin} with $d_i \equiv p$ and $d_o = \deg^{\mathrm{out}}_{\vec{\G}}$ to the inclusion $\vec{\G}_S \subseteq \vec{\G}$. This yields a directed forest $\vec \G_0$ containing $\vec{\G}_S$ with $E_{\vec\G_0}=E_{\vec\G}=E$ and in-degree $\deg^{\mathrm{in}}_{\vec{\G}_0}(x)=p$ a.e.

Our next task will be to slide the tail of each edge in $\vec{\G}_0\smallsetminus\vec{\G}_S$ so that each positive edge $e$ in the resulting graphing $\vec{\G}_1$ satisfies $\w_\mu(\targ(e),\orig(e))=\frac{p}{q}$.
Assumption \labelcref{it: same values for w} allows us to choose a Borel map $\varphi$ assigning to each positive edge $e$ in $\vec{\G}_0\smallsetminus\vec{\G}_S$ a point 
$\varphi (e)$ in the $E_S$-class of $\orig(e)$
satisfying $\w_\mu(\targ(e),\varphi(e))=\frac{p}{q}$.
This map $\varphi$ gives destinations for the edge slide along $\G_S$: for each positive edge $e$ of $\vec{\G}_0\smallsetminus\vec{\G}_S$, we detach its tail from $\orig(e)$, and slide it along $\G_S$, reattaching the tail at the vertex $\varphi (e)$.
Formally, we obtain the directed Borel graph $\vec{\G}_1$ by adding to $\vec{\G}_S$ a positive edge from $\varphi (e)$ to $\targ(e)$ for each positive $e\in \vec{\G}_0\smallsetminus\vec{\G}_S$, see \cref{fig:edgeSlideWeights}.
The graph $\vec{\G}_1$ still has constant in-degree $p$, and each of its positive edges $e$ satisfies 
\begin{equation}
\label{eq: toto}
\w_\mu(\targ(e),\orig(e))=\frac{p}{q}.
\end{equation}

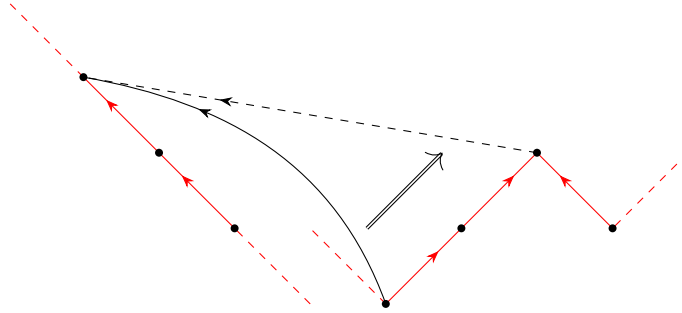
\begin{figure}[ht]
\centering
\begin{tikzpicture}[scale = 1]
\foreach \x/\y [remember=\y as \lasty] in {1/4, 2/3, 3/2, 4/1, 5/0}{
\ifthenelse{\x > 1}{
\ifthenelse{\x = 2 \OR \x = 5}{
\draw[red,dashed] (\x,\y) -- (\x - 1,\lasty);
}{
\ifthenelse{\y > \lasty}{
\draw[red, ->-] (\x - 1,\lasty) -- (\x,\y) ;
}{
\draw[red, ->-] (\x,\y) -- (\x - 1,\lasty);
}
}
}{}
}
\foreach \x/\y in {2/3, 3/2, 4/1}{
\fill (\x, \y) circle (1.5pt);
}
\foreach \x/\y [remember=\y as \lasty] in {5/1, 6/0, 7/1, 8/2, 9/1, 10/2}{
\ifthenelse{\x > 5}{
\ifthenelse{\x =6 \OR\x = 10}{
\draw[red,dashed] (\x,\y) -- (\x - 1,\lasty);
}{
\ifthenelse{\y > \lasty}{
\draw[red, ->-] (\x - 1,\lasty) -- (\x,\y) ;
}{
\draw[red, ->-] (\x,\y) -- (\x - 1,\lasty);
}
}
}{}
}
\foreach \x/\y in {6/0, 7/1, 8/2, 9/1}{
\fill (\x, \y) circle (1.5pt);
}

\draw[->-] (6,0) to[out=110,in=-10] (2,3);
\draw[->-, dashed] (8,2) -- (2,3);
\draw[->, double] (5.75,1) -- (6.75,2);
\end{tikzpicture}
\caption{Tails of directed edges of $\vec{\G}_0 \smallsetminus \vec{\G}_S$ (in black) are slid along the subgraph $\vec{\G}_S$ (in red) to witness the correct cocycle value on edges of the resulting graph $\vec{\G}_1$.}
\label{fig:edgeSlideWeights}
\end{figure}

Using \cref{lem:mass_transport_homomorphism} (mass transport across $\pi$) with $f$ the indicator function of $(\targ \times \orig)( \vec{\G}_1)$ and $F$ the cross-section equivalence relation on $Z$, we have for almost every $z \in Z$,
\begin{align}
\nonumber p = \int_X \sum_{y \in [x]_E} f(x,y) \; \deriv \mu_z(x) 
&= \sum_{z' \in [z]_F} \w_{\pi_*\mu}(z',z) \int_X \sum_{\substack{y \in [x]_E  \\ \pi(y) = z}} f(y,x) \w_\mu(y,x) \; \deriv\mu_{z'}(x)
\\
\label{eq:1}
&= \sum_{z' \in [z]_F} \int_X \sum_{\substack{y \in [x]_E  \\ \pi(y) = z}} f(y,x) \cdot \frac{p}{q} \; \deriv\mu_{z'}(x)
\\
\label{eq:2}
&= \frac{p}{q} \int_X \sum_{\substack{y \in [x]_E  \\ \pi(y) = z}} f(y,x) \; \deriv\mu_{\frac{q}{p} . z}(x)
\\
\label{eq:3}
&= \frac{p}{q} \int_X \deg^{\mathrm{out}}_{\vec{\G}_1}(x) \; \deriv\mu_{\frac{q}{p} . z}(x).
\end{align}
Here, \labelcref{eq:1} uses the fact that $f(y,x) \ne 0$ implies $\w_\mu(y,x) = \frac{p}{q}$ by \labelcref{eq: toto}, and
that $\w_\mu$ is balanced, i.e., $\w_{\pi_* \mu} \equiv 1$;
the next equality \labelcref{eq:2} uses the same fact and the equivariance of $\pi$ via 
\[
z' = \pi(x) = \w_\mu(x,y). \pi(y) = \w_\mu(x,y). z = \frac{q}{p}. z;
\]
finally, \labelcref{eq:3} follows from the definition of $f$.

Since $z \mapsto \frac{q}{p} . z$ is mcp, it follows that the $\mu_z$-expected out-degree of $\vec{\G}_1$ is $q$  for almost every $z \in Z$, i.e., 
\[
\int_X \deg^{\mathrm{out}}_{\vec{\G}_1}(x) \; \deriv\mu_z(x) = q.
\]
By \cref{pi is ergodic decomposition}, $\pi$ is an ergodic decomposition for $\ker_E\w_\mu$, and by assumption $\ker_{E_S}\w_\mu\leq \ker_E\w_\mu$ is relatively ergodic, hence the conditional expectations satisfy, for a.e.\ $x \in X$,
\[
\E[\deg^{\mathrm{out}}_{\vec{\G}_1}\mid\ker_{E_S}\w_\mu ](x)
=
\E[\deg^{\mathrm{out}}_{\vec{\G}_1} \mid \ker_E\w_\mu](x) 
=
\int_X \deg^{\mathrm{out}}_{\vec{\G}_1}(y) \; \deriv\mu_{\pi(x)}(y) 
=
q.
\]

In order to produce a final directed treeing $\vec{\G}_2$ of $E_\G$ with constant out-degree $q$ (and maintaining constant in-degree $p$ as well as the desired cocycle values), we will now slide the tails of edges of $\vec{\G}_1$ along $\G_S$, starting and ending in the same $\ker_{E_S}\w_\mu$-class; 
since $\ker_{E_S}\w_\mu$ is pmp, this does not change the values of $\w_\mu$ on edges.

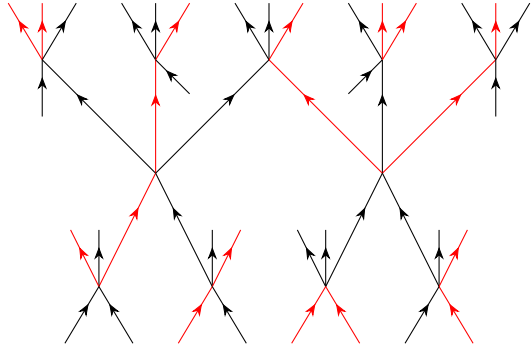
\begin{figure}[ht]
\centering
\begin{tikzpicture}[scale = 0.15, rotate = 90]
\draw[->-] (0,10) -- (10,20);
\draw[->-, red] (0,10) -- (10,10);
\draw[->-] (0,10) -- (10,0);
\draw[->-, red] (-10,15) -- (0,10);
\draw[->-] (-10, 5) -- (0,10);

\draw[->-, red] (0,-10) -- (10,-20);
\draw[->-] (0,-10) -- (10,-10);
\draw[->-, red] (0,-10) -- (10,0);
\draw[->-] (-10,-15) -- (0,-10);
\draw[->-] (-10, -5) -- (0,-10);

\draw[->-] (5, 20) -- (10, 20);
\draw[->-, red] (10, 20) -- (15,23);
\draw[->-, red] (10, 20) -- (15,20);
\draw[->-] (10, 20) -- (15,17);

\draw[->-] (7, 7) -- (10, 10);
\draw[->-] (10, 10) -- (15,13);
\draw[->-] (10, 10) -- (15,10);
\draw[->-, red] (10, 10) -- (15,7);

\draw[->-] (10, 0) -- (15,3);
\draw[->-] (10, 0) -- (15,0);
\draw[->-, red] (10, 0) -- (15,-3);

\draw[->-] (7, -7) -- (10, -10);
\draw[->-] (10, -10) -- (15,-7);
\draw[->-, red] (10, -10) -- (15,-10);
\draw[->-, red] (10, -10) -- (15,-13);

\draw[->-] (5, -20) -- (10, -20);
\draw[->-] (10, -20) -- (15,-23);
\draw[->-, red] (10, -20) -- (15,-20);
\draw[->-] (10, -20) -- (15,-17);

\draw[->-] (-10, 15) -- (-5,15);
\draw[->-, red] (-10, 15) -- (-5,17.5);
\draw[->-] (-15, 18) -- (-10, 15);
\draw[->-] (-15,12) -- (-10, 15);

\draw[->-] (-10, 5) -- (-5,5);
\draw[->-,red] (-10, 5) -- (-5,2.5);
\draw[->-,red] (-15, 8) -- (-10, 5);
\draw[->-] (-15,2) -- (-10, 5);

\draw[->-] (-10, -5) -- (-5,-5);
\draw[->-] (-10, -5) -- (-5,-2.5);
\draw[->-,red] (-15, -8) -- (-10, -5);
\draw[->-,red] (-15,-2) -- (-10, -5);

\draw[->-] (-10, -15) -- (-5,-15);
\draw[->-,red] (-10, -15) -- (-5,-17.5);
\draw[->-,red] (-15, -18) -- (-10, -15);
\draw[->-] (-15,-12) -- (-10, -15);

\end{tikzpicture}
\caption{A uniformly balanced $T_{2,3}$-treeing with the edges of $\vec\G_S$ colored red.}
\label{fig:final_pmp_edge-slide}
\end{figure}

The graph $\vec{\G}_S$ contributes at most $2$ to the out-degree of any given point, and $q>2$. 
Applying \cref{typeIISlide} with $g(x) \coloneq q-\deg^{\mathrm{out}}_{\vec{\G}_S}(x)$, with $R \coloneq \ker_{E_S}\w_\mu$, and the map $\phi: \vec\G_1\smallsetminus\vec \G_S\to X$ that sends a directed edge $e$ to $\orig (e)$, 
we get a map $\psi: \vec\G_1 \smallsetminus \vec \G_S \to X$ such that $\psi(e)$ is $\ker_{E_S}\w_\mu$-equivalent to $\orig(e)$ and $|\psi^{-1}(x)|=q-\deg^{\mathrm{out}}_{\vec{\G}_S}(x)$ for a.e.\ $x \in X$. 
We now obtain the directed graph $\vec\G_2$ by adding to $\vec \G_S$ a positive edge $\Psi(e)$ from $\psi(e)$ to $\targ(e)$ for each positive edge $e$ of $\vec \G_1 \smallsetminus \vec \G_S$, see \cref{fig:final_pmp_edge-slide}.

By construction, the $\vec\G_2$-in-degree of each vertex is the same as its $\vec\G_1$-in-degree, while the $\vec\G_2$-out-degree is constant $q$ because $|\psi^{-1}(x)| + \deg^{\mathrm{out}}_{\vec{\G}_S}(x) = q$ for a.e.\ $x \in X$.
Furthermore, since for each positive edge $e$ of $\vec\G_1 \smallsetminus \vec\G_S$, the vertices $\orig(e)$ and $\psi(e)$ are in the same $\ker_{E_S} \w_\mu$-class and $\ker_{E_S} \w_\mu$ is pmp, the values of $\w_\mu$ on $e$ and $\Psi(e)$ coincide, hence $\w_\mu(\targ(\Psi (e)),\orig(\Psi (e)))=\frac{p}{q}$.
Lastly, $\vec \G_2$ is an acyclic graphing of $E$ because such is $\vec \G_1$ and $\vec\G_2$ was obtained by sliding edges of $\vec\G_1 \smallsetminus \vec\G_S$ along the subforest $\vec\G_S \subseteq \vec \G_1$.
Thus, $\vec\G_2$ is a $\mu$-homogeneous $T_{p,q}$-treeing of $E$.
\end{proof}

\section{Proof of Theorem \ref{thm:nonunim_nonamen_treeable_ME}}\label{sec:proof-treeable-ME}

\begin{theorem}\label{circle_factor_pmp_Krieger_gives_Tpq-treeing}
Let $E$ be a treeable nonamenable  ergodic mcp type $\mathrm{III}_0$ countable Borel equivalence relation on a standard probability space $(X,\mu)$, and let $2\leq p <q$ be integers.
Suppose that the measure class $[\eta_E]$ of its Krieger flow $\K_E$ contains an invariant probability measure.
Suppose further that $\RR \acts \RR/\langle \tfrac{p}{q}\rangle$ is a potential flow of $E$. 
Then there exists an equivalent probability measure $\mu_0 \sim \mu$ and a $\mu_0$-homogeneous $T_{p,q}$-treeing of $E$.
\end{theorem}

\begin{proof}
By the universal property, after discarding an $\RR$-invariant Borel null set, the Krieger flow $\K_E$ factors onto $\RR \acts \RR/\langle \tfrac{p}{q}\rangle$, and pulling back the set $\{1\} \subseteq \RR/\langle \tfrac{p}{q}\rangle$ through the factor map yields a $\frac{p}{q}$-equispaced (see Section~\ref{sect: real flows} for the definition) Borel cross-section $Z \subseteq \K_E$.
Applying \cref{changing measure to reach a pmp in Krieger flow} to this cross-section $Z$ gives an equivalent probability measure $\mu_0 \sim \mu$ and a Krieger scalar potential $\pi : X \to \K_E$ for $(E,X,\mu_0)$ with essential image $Z$
such that the pushforward measure $\pi'_* \mu_0$ is preserved by the cross-section equivalence relation $E^\RR_Z$ on $Z$.
It then follows that $\w_{\mu_0}([x]_E \times \{x\}) = (\frac{p}{q})^\Z$ for a.e.\ $x \in X$.

Applying \cref{thm:sameTypeHyperfiniteSubgraph} to a treeing $\G$ of $E$, we obtain a type $\mathrm{III}_0$ hyperfinite subgraph $\F \subseteq \G$ with $\w_{\mu_0} ([x]_{\F} \times \{x\}) = (\frac{p}{q})^\mathbf{Z}$ for a.e.\ $x\in X$ and $\ker_\F\w_{\mu_0} \leq \ker_\G\w_{\mu_0}$ relatively ergodic.
Applying \cref{prop:graphWithMinCocycleValues} to $R \coloneq E_\F$ yields $S \in [E_\F]$ which generates $E_\F$ and satisfies $\w_{\mu_0}(S x, x) = \alpha^{\pm 1}$ for a.e.\ $x \in X$.
Up to replacing $\F$ with $\G_S$, we can assume $\G_S$ is a subgraph of $\G$.
Thus, together with nonamenability of $E$, $\G$ and $S$ satisfy all the hypotheses of \cref{prop:edgeSlideTpqTreeing}, which in turn delivers the desired $\mu_0$-homogeneous $T_{p,q}$-treeing of $E$.
\end{proof}

\begin{theorem}[\cref{thm:nonunim_nonamen_treeable_ME}]\label{ME_with_Autpq}
Let $G$ be a treeable nonamenable nonunimodular lcsc group, and let $2\leq p <q$ be integers.
Then $G$ is measure equivalent with $\Aut(T_{p,q})$.
\end{theorem}

\begin{proof}
Fix $\alpha \in (0,1)$ such that the action $\RR \acts \Omega \coloneq \RR/\langle \frac{p}{q}\rangle\times \RR/\langle \alpha \rangle$ is free.
Let $G \acts (\~Y_0,\~\nu_0)$ be a free pmp treeable action of $G$, and let $G\acts (\~Y_0\times \Omega , \~\nu_0\times\mu_\Omega )$ be the product action, where $\mu_\Omega$ is Haar measure and $G$ acts on $\Omega$ through its modular homomorphism $\Delta$.
This $G$-action is still free and treeable, being a class-bijective extension of the $G$-action on $\~Y_0$, hence the action of $G$ on almost every ergodic component of $\~Y_0\times\Omega$ is also free and treeable.

We fix one such ergodic component $G\acts (\~X,\~\mu)$ along with a uniformly discrete Borel cross-section $X$ for this action. 
Let $E$ be the cross-section equivalence relation on $X$ and $\mu$ the canonical cross-section measure on $X$, provided by \cref{pmp_cross-section}, which, after discarding a $G$-invariant null set from $\~X$, satisfies $\w_\mu (g.x , x)=\Delta_G(g)$ whenever $x, g.x \in X$.
Letting $\pi: X \rightarrow \Omega$ denote the restriction to $X$ of the coordinate projection $\~Y_0\times \Omega \to \Omega$, it follows that for $(g.x,x)\in E$ we have $\pi (g.x) = \Delta (g).\pi (x)=\w_\mu (g.x , x).\pi (x)$.
Thus, by \cref{unique measure-class on scalar potential}, the map $\pi$ is a scalar potential for $(E,X,\mu )$, as is the composition of $\pi$ with the projection $\Omega \to \RR/\langle\frac{p}{q}\rangle$. 
Moreover, by \cref{th: pmp G-action implies potential flow} the measure class of the Krieger flow of $(E,X,\mu )$ contains an invariant probability measure, and this Krieger flow is free since it factors onto $\Omega$.
It follows that $(E,X,\mu )$ is type $\mathrm{III}_0$ and has $\RR/\langle\frac{p}{q}\rangle$ as a potential flow.

The nonamenability of $G$, along with treeability and ergodicity of the free pmp action $G\acts (\~X,\~\mu )$, ensures that $(E,X,\mu )$ is treeable, nonamenable, and ergodic. 
\cref{circle_factor_pmp_Krieger_gives_Tpq-treeing} applies, yielding an equivalent probability measure $\mu_0\sim\mu$ and a $\mu_0$-homogeneous $T_{p,q}$-treeing of $E$. 
By \cref{lem:cocycleTpqTreeingArisesFromAutTpqAction}, $(E,X,\mu _0)$ is the cross-section equivalence relation of some free pmp $\Aut (T_{p,q})$-action.
By \cite[Theorem A]{Koivisto-Kyed-Raum-ME-non-unim}, the groups $G$ and $\Aut (T_{p,q})$ are measure equivalent.
\end{proof}

\appendix
\section{Algebraic measure equivalence couplings}
\label{app:algebraic_ME}

In this appendix we formulate measure equivalence directly at the level of measure algebras, and explain how it is related to the strict (pointwise) notion from Definition~\ref{def:strictMEcoupling}.

A \textbf{measure-class algebra} is a complete Boolean algebra $\mathcal{M}$ admitting a
faithful probability measure, i.e., a countably additive probability measure $\rho$ satisfying $\rho (a)=0$ iff $a=0$. 
If $\rho$ is such a measure on $\mathcal{M}$, then
$d_\rho(a,b)\coloneqq \rho(a\triangle b)$ defines a complete metric on
$\mathcal{M}$, and the associated topology is independent of the choice of
$\rho$. 
We always consider $\mathcal{M}$ as equipped with this topology.

Given a standard Borel space $X$ equipped with a measure class $\mathscr{C}$ of $\sigma$-finite Borel measures, we write $\MAlg(X,\mathscr{C})$ for its associated measure-class algebra, i.e.\ the Boolean algebra of Borel subsets of $X$ modulo $\mathscr{C}$-null sets.
We say that a measure-class algebra is \textbf{standard} if it is
isomorphic to $\MAlg(X,\mathscr C)$ for some standard Borel space $X$ equipped
with a measure class $\mathscr C$ of $\sigma$-finite Borel measures.

By Maharam's Theorem and its consequences \cite[331O, 331P, 332B]{Fremlin3}, a measure-class algebra is standard if and
only if its canonical topology is separable. 
In particular, every closed
subalgebra of a standard measure-class algebra is standard.

Let $\mathcal{M}$ and $\mathcal{N}$ be standard measure-class algebras equipped with
faithful probability measures $\rho$ and $\eta$ respectively. 
Then the metric completion of the Boolean ring
tensor product $\mathcal{M}\otimes\mathcal{N}$ with respect to the metric
$d_{\rho\otimes\eta}(x,y)\coloneqq (\rho\otimes\eta )(x\triangle y)$ is a standard measure-class algebra, denoted $\mathcal{M}\mathbin{\widehat\otimes}\mathcal{N}$, that is independent
(up to canonical isomorphism) of the choice of $\rho$ and $\eta$: if $\mathcal{M}\cong \MAlg(X,[\rho ])$ and $\mathcal{N}\cong \MAlg(Y,[\eta ])$, then $\mathcal{M}\mathbin{\widehat{\otimes}}\mathcal{N} \cong \MAlg(X\times Y,[\rho\times\eta ])$.

Let $G$ be an lcsc group and let $\mathcal{M}$ be a standard measure-class algebra. 
By a \textbf{continuous (left) Boolean action} of $G$ on $\mathcal{M}$, we mean an action of $G$ on $\mathcal{M}$ by Boolean algebra automorphisms such that the action map $G\times\mathcal{M}\to\mathcal{M}$ is jointly continuous. 
Given such an action $G\curvearrowright\mathcal{M}$, we write $\mathcal{M}^G$ for the closed subalgebra of all $G$-invariant elements of $\mathcal{M}$. 

Let $\mathcal{A}$ be a standard
measure-class algebra equipped with the trivial action of $G$, and equip the algebra $\MAlg(G,[\mu_G])\mathbin{\widehat\otimes}\mathcal{A}$ with the product action, where $G$ acts on $\MAlg(G,[\mu_G])$ by left translation. 
Using Fubini's Theorem one sees that $\left(\MAlg(G,[\mu_G])\mathbin{\widehat\otimes}\mathcal{A} \right)^G = 1\mathbin{\widehat\otimes}\mathcal{A}$.

Given another lcsc group $H$, a \textbf{continuous Boolean $(G,H)$-biaction} on $\mathcal{M}$ consists of commuting continuous left and right actions of $G$ and $H$ respectively on $\mathcal{M}$.

\begin{definition}\label{def:algebraic_ME_coupling}
Let $G$ and $H$ be lcsc groups. An \textbf{algebraic $(G,H)$-ME-coupling} is a standard measure-class algebra $\mathcal{M}$ equipped with a continuous Boolean $(G,H)$-biaction such that:
\begin{enumerate}[(i)]
\item There exists a $G$-equivariant Boolean isomorphism $\Phi_G:\MAlg(G,[\mu_G])\mathbin{\widehat{\otimes}}\mathcal{M}^G\to\mathcal{M}$ satisfying $\Phi_G(1\otimes a)=a$ for every $a\in\mathcal{M}^G$, where $G$ acts on $\MAlg(G,[\mu_G])$ by left translation and trivially on $\mathcal{M}^G$; likewise, there exists an $H$-equivariant Boolean isomorphism $\Phi_H:\mathcal{M}^H\mathbin{\widehat{\otimes}}\MAlg(H,[\mu_H])\to\mathcal{M}$ satisfying $\Phi_H(a\otimes 1)=a$ for every $a\in\mathcal{M}^H$, where $H$ acts trivially on $\mathcal{M}^H$ and by right translation on $\MAlg(H,[\mu_H])$;
\item The algebra $\mathcal{M}^H$ admits a $G$-invariant faithful probability measure, and likewise the algebra $\mathcal{M}^G$ admits an $H$-invariant faithful probability measure.
\end{enumerate}
\end{definition}

An mcp \textbf{(left) near action} of an lcsc group $G$ on a standard $\sigma$-finite measure space $(\Omega,\nu )$ is a Borel map $G\times\Omega\to\Omega$, $(g,\omega )\mapsto g.\omega$, such that (a) for each $g\in G$ the map $\omega \mapsto g.\omega$ is mcp, (b) $e_G.\omega =\omega$ for $\nu$-a.e.\ $\omega\in \Omega$, and (c) for each $g,h\in G$, for $\nu$-a.e.\ $\omega\in \Omega$ we have $g.(h.\omega )=(gh).\omega$. 

Given such a near action of $G$, we call a Borel map $\alpha :G\times \Omega \to \mathcal{L}$ into a Polish group $\mathcal{L}$ a \textbf{near cocycle} of the near action if for all $g,h\in G$, for $\nu$-a.e.\ $\omega\in \Omega$ we have $\alpha (gh,\omega )=\alpha (g,h.\omega )\alpha (h,\omega )$.
We likewise define mcp right near actions and their near cocycles.

The proof of the converse implication of the following proposition is essentially the argument from \cite[Theorem 2.12]{KoivistoKyedRaum:measure_equivalence_unimodular} adapted to our framework.

\begin{proposition}\label{prop:algebraic_to_strict}
Every strict $(G,H)$-ME-coupling induces an algebraic
$(G,H)$-ME-coupling on its measure-class algebra. 
Conversely, every algebraic $(G,H)$-ME-coupling is isomorphic to one induced from a strict
$(G,H)$-ME-coupling.
\end{proposition}

\begin{proof}
Let $(\Omega ,\mathscr{C})$ be a strict $(G,H)$-ME-coupling, and set $\mathcal{M}\coloneqq\MAlg(\Omega,\mathscr{C})$, equipped with the induced continuous Boolean $(G,H)$-biaction. 
Since the quotient spaces $\Omega/H$ and $G\backslash\Omega$ are standard Borel, the quotient maps admit Borel sections \cite[Corollary 5.4.12]{Gao08}.
Since the actions on $\Omega$ are free, these sections give Borel identifications of $\Omega$ with $(\Omega/H)\times H$ and $G\times(G\backslash\Omega)$. 
By \cite[Lemma 2.14(i)]{KoivistoKyedRaum:measure_equivalence_unimodular}, these two identifications take $\mathscr{C}$ to $\mathscr{C}/H \times [\mu_H]$ and $[\mu_G]\times G\backslash\mathscr{C}$ respectively.
The quotient maps induce identifications of $\MAlg(\Omega/H,\mathscr{C}/H)$ with $\mathcal{M}^H$ and of $\MAlg(G\backslash\Omega,G\backslash\mathscr{C})$ with $\mathcal{M}^G$. 
We thus obtain the Boolean isomorphisms $\Phi_H$ and $\Phi_G$ required by (i) of Definition \ref{def:algebraic_ME_coupling}, while the invariant probability measures on $\Omega /H$ and $G\backslash\Omega$ give (ii).

For the converse, let $\mathcal{M}$ be an algebraic $(G,H)$-ME-coupling, and let
$L\coloneqq G\times H$ act on $\mathcal{M}$ by
$(g,h).a\coloneqq g.a.h^{-1}$.
Choose a faithful probability measure on $\mathcal{M}$ and, by Mackey's point
realization theorem \cite[Theorem 1]{Mackey:PointRealizations}, let
$(\Sigma,\mathscr C_\Sigma)$ be a point realization of this Boolean $L$-action.

Since $\mathcal{M}^H$ and $\mathcal{M}^G$ are both standard, we may write
$\mathcal{M}^H=\MAlg(X,\nu_X)$ and $\mathcal{M}^G=\MAlg(Y,\nu_Y)$, where
$(X,\nu_X)$ and $(Y,\nu_Y)$ are standard probability spaces, and by (ii) of 
Definition \ref{def:algebraic_ME_coupling} we may assume that $\nu_X$ and
$\nu_Y$ are $G$-invariant and $H$-invariant for the Boolean actions on
$\mathcal{M}^H$ and $\mathcal{M}^G$, respectively.
Let
\[
\Phi_H:\MAlg(X\times H,[\nu_X\times\mu_H])\to \mathcal{M}
\quad\text{and}\quad
\Phi_G:\MAlg(G\times Y,[\mu_G\times\nu_Y])\to \mathcal{M}
\]
be the $H$-equivariant and $G$-equivariant Boolean isomorphisms given by (i) of Definition \ref{def:algebraic_ME_coupling}.

By \cite[Theorem 2]{Mackey:PointRealizations}, after replacing $X$ and $Y$ by conull Borel sets, 
the map $\Phi_H$ is induced by an $H$-equivariant mcp Borel isomorphism $i:X\times H\to \Sigma_H$ onto an $H$-invariant conull Borel subset $\Sigma_H\subseteq \Sigma$, 
and the map $\Phi_G$ is induced by a $G$-equivariant mcp Borel isomorphism $j:G\times Y\to \Sigma_G$ onto a $G$-invariant conull Borel subset $\Sigma_G\subseteq \Sigma$.
Here, $H$ is acting trivially on the $X$-coordinate and $G$ is acting trivially on the $Y$-coordinate.

As in the proof of \cite[Theorem 2.12]{KoivistoKyedRaum:measure_equivalence_unimodular}, the map $i$ and the left $G$-action on $\Sigma$ induce a near action of $G$ on $X$ together with a near cocycle $\bar{\alpha}_H:G\times X\to H$, 
while the map $j$ and the right $H$-action on $\Sigma$ induce a near right action of $H$ on $Y$ together with a near cocycle $\bar{\alpha}_G:Y\times H\to G$.
Since $\Phi_H(a\otimes 1)=a$ and $\Phi_G(1\otimes b)=b$, the Boolean actions induced by these near actions on $\MAlg(X,\nu_X)$ and $\MAlg(Y,\nu_Y)$ are the given actions on $\mathcal{M}^H$ and $\mathcal{M}^G$, respectively, so the near actions preserve $\nu_X$ and $\nu_Y$.

By \cite[Lemma 2.11]{KoivistoKyedRaum:measure_equivalence_unimodular}, there are standard probability spaces $(X_0,\nu_{X_0})$ and $(Y_0,\nu_{Y_0})$ with genuine pmp Borel actions of $G$ and $H$ respectively, 
and measure-space isomorphisms $f_X:X\to X_0$ and $f_Y:Y\to Y_0$, such that $f_X(g.x)=g.f_X(x)$ for almost every $x\in X$, for each $g\in G$, and $f_Y(y.h)=f_Y(y).h$ for almost every $y\in Y$, for each $h\in H$. 
Define maps $i_0:X_0\times H\to\Sigma_H$ and $j_0:G\times Y_0\to\Sigma_G$, and near cocycles $\bar{\alpha}_{H,0}$ and $\bar{\alpha}_{G,0}$ by
\[
\begin{aligned}
i_0(x_0,h)&\coloneqq i(f_X^{-1}(x_0),h),
&\bar{\alpha}_{H,0}(g,x_0)&\coloneqq\bar{\alpha}_H(g,f_X^{-1}(x_0)),\\
j_0(g,y_0)&\coloneqq j(g,f_Y^{-1}(y_0)),
&\bar{\alpha}_{G,0}(y_0,h)&\coloneqq\bar{\alpha}_G(f_Y^{-1}(y_0),h).
\end{aligned}
\]
After replacing $(X,\nu_X),(Y,\nu_Y),i,j,\bar{\alpha}_H$, and $\bar{\alpha}_G$ respectively by $(X_0,\nu_{X_0}),(Y_0,\nu_{Y_0}),i_0,j_0,\bar{\alpha}_{H,0}$, and $\bar{\alpha}_{G,0}$, we may assume that the near actions on $(X,\nu_X)$ and $(Y,\nu_Y)$ are already genuine pmp Borel actions.

By \cite[Theorem B.9]{Zimmer-Erg-th-ss-gp-Book-84}, there are strict Borel cocycles $\alpha_H:G\times X\to H$ and $\alpha_G:Y\times H\to G$ such that $\alpha_H(g,\cdot)=\bar{\alpha}_H(g,\cdot)$ almost everywhere for each $g\in G$, 
and $\alpha_G(\cdot,h)=\bar{\alpha}_G(\cdot,h)$ almost everywhere for each $h\in H$.
Thus, we obtain mcp Borel $(G,H)$-biactions (and associated $L$-actions) on $X\times H$ and $G\times Y$, defined by
\[
g.(x,h_0).h \coloneqq (g.x,\alpha_H(g,x)h_0h), \qquad \text{and} \qquad
g.(g_0,y).h\coloneqq (gg_0\alpha_G(y,h),y.h).
\]
The maps $\Phi_H$ and $\Phi_G$ are $L$-equivariant for the induced Boolean $L$-actions, so the composition
\[
\Phi_G^{-1}\circ\Phi_H: \MAlg(X\times H,[\nu_X\times\mu_H]) \longrightarrow \MAlg(G\times Y,[\mu_G\times\nu_Y])
\]
is an $L$-equivariant Boolean isomorphism. Applying \cite[Theorem 2]{Mackey:PointRealizations}, we obtain $L$-invariant conull Borel subsets
$\Omega_H\subseteq X\times H$ and $\Omega_G\subseteq G\times Y$, together with an $L$-equivariant mcp Borel isomorphism $\theta:\Omega_H\to\Omega_G$ which induces $\Phi_G^{-1}\circ\Phi_H$. 

Since $\Omega_H$ is $H$-invariant, and the right $H$-action on $X\times H$ is just right translation on the $H$-coordinate (and trivial on the $X$-coordinate),
we have $\Omega_H=X_1\times H$ for some conull Borel subset $X_1\subseteq X$.
Since $\Omega_H$ is moreover $L$-invariant, the set $X_1$ is $G$-invariant. 
Likewise, $\Omega_G=G\times Y_1$ for some $H$-invariant conull Borel subset $Y_1$ of $Y$.

We now set $\Omega\coloneq X_1\times H$ and $\mathscr{C}_\Omega\coloneqq \left[(\nu_X\times\mu_H)|_{X_1\times H}\right]$.
We identify $\Omega$ with $G\times Y_1$ via $\theta$, which takes $\mathscr{C}_\Omega$ to $\left[(\mu_G\times\nu_Y)|_{G\times Y_1}\right]$.
Then $\Omega$ is a standard Borel space equipped with an mcp Borel $(G,H)$-biaction. 
The right $H$-action is free, and the quotient measurable left $G$-space $\Omega/H$ is standard, being isomorphic to $X_1$. 
Under this isomorphism, $\Omega/H$ carries the $G$-invariant Borel probability measure $\nu_X |_{X_1}$, which belongs to the measure class $\mathscr{C}_\Omega/H$. 
Similarly, the left $G$-action is free, the quotient measurable right $H$-space $G\backslash\Omega$ is standard and, under the Borel isomorphism $G\backslash\Omega \cong Y_1$, it carries the $H$-invariant Borel probability measure $\nu_Y|_{Y_1}$, which belongs to the measure class $G\backslash\mathscr C_\Omega$.

Thus $(\Omega,\mathscr{C}_\Omega)$ is a strict $(G,H)$-ME-coupling. 
By construction, the induced Boolean $(G,H)$-biaction on $\MAlg(\Omega,\mathscr{C}_\Omega)$ is isomorphic to the original $(G,H)$-biaction on $\mathcal{M}$.
\end{proof}

\def\MR#1{}
\bibliographystyle{amsalpha} 
\bibliography{refs}

\parindent=0pt\relax
\end{document}

%% file: Tikz-folders/tikzF-1.tex
 \begin{tikzpicture}[scale=0.7]
    \draw (0,0) ellipse (5cm and 1cm);
    \draw (5.5,0) node {$X_1$};
    \draw (0,2.5) ellipse (5cm and 1cm);
    \draw (5.5,2.5) node {$X_2$};
    \draw (0,5) ellipse (5cm and 1cm);
    \draw (5.5,5) node {$X_3$};

        \draw (0,0) -- (0,2.1);
        \draw (0,0) -- (0.7,4.9);

        \draw (-1,-.5) -- (-1,2);
        \draw (-1,2) -- (0.1,5.7);
        \draw (-1,2) -- (-1.8,5.3);
        
        \draw (-3,0) -- (-3.1, 2.9);
        \draw (-3,0) -- (-2.3,2.2);
        \draw (-2.3,2.2) -- (-3.1,4.5);
        \draw (-2.3,2.2) -- (-2.7,5);
        
        \draw (1.9, -.3) -- (1.1,2.6);
        \draw (1.9, -.3) -- (2,3);
        \draw (2,3) -- (1.2,5);

        \draw (2.8, -.3) -- (3.1,2.4);
        \draw (3.1,2.4) -- (2.3,5.2);

        \draw[teal, dashed, thick] (-2.7,5) -- (-1.8,5.3);
        \draw (-3,0) -- (-2,.5);
        \draw (-2.7,5) -- (-3.1,4.5);
        \draw (-2.3,2.2) -- (-3.1, 2.9);
        \draw[teal, dashed, thick] (-2.3,2.2) -- (-1,2); 

        \draw (1,0.2)[orange, dashed, thick] -- (0,2.1);
        \draw (1,0.2)[orange, dashed, thick] -- (0.7,4.9);
            
        \draw[purple, dashed, thick] (1.2,5) -- (2.3,5.2);
        \draw[purple, dashed, thick] (1.1,2.6) -- (2.3,5.2);
        \draw[purple, dashed, thick] (2,3) -- (3.1,2.4);

    \draw[teal, thick] (-3,0) -- (-1,-0.5) ;
    \draw[orange, thick] (0,0) -- (1,0.2) ;
    \draw[purple, thick] (1.9,-.3) -- (2.8,-0.3) ;
    
    \node[draw, shape=circle, fill=black, minimum size = 0.1cm, inner sep = 0cm] at (0,0) {} ;
    \node[draw, shape=circle, fill=black, minimum size = 0.1cm, inner sep = 0cm] at (1,0.2) {} ;
    \node[draw, shape=circle, fill=black, minimum size = 0.1cm, inner sep = 0cm] at (-3,0) {} ;
    \node[draw, shape=circle, fill=black, minimum size = 0.1cm, inner sep = 0cm] at (-1,-0.5) {} ;
    \node[draw, shape=circle, fill=black, minimum size = 0.1cm, inner sep = 0cm] at (1.9,-.3) {};
    \node[draw, shape=circle, fill=black, minimum size = 0.1cm, inner sep = 0cm] at (2.8,-0.3) {};
    \node[draw, shape=circle, fill=black, minimum size = 0.1cm, inner sep = 0cm] at (-2,0.5) {};

    \node[draw, shape=circle, fill=black, minimum size = 0.1cm, inner sep = 0cm] at (0,2.1) {} ;
    \node[draw, shape=circle, fill=black, minimum size = 0.1cm, inner sep = 0cm] at (1.1,2.6) {} ;
    \node[draw, shape=circle, fill=black, minimum size = 0.1cm, inner sep = 0cm] at (-3.1,2.9) {} ;
    \node[draw, shape=circle, fill=black, minimum size = 0.1cm, inner sep = 0cm] at (-1,2) {} ;
    \node[draw, shape=circle, fill=black, minimum size = 0.1cm, inner sep = 0cm] at (3.1,2.4) {};
    \node[draw, shape=circle, fill=black, minimum size = 0.1cm, inner sep = 0cm] at (-2.3,2.2) {};
    \node[draw, shape=circle, fill=black, minimum size = 0.1cm, inner sep = 0cm] at (2,3) {};

    \node[draw, shape=circle, fill=black, minimum size = 0.1cm, inner sep = 0cm] at (0.1,5.7) {} ;
    \node[draw, shape=circle, fill=black, minimum size = 0.1cm, inner sep = 0cm] at (0.7,4.9) {} ;
    \node[draw, shape=circle, fill=black, minimum size = 0.1cm, inner sep = 0cm] at (-3.1,4.5) {} ;
    \node[draw, shape=circle, fill=black, minimum size = 0.1cm, inner sep = 0cm] at (-1.8,5.3) {} ;
    \node[draw, shape=circle, fill=black, minimum size = 0.1cm, inner sep = 0cm] at (1.2,5) {};
    \node[draw, shape=circle, fill=black, minimum size = 0.1cm, inner sep = 0cm] at (2.3,5.2) {};
    \node[draw, shape=circle, fill=black, minimum size = 0.1cm, inner sep = 0cm] at (-2.7,5) {};

    \end{tikzpicture}

%% file: Tikz-folders/tikzK-1.tex
\begin{tikzpicture}[scale=0.7]
    \draw (0,0) ellipse (5cm and 1cm);
    \draw (5.5,0) node {$X_1$};
    \draw (0,2.5) ellipse (5cm and 1cm);
    \draw (5.5,2.5) node {$X_2$};
    \draw (0,5) ellipse (5cm and 1cm);
    \draw (5.5,5) node {$X_3$};

        \draw (0,0) -- (0,2.1);
        \draw (0,0) -- (0.7,4.9);

        \draw (-1,-.5) -- (-1,2);
        \draw (-1,2) -- (0.1,5.7);
        \draw (-1,2) -- (-1.8,5.3);
        
        \draw (-3,0) -- (-3.1, 2.9);
        \draw (-3,0) -- (-2.3,2.2);
        \draw (-2.3,2.2) -- (-3.1,4.5);
        \draw (-2.3,2.2) -- (-2.7,5);
        
        \draw (1.9, -.3) -- (1.1,2.6);
        \draw (1.9, -.3) -- (2,3);
        \draw (2,3) -- (1.2,5);

        \draw (2.8, -.3) -- (3.1,2.4);
        \draw (3.1,2.4) -- (2.3,5.2);

        \draw[teal, dashed, thick] (-2.7,5) -- (-1.8,5.3);
        \draw (-3,0) -- (-2,.5);
        \draw (-2.7,5) -- (-3.1,4.5);
        \draw (-2.3,2.2) -- (-3.1, 2.9);
        \draw[teal, dashed, thick] (-2.3,2.2) -- (-1,2); 

        \draw (1,0.2)[orange, dashed, thick] -- (0,2.1);
        \draw (1,0.2)[orange, dashed, thick] -- (0.7,4.9);
            
        \draw[purple, dashed, thick] (1.2,5) -- (2.3,5.2);
        \draw[purple, dashed, thick] (1.1,2.6) -- (2.3,5.2);
        \draw[purple, dashed, thick] (2,3) -- (3.1,2.4);
        
        \draw[green, dashed, thick] (0.1,5.7) -- (1.2,5);
        \draw[blue, dashed, thick] (0,2.1) -- (1.1,2.6);
        \draw[red, dashed, thick] (0,2.1) -- (-1,2);

    \draw[teal, thick] (-3,0) -- (-1,-0.5) ;
    \draw[orange, thick] (0,0) -- (1,0.2) ;
    \draw[purple, thick] (1.9,-.3) -- (2.8,-0.3) ;
    \draw[red, thick] (0,0) -- (-1,-0.5) ;
    \draw[blue, thick] (0,0) -- (1.9,-.3) ;
    \draw[green, thick] (-1,-0.5) -- (1.9,-.3) ;
    
    \node[draw, shape=circle, fill=black, minimum size = 0.1cm, inner sep = 0cm] at (0,0) {} ;
    \node[draw, shape=circle, fill=black, minimum size = 0.1cm, inner sep = 0cm] at (1,0.2) {} ;
    \node[draw, shape=circle, fill=black, minimum size = 0.1cm, inner sep = 0cm] at (-3,0) {} ;
    \node[draw, shape=circle, fill=black, minimum size = 0.1cm, inner sep = 0cm] at (-1,-0.5) {} ;
    \node[draw, shape=circle, fill=black, minimum size = 0.1cm, inner sep = 0cm] at (1.9,-.3) {};
    \node[draw, shape=circle, fill=black, minimum size = 0.1cm, inner sep = 0cm] at (2.8,-0.3) {};
    \node[draw, shape=circle, fill=black, minimum size = 0.1cm, inner sep = 0cm] at (-2,0.5) {};

    \node[draw, shape=circle, fill=black, minimum size = 0.1cm, inner sep = 0cm] at (0,2.1) {} ;
    \node[draw, shape=circle, fill=black, minimum size = 0.1cm, inner sep = 0cm] at (1.1,2.6) {} ;
    \node[draw, shape=circle, fill=black, minimum size = 0.1cm, inner sep = 0cm] at (-3.1,2.9) {} ;
    \node[draw, shape=circle, fill=black, minimum size = 0.1cm, inner sep = 0cm] at (-1,2) {} ;
    \node[draw, shape=circle, fill=black, minimum size = 0.1cm, inner sep = 0cm] at (3.1,2.4) {};
    \node[draw, shape=circle, fill=black, minimum size = 0.1cm, inner sep = 0cm] at (-2.3,2.2) {};
    \node[draw, shape=circle, fill=black, minimum size = 0.1cm, inner sep = 0cm] at (2,3) {};

    \node[draw, shape=circle, fill=black, minimum size = 0.1cm, inner sep = 0cm] at (0.1,5.7) {} ;
    \node[draw, shape=circle, fill=black, minimum size = 0.1cm, inner sep = 0cm] at (0.7,4.9) {} ;
    \node[draw, shape=circle, fill=black, minimum size = 0.1cm, inner sep = 0cm] at (-3.1,4.5) {} ;
    \node[draw, shape=circle, fill=black, minimum size = 0.1cm, inner sep = 0cm] at (-1.8,5.3) {} ;
    \node[draw, shape=circle, fill=black, minimum size = 0.1cm, inner sep = 0cm] at (1.2,5) {};
    \node[draw, shape=circle, fill=black, minimum size = 0.1cm, inner sep = 0cm] at (2.3,5.2) {};
    \node[draw, shape=circle, fill=black, minimum size = 0.1cm, inner sep = 0cm] at (-2.7,5) {};

        \draw[dashed] (-3.1, 2.9) -- (-4, 2.5);
        \draw[dashed] (-3.1,4.5) -- (-4, 5);
        \draw[dashed] (-3,0) -- (-4, 0);
        \draw[dashed] (2.8, -.3) -- (4,0);
        \draw[dashed] (3.1,2.4) -- (4,2.5);
        \draw[dashed] (2.3,5.2) -- (4,5);
    \end{tikzpicture}

%% file: Tikz-folders/tikzF-2.tex
    \begin{tikzpicture}[scale=0.7]
    \draw (0,0) ellipse (5cm and 1cm);
    \draw (5.5,0) node {$X_1$};
        \draw (-3,0) -- (-2,.5);

    \draw (-3,0) -- (-1,-0.5) ;
    \draw (0,0) -- (1,0.2) ;
    \draw (1.9,-.3) -- (2.8,-0.3) ;

    \node[draw, shape=circle, fill=black, minimum size = 0.1cm, inner sep = 0cm] at (0,0) {} ;
    \node[draw, shape=circle, fill=black, minimum size = 0.1cm, inner sep = 0cm] at (1,0.2) {} ;
    \node[draw, shape=circle, fill=black, minimum size = 0.1cm, inner sep = 0cm] at (-3,0) {} ;
    \node[draw, shape=circle, fill=black, minimum size = 0.1cm, inner sep = 0cm] at (-1,-0.5) {} ;
    \node[draw, shape=circle, fill=black, minimum size = 0.1cm, inner sep = 0cm] at (1.9,-.3) {};
    \node[draw, shape=circle, fill=black, minimum size = 0.1cm, inner sep = 0cm] at (2.8,-0.3) {};
    \node[draw, shape=circle, fill=black, minimum size = 0.1cm, inner sep = 0cm] at (-2,0.5) {};
    
    \end{tikzpicture}

%% file: Tikz-folders/tikzK-2.tex
    \begin{tikzpicture}[scale=0.7]
    \draw (0,0) ellipse (5cm and 1cm);
    \draw (5.5,0) node {$X_1$};
        \draw (-3,0) -- (-2,.5);

    \draw (-3,0) -- (-1,-0.5) ;
    \draw (0,0) -- (1,0.2) ;
    \draw (1.9,-.3) -- (2.8,-0.3) ;
    \draw (0,0) -- (-1,-0.5) ;
    \draw (0,0) -- (1.9,-.3) ;
    \draw (-1,-0.5) -- (1.9,-.3) ;
    
    \node[draw, shape=circle, fill=black, minimum size = 0.1cm, inner sep = 0cm] at (0,0) {} ;
    \node[draw, shape=circle, fill=black, minimum size = 0.1cm, inner sep = 0cm] at (1,0.2) {} ;
    \node[draw, shape=circle, fill=black, minimum size = 0.1cm, inner sep = 0cm] at (-3,0) {} ;
    \node[draw, shape=circle, fill=black, minimum size = 0.1cm, inner sep = 0cm] at (-1,-0.5) {} ;
    \node[draw, shape=circle, fill=black, minimum size = 0.1cm, inner sep = 0cm] at (1.9,-.3) {};
    \node[draw, shape=circle, fill=black, minimum size = 0.1cm, inner sep = 0cm] at (2.8,-0.3) {};
    \node[draw, shape=circle, fill=black, minimum size = 0.1cm, inner sep = 0cm] at (-2,0.5) {};
    
    \draw[dashed] (-3,0) -- (-4, 0);
    \draw[dashed] (2.8, -.3) -- (4,0);
    \end{tikzpicture}

%% file: refs.bib
@article{Dang73:typeIII_elasticity,
author = {Dang-Ngoc-Nghiem},
journal = {Annales de l'I.H.P. Probabilités et statistiques},
language = {eng},
number = {4},
pages = {397-425},
publisher = {Gauthier-Villars},
title = {On the classification of dynamical systems},
url = {http://eudml.org/doc/76990},
volume = {9},
year = {1973},
}

@article {Krieger:nonsingular_trans_I_II,
    AUTHOR = {Krieger, Wolfgang},
     TITLE = {On non-singular transformations of a measure space. {I}, {II}},
   JOURNAL = {Z. Wahrscheinlichkeitstheorie und Verw. Gebiete},
  FJOURNAL = {Zeitschrift f\"ur Wahrscheinlichkeitstheorie und Verwandte
              Gebiete},
    VOLUME = {11},
      YEAR = {1969},
     PAGES = {83--97; ibid. 11 (1969), 98--119},
   MRCLASS = {28.70},
  MRNUMBER = {240279},
MRREVIEWER = {R.\ V.\ Chacon},
       DOI = {10.1007/BF00531811},
       URL = {https://doi.org/10.1007/BF00531811},
}

@article{CGMTD,
  title={One-ended spanning subforests and treeability of groups},
  author={Conley, Clinton T. and Gaboriau, Damien and Marks, Andrew S. and Tucker-Drob, Robin D.},
  journal={preprint arXiv:2104.07431},
  year={2021},
  note={to appear in Trans. Amer. Math. Soc.}
}

@book {BdlHV,
    AUTHOR = {Bekka, Bachir and de la Harpe, Pierre and Valette, Alain},
     TITLE = {Kazhdan's property ({T})},
    SERIES = {New Mathematical Monographs},
    VOLUME = {11},
 PUBLISHER = {Cambridge University Press, Cambridge},
      YEAR = {2008},
     PAGES = {xiv+472},
      ISBN = {978-0-521-88720-5},
   MRCLASS = {22-02 (22E40 28D15 37A15 43A07 43A35)},
  MRNUMBER = {2415834},
MRREVIEWER = {Markus\ Neuhauser},
       DOI = {10.1017/CBO9780511542749},
       URL = {https://doi.org/10.1017/CBO9780511542749},
}

@article{DL14,
  title={Permanence properties of property {A} and coarse embeddability for locally compact groups},
  author={Deprez, Steven and Li, Kang},
  journal={arXiv:1403.7111 [math.GR]},
  year={2014}
}

@article {KPV15,
    AUTHOR = {Kyed, David and Petersen, Henrik Densing and Vaes, Stefaan},
     TITLE = {{$L^2$}-{B}etti numbers of locally compact groups and their
              cross section equivalence relations},
   JOURNAL = {Trans. Amer. Math. Soc.},
  FJOURNAL = {Transactions of the American Mathematical Society},
    VOLUME = {367},
      YEAR = {2015},
    NUMBER = {7},
     PAGES = {4917--4956},
      ISSN = {0002-9947,1088-6850},
   MRCLASS = {22D25 (22E40 22F10 37A20)},
  MRNUMBER = {3335405},
MRREVIEWER = {Bachir\ Bekka},
       DOI = {10.1090/S0002-9947-2015-06449-6},
}

@article {Ra71,
    AUTHOR = {Ramsay, Arlan},
     TITLE = {Virtual groups and group actions},
   JOURNAL = {Advances in Math.},
  FJOURNAL = {Advances in Mathematics},
    VOLUME = {6},
      YEAR = {1971},
     PAGES = {253--322},
      ISSN = {0001-8708},
   MRCLASS = {28.70 (20.00)},
  MRNUMBER = {281876},
MRREVIEWER = {B.\ Schreiber},
       DOI = {10.1016/0001-8708(71)90018-1},
}

@article{OW87,
	author = {Ornstein, Donald and Weiss, Benjamin},
	coden = {JOAMAV},
	fjournal = {Journal d'Analyse Math\'ematique},
	issn = {0021-7670},
	journal = {J. Analyse Math.},
	mrclass = {28D20 (22D40)},
	mrnumber = {88j:28014},
	mrreviewer = {John C. Kieffer},
	pages = {1--141},
	title = {Entropy and isomorphism theorems for actions of amenable groups},
	volume = {48},
	year = {1987}}

@article{Baumslag-Solitar-1962,
	author = {Baumslag, Gilbert and Solitar, Donald},
	fjournal = {Bulletin of the American Mathematical Society},
	issn = {0002-9904},
	journal = {Bull. Amer. Math. Soc.},
	mrclass = {20.99},
	mrnumber = {MR0142635 (26 \#204)},
	mrreviewer = {B. H. Neumann},
	pages = {199--201},
	title = {Some two-generator one-relator non-{H}opfian groups},
	volume = {68},
	year = {1962}}

@article {Farb-Mosher,
    AUTHOR = {Farb, Benson and Mosher, Lee},
     TITLE = {A rigidity theorem for the solvable {B}aumslag-{S}olitar
              groups},
      NOTE = {With an appendix by Daryl Cooper},
   JOURNAL = {Invent. Math.},
  FJOURNAL = {Inventiones Mathematicae},
    VOLUME = {131},
      YEAR = {1998},
    NUMBER = {2},
     PAGES = {419--451},
      ISSN = {0020-9910},
   MRCLASS = {57M07 (20F32 57M20)},
  MRNUMBER = {1608595},
MRREVIEWER = {Raul Quiroga-Barranco},
       DOI = {10.1007/s002220050210},
}

@article {Whyte,
    AUTHOR = {Whyte, Kevin},
     TITLE = {The large scale geometry of the higher {B}aumslag-{S}olitar
              groups},
   JOURNAL = {Geom. Funct. Anal.},
  FJOURNAL = {Geometric and Functional Analysis},
    VOLUME = {11},
      YEAR = {2001},
    NUMBER = {6},
     PAGES = {1327--1343},
      ISSN = {1016-443X},
   MRCLASS = {20F65 (20F69)},
  MRNUMBER = {1878322},
MRREVIEWER = {Lee Mosher},
       DOI = {10.1007/s00039-001-8232-6},
}

@article {Tserunyan:hyperfin_erg_subgraph,
    AUTHOR = {Tserunyan, Anush},
     TITLE = {Pointwise ergodic theorem for locally countable quasi-pmp
              graphs},
   JOURNAL = {J. Mod. Dyn.},
  FJOURNAL = {Journal of Modern Dynamics},
    VOLUME = {18},
      YEAR = {2022},
     PAGES = {609--655},
      ISSN = {1930-5311,1930-532X},
   MRCLASS = {37A20 (03E15 05C63 37A25 37A30)},
  MRNUMBER = {4521687},
MRREVIEWER = {Andrzej\ Wi\'{s}nicki},
       DOI = {10.3934/jmd.2022019},
}

@article {KatznelsonWeiss:NonsingularTransformations,
    AUTHOR = {Katznelson, Yitzhak and Weiss, Benjamin},
     TITLE = {The classification of nonsingular actions, revisited},
   JOURNAL = {Ergodic Theory Dynam. Systems},
  FJOURNAL = {Ergodic Theory and Dynamical Systems},
    VOLUME = {11},
      YEAR = {1991},
    NUMBER = {2},
     PAGES = {333--348},
      ISSN = {0143-3857,1469-4417},
   MRCLASS = {28D05},
  MRNUMBER = {1116644},
MRREVIEWER = {B.\ Kami\'nski},
       DOI = {10.1017/S0143385700006179},
}

@article {KoivistoKyedRaum:measure_equivalence_unimodular,
    AUTHOR = {Koivisto, Juhani and Kyed, David and Raum, Sven},
     TITLE = {Measure equivalence and coarse equivalence for unimodular
              locally compact groups},
   JOURNAL = {Groups Geom. Dyn.},
  FJOURNAL = {Groups, Geometry, and Dynamics},
    VOLUME = {15},
      YEAR = {2021},
    NUMBER = {1},
     PAGES = {223--267},
      ISSN = {1661-7207,1661-7215},
   MRCLASS = {20F65 (22D05 57M07)},
  MRNUMBER = {4235753},
MRREVIEWER = {Pierre-Emmanuel\ Caprace},
       DOI = {10.4171/ggd/597},
}

@article {Hahn:HaarMeasureGroupoids,
    AUTHOR = {Hahn, Peter},
     TITLE = {Haar measure for measure groupoids},
   JOURNAL = {Trans. Amer. Math. Soc.},
  FJOURNAL = {Transactions of the American Mathematical Society},
    VOLUME = {242},
      YEAR = {1978},
     PAGES = {1--33},
      ISSN = {0002-9947,1088-6850},
   MRCLASS = {28C10 (22D40)},
  MRNUMBER = {496796},
MRREVIEWER = {L.\ Corwin},
       DOI = {10.2307/1997726},
}

@book {HamachiOsikawa:ErgGroupsAndKrieger,
    AUTHOR = {Hamachi, Toshihiro and Osikawa, Motosige},
     TITLE = {Ergodic groups of automorphisms and {K}rieger's theorems},
    SERIES = {Seminar on Mathematical Sciences},
    VOLUME = {3},
 PUBLISHER = {Keio University, Department of Mathematics, Yokohama},
      YEAR = {1981},
     PAGES = {113},
   MRCLASS = {46L40 (28Dxx 46L55)},
  MRNUMBER = {617740},
}

@article {Mackey:PointRealizations,
    AUTHOR = {Mackey, George W.},
     TITLE = {Point realizations of transformation groups},
   JOURNAL = {Illinois J. Math.},
  FJOURNAL = {Illinois Journal of Mathematics},
    VOLUME = {6},
      YEAR = {1962},
     PAGES = {327--335},
      ISSN = {0019-2082},
   MRCLASS = {28.65 (22.40)},
  MRNUMBER = {143874},
MRREVIEWER = {H.\ L.\ Royden},
}

@article {HoudayerRaum:BSGroups,
    AUTHOR = {Houdayer, Cyril and Raum, Sven},
     TITLE = {Baumslag-{S}olitar groups, relative profinite completions and
              measure equivalence rigidity},
   JOURNAL = {J. Topol.},
  FJOURNAL = {Journal of Topology},
    VOLUME = {8},
      YEAR = {2015},
    NUMBER = {1},
     PAGES = {295--313},
      ISSN = {1753-8416,1753-8424},
   MRCLASS = {20F65 (20E08 37A20)},
  MRNUMBER = {3335256},
MRREVIEWER = {Dmytro\ M.\ Savchuk},
       DOI = {10.1112/jtopol/jtu026},
}

@article {Kida:BSGroups,
    AUTHOR = {Kida, Yoshikata},
     TITLE = {Invariants of orbit equivalence relations and
              {B}aumslag-{S}olitar groups},
   JOURNAL = {Tohoku Math. J. (2)},
  FJOURNAL = {The Tohoku Mathematical Journal. Second Series},
    VOLUME = {66},
      YEAR = {2014},
    NUMBER = {2},
     PAGES = {205--258},
      ISSN = {0040-8735,2186-585X},
   MRCLASS = {37A20 (20Exx)},
  MRNUMBER = {3229595},
       DOI = {10.2748/tmj/1404911861},
}

@article {Moldavanskii:BSGroupsIsomorphism,
    AUTHOR = {Moldavanskii, David I.},
     TITLE = {On the isomorphisms of {B}aumslag-{S}olitar groups},
   JOURNAL = {Ukrain. Mat. Zh.},
  FJOURNAL = {Akademiya Nauk Ukrainsko\u i\ SSR. Institut Matematiki.
              Ukrainski\u i\ Matematicheski\u i\ Zhurnal},
    VOLUME = {43},
      YEAR = {1991},
    NUMBER = {12},
     PAGES = {1684--1686},
      ISSN = {0041-6053},
   MRCLASS = {20F05},
  MRNUMBER = {1172309},
MRREVIEWER = {Oleg\ I.\ Tavgen\cprime},
       DOI = {10.1007/BF01066696},
}

@article {Casals-RuizKazachkovZakharov:CommensurabilityOfBS,
    AUTHOR = {Casals-Ruiz, Montserrat and Kazachkov, Ilya and Zakharov,
              Alexander},
     TITLE = {Commensurability of {B}aumslag-{S}olitar groups},
   JOURNAL = {Indiana Univ. Math. J.},
  FJOURNAL = {Indiana University Mathematics Journal},
    VOLUME = {70},
      YEAR = {2021},
    NUMBER = {6},
     PAGES = {2527--2555},
      ISSN = {0022-2518,1943-5258},
   MRCLASS = {20E07 (20E34 20F05)},
  MRNUMBER = {4359918},
MRREVIEWER = {Matteo\ Vannacci},
       DOI = {10.1512/iumj.2021.70.9496},
}

@article {Forester:BSComplexes,
    AUTHOR = {Forester, Max},
     TITLE = {Incommensurable lattices in {B}aumslag-{S}olitar complexes},
   JOURNAL = {J. Lond. Math. Soc. (2)},
  FJOURNAL = {Journal of the London Mathematical Society. Second Series},
    VOLUME = {109},
      YEAR = {2024},
    NUMBER = {3},
     PAGES = {Paper No. e12879, 31},
      ISSN = {0024-6107,1469-7750},
   MRCLASS = {20F65 (20E08 22D05 57M07)},
  MRNUMBER = {4754444},
       DOI = {10.1112/jlms.12879},
}

@article{Kida:TreeingsHNN,
 author = {Kida, Yoshikata},
 title = {On treeings arising from {HNN} extensions},
 fjournal = {Journal of Topology and Analysis},
 journal = {J. Topol. Anal.},
 issn = {1793-5253},
 volume = {18},
 number = {1},
 pages = {243--278},
 year = {2026},
 language = {English},
 doi = {10.1142/S1793525324500304},
 URL = {https://doi.org/10.1142/S1793525324500304}
}

@article {Kida:stability,
    AUTHOR = {Kida, Yoshikata},
     TITLE = {Stability in orbit equivalence for {B}aumslag-{S}olitar groups
              and {V}aes groups},
   JOURNAL = {Groups Geom. Dyn.},
  FJOURNAL = {Groups, Geometry, and Dynamics},
    VOLUME = {9},
      YEAR = {2015},
    NUMBER = {1},
     PAGES = {203--235},
      ISSN = {1661-7207,1661-7215},
   MRCLASS = {37A20 (20E06 20E08 37C85)},
  MRNUMBER = {3343352},
MRREVIEWER = {Rafael\ Oswaldo\ Ruggiero},
       DOI = {10.4171/GGD/311},
}

@article {JS87,
    AUTHOR = {Jones, Vaughan F. R. and Schmidt, Klaus},
     TITLE = {Asymptotically invariant sequences and approximate finiteness},
   JOURNAL = {Amer. J. Math.},
  FJOURNAL = {American Journal of Mathematics},
    VOLUME = {109},
      YEAR = {1987},
    NUMBER = {1},
     PAGES = {91--114},
      ISSN = {0002-9327,1080-6377},
   MRCLASS = {28D99 (46L55)},
  MRNUMBER = {878200},
MRREVIEWER = {S.\ G.\ Dani},
       DOI = {10.2307/2374553},
}

@article{Levitt:GBS,
 author = {Levitt, Gilbert},
 title = {On the automorphism group of generalized {Baumslag--Solitar} groups},
 fjournal = {Geometry \& Topology},
 journal = {Geom. Topol.},
 volume = {11},
 pages = {473--515},
 year = {2007},
 doi = {10.2140/gt.2007.11.473},
 URL = {https://doi.org/10.2140/gt.2007.11.473}
}

@article {MeesschaertVaes:BSgroups,
    AUTHOR = {Meesschaert, Niels and Vaes, Stefaan},
     TITLE = {Partial classification of the {B}aumslag-{S}olitar group von
              {N}eumann algebras},
   JOURNAL = {Doc. Math.},
  FJOURNAL = {Documenta Mathematica},
    VOLUME = {19},
      YEAR = {2014},
     PAGES = {629--645},
      ISSN = {1431-0635,1431-0643},
   MRCLASS = {46L36 (20E06 22D25)},
  MRNUMBER = {3247797},
MRREVIEWER = {Sven\ Raum},
}

@article {Meesschaert:BSactions,
    AUTHOR = {Meesschaert, Niels},
     TITLE = {A rigidity result for crossed products of actions of
              {B}aumslag-{S}olitar groups},
   JOURNAL = {Internat. J. Math.},
  FJOURNAL = {International Journal of Mathematics},
    VOLUME = {26},
      YEAR = {2015},
    NUMBER = {14},
     PAGES = {1550117, 32},
      ISSN = {0129-167X,1793-6519},
   MRCLASS = {46L10 (20F05 22F10 28D15 37A55 37C85 46L54 46L55)},
  MRNUMBER = {3438170},
MRREVIEWER = {R\'emi\ Boutonnet},
       DOI = {10.1142/S0129167X15501177},
}

@article {ChifanIoana:MEvsW*,
    AUTHOR = {Chifan, Ionut and Ioana, Adrian},
     TITLE = {On a question of {D}. {S}hlyakhtenko},
   JOURNAL = {Proc. Amer. Math. Soc.},
  FJOURNAL = {Proceedings of the American Mathematical Society},
    VOLUME = {139},
      YEAR = {2011},
    NUMBER = {3},
     PAGES = {1091--1093},
      ISSN = {0002-9939,1088-6826},
   MRCLASS = {46L10 (37A20)},
  MRNUMBER = {2745659},
MRREVIEWER = {Stefaan\ Vaes},
       DOI = {10.1090/S0002-9939-2010-10553-1},
}

@article{Hjorth:cost-attained,
 author = {Hjorth, Greg},
 title = {A lemma for cost attained},
 fjournal = {Annals of Pure and Applied Logic},
 journal = {Ann. Pure Appl. Logic},
 issn = {0168-0072},
 volume = {143},
 number = {1-3},
 pages = {87--102},
 year = {2006},
 language = {English},
 doi = {10.1016/j.apal.2005.05.034},
 URL = {https://doi.org/10.1016/j.apal.2005.05.034}
}

@article{Maharam:Extension1964,
author = {Maharam, Dorothy},
journal = {Fundamenta Mathematicae},
language = {eng},
number = {1},
pages = {35-50},
title = {Incompressible transformations},
url = {http://eudml.org/doc/213824},
volume = {56},
year = {1964},
}

@article {Krieger:Flow,
    AUTHOR = {Krieger, Wolfgang},
     TITLE = {On ergodic flows and the isomorphism of factors},
   JOURNAL = {Math. Ann.},
  FJOURNAL = {Mathematische Annalen},
    VOLUME = {223},
      YEAR = {1976},
    NUMBER = {1},
     PAGES = {19--70},
      ISSN = {0025-5831,1432-1807},
   MRCLASS = {46L10 (28A65)},
  MRNUMBER = {415341},
MRREVIEWER = {M.\ Takesaki},
       DOI = {10.1007/BF01360278},
       URL = {https://doi.org/10.1007/BF01360278},
}

@article {Poulin,
    AUTHOR = {Poulin, Antoine},
     TITLE = {Elasticity of free type {III} actions of free groups},
   JOURNAL = {J. Topol. Anal.},
  FJOURNAL = {Journal of Topology and Analysis},
    VOLUME = {18},
      YEAR = {2026},
    NUMBER = {5},
     PAGES = {1477--1487},
      ISSN = {1793-5253,1793-7167},
   MRCLASS = {37A20 (03E15)},
  MRNUMBER = {5051785},
       DOI = {10.1142/S1793525325500219},
       URL = {https://doi.org/10.1142/S1793525325500219},
}

@article {Koivisto-Kyed-Raum-ME-non-unim,
    AUTHOR = {Koivisto, Juhani and Kyed, David and Raum, Sven},
     TITLE = {Measure equivalence for non-unimodular groups},
   JOURNAL = {Transform. Groups},
  FJOURNAL = {Transformation Groups},
    VOLUME = {26},
      YEAR = {2021},
    NUMBER = {1},
     PAGES = {327--346},
      ISSN = {1083-4362},
   MRCLASS = {22D05 (28C10 28D05 37A15 43A05 43A07)},
  MRNUMBER = {4229668},
MRREVIEWER = {Radhakrishnan Nair},
       DOI = {10.1007/s00031-021-09640-5},
       URL = {https://doi-org.acces.bibliotheque-diderot.fr/10.1007/s00031-021-09640-5},
}

@article {LeMaitre:full-groups,
    AUTHOR = {Le Ma\^itre, Fran\c{c}ois},
     TITLE = {On full groups of non-ergodic probability-measure-preserving
              equivalence relations},
   JOURNAL = {Ergodic Theory Dynam. Systems},
  FJOURNAL = {Ergodic Theory and Dynamical Systems},
    VOLUME = {36},
      YEAR = {2016},
    NUMBER = {7},
     PAGES = {2218--2245},
      ISSN = {0143-3857,1469-4417},
   MRCLASS = {28D15 (22F10 37A05)},
  MRNUMBER = {3568978},
       DOI = {10.1017/etds.2015.20},
       URL = {https://doi.org/10.1017/etds.2015.20},
}

@article{Dye:groups_mp_transformations_I,
    AUTHOR = {Dye, Henry A.},
     TITLE = {On groups of measure preserving transformations. {I}},
   JOURNAL = {Amer. J. Math.},
  FJOURNAL = {American Journal of Mathematics},
    VOLUME = {81},
      YEAR = {1959},
     PAGES = {119--159},
      ISSN = {0002-9327,1080-6377},
   MRCLASS = {28.70},
  MRNUMBER = {131516},
MRREVIEWER = {M.\ Mahowald},
       DOI = {10.2307/2372852},
       URL = {https://doi.org/10.2307/2372852},
}

@article{FM77a,
	author = {Feldman, Jacob and Moore, Calvin C.},
	journal = {Trans. Amer. Math. Soc.},
	number = {2},
	pages = {289--324},
	title = {Ergodic equivalence relations, cohomology, and von {N}eumann algebras. {I}},
	volume = {234},
	year = {1977},
}

@book{Schmidt-1990-Algebraic-ideas,
	author = {Schmidt, Klaus},
	doi = {10.1090/cbms/076},
	isbn = {0-8218-0727-7},
	mrclass = {28D05 (22D40 28D15 28D20 46L55 54C70)},
	mrnumber = {1074576},
	mrreviewer = {A. I. Danilenko},
	pages = {vi+94},
	publisher = {Published for the Conference Board of the Mathematical Sciences, Washington, DC; by the American Mathematical Society, Providence, RI},
	series = {CBMS Regional Conference Series in Mathematics},
	title = {Algebraic ideas in ergodic theory},
	url = {https://doi-org.acces.bibliotheque-diderot.fr/10.1090/cbms/076},
	volume = {76},
	year = {1990}}

@article{Mackey-1966-Erg-th-virt-gps,
	author = {Mackey, George W.},
	doi = {10.1007/BF01361167},
	fjournal = {Mathematische Annalen},
	issn = {0025-5831},
	journal = {Math. Ann.},
	mrclass = {22.20 (28.70)},
	mrnumber = {201562},
	mrreviewer = {F. Hahn},
	pages = {187--207},
	title = {Ergodic theory and virtual groups},
	url = {https://doi-org.acces.bibliotheque-diderot.fr/10.1007/BF01361167},
	volume = {166},
	year = {1966}}

@book{Zimmer-Erg-th-ss-gp-Book-84,
	address = {Basel},
	author = {Zimmer, Robert J.},
	isbn = {3-7643-3184-4},
	mrclass = {22E40 (22D40 28D15)},
	mrnumber = {MR776417 (86j:22014)},
	mrreviewer = {S. G. Dani},
	pages = {x+209},
	publisher = {Birkh\"auser Verlag},
	series = {Monographs in Mathematics},
	title = {Ergodic theory and semisimple groups},
	volume = {81},
	year = {1984}}

@inproceedings{Krieger-1970-Araki-Woods,
	author = {Krieger, Wolfgang},
	booktitle = {Contributions to {E}rgodic {T}heory and {P}robability ({P}roc. {C}onf., {O}hio {S}tate {U}niv., {C}olumbus, {O}hio, 1970)},
	mrclass = {28A65 (46L10)},
	mrnumber = {414823},
	pages = {158--177},
	publisher = {Springer, Berlin-New York},
	series = {Lecture Notes in Math., Vol. 160},
	title = {On the {A}raki-{W}oods asymptotic ratio set and non-singular transformations of a measure space},
	year = {1970}}

@incollection {Fur,
    AUTHOR = {Furman, Alex},
     TITLE = {A survey of measured group theory},
 BOOKTITLE = {Geometry, rigidity, and group actions},
    SERIES = {Chicago Lectures in Math.},
     PAGES = {296--374},
 PUBLISHER = {Univ. Chicago Press, Chicago, IL},
      YEAR = {2011},
      ISBN = {978-0-226-23788-6; 0-226-23788-5},
   MRCLASS = {37A20 (20F65 22D40)},
  MRNUMBER = {2807836},
MRREVIEWER = {Claire\ Anantharaman-Delaroche},
}

@incollection{Gromov-asympt-inv,
	address = {Cambridge},
	author = {Gromov, Mikhael},
	booktitle = {Geometric group theory, {V}ol.\ 2 ({S}ussex, 1991)},
	mrclass = {20F32 (57M07)},
	mrnumber = {1253544 (95m:20041)},
	pages = {1--295},
	publisher = {Cambridge Univ. Press},
	series = {London Math. Soc. Lecture Note Ser.},
	title = {Asymptotic invariants of infinite groups},
	volume = {182},
	year = {1993}}

@article{BFS-13,
	author = {Bader, Uri and Furman, Alex and Sauer, Roman},
	doi = {10.1007/s00222-012-0445-9},
	fjournal = {Inventiones Mathematicae},
	issn = {0020-9910},
	journal = {Invent. Math.},
	mrclass = {22E40 (58D19)},
	mrnumber = {3117525},
	mrreviewer = {A. I. Danilenko},
	number = {2},
	pages = {313--379},
	title = {Integrable measure equivalence and rigidity of hyperbolic lattices},
	url = {https://doi-org.acces.bibliotheque-diderot.fr/10.1007/s00222-012-0445-9},
	volume = {194},
	year = {2013}}

@book{Kechris:Classical_DST,
    AUTHOR = {Kechris, Alexander S.},
     TITLE = {Classical descriptive set theory},
    SERIES = {Graduate Texts in Mathematics},
    VOLUME = {156},
 PUBLISHER = {Springer-Verlag, New York},
      YEAR = {1995},
     PAGES = {xviii+402},
      ISBN = {0-387-94374-9},
   MRCLASS = {03E15 (03-01 03-02 04A15 28A05 54H05 90D44)},
  MRNUMBER = {1321597},
MRREVIEWER = {Jakub\ Jasi\'nski},
       DOI = {10.1007/978-1-4612-4190-4},
       URL = {https://doi.org/10.1007/978-1-4612-4190-4},
}

@article {JKLcber,
    AUTHOR = {Jackson, Stephen C. and Kechris, Alexander~S. and Louveau, Alain},
     TITLE = {Countable {B}orel equivalence relations},
   JOURNAL = {J. Math. Log.},
  FJOURNAL = {Journal of Mathematical Logic},
    VOLUME = {2},
      YEAR = {2002},
    NUMBER = {1},
     PAGES = {1--80},
      ISSN = {0219-0613},
   MRCLASS = {03E15 (22A05)},
  MRNUMBER = {1900547},
MRREVIEWER = {Otmar Spinas},
       DOI = {10.1142/S0219061302000138},
       URL = {https://doi.org/10.1142/S0219061302000138},
}

@article {Gaboriau05,
    AUTHOR = {Gaboriau, Damien},
     TITLE = {Invariant percolation and harmonic {D}irichlet functions},
   JOURNAL = {Geom. Funct. Anal.},
  FJOURNAL = {Geometric and Functional Analysis},
    VOLUME = {15},
      YEAR = {2005},
    NUMBER = {5},
     PAGES = {1004--1051},
      ISSN = {1016-443X},
   MRCLASS = {60K35 (37A20)},
  MRNUMBER = {2221157},
MRREVIEWER = {Filippo Cesi},
       DOI = {10.1007/s00039-005-0539-2}
}

@misc{Chen-Terlov-Tserunyan,
      title={Nonamenable subforests of multi-ended quasi-pmp graphs}, 
      author={Ruiyuan Chen and Grigory Terlov and Anush Tserunyan},
      year={2025},
      eprint={2211.07908},
      archivePrefix={arXiv},
      primaryClass={math.DS},
      url={https://arxiv.org/abs/2211.07908}, 
}

@article{Hopf-1932,
 author = {Hopf, Eberhard},
 title = {Theory of measure and invariant integrals},
 fjournal = {Transactions of the American Mathematical Society},
 journal = {Trans. Am. Math. Soc.},
 issn = {0002-9947},
 volume = {34},
 pages = {373--393},
 year = {1932},
 language = {English},
 doi = {10.2307/1989547},
 zbMATH = {3005018},
 Zbl = {0004.26001}
}

@article{IPV:superrigidity,
 author = {Ioana, Adrian and Popa, Sorin and Vaes, Stefaan},
 title = {A class of superrigid group von {Neumann} algebras},
 fjournal = {Annals of Mathematics. Second Series},
 journal = {Ann. of Math. (2)},
 issn = {0003-486X},
 volume = {178},
 number = {1},
 pages = {231--286},
 year = {2013},
 language = {English},
 doi = {10.4007/annals.2013.178.1.4},
 zbMATH = {6190560},
 Zbl = {1295.46041}
}

@article{CdSS:products,
 author = {Chifan, Ionut and de Santiago, Rolando and Sinclair, Thomas},
 title = {{$W^*$}-rigidity for the von {Neumann} algebras of products of hyperbolic groups},
 fjournal = {Geometric and Functional Analysis},
 journal = {Geom. Funct. Anal.},
 issn = {1016-443X},
 volume = {26},
 number = {1},
 pages = {136--159},
 year = {2016},
 language = {English},
 doi = {10.1007/s00039-016-0361-z}
}

@article{Cornulier_Valette_2015,
  title={On equivariant embeddings of generalized Baumslag--Solitar groups},
  author={Cornulier, Yves and Valette, Alain},
  journal={Geometriae Dedicata},
  volume={175},
  number={1},
  pages={385--401},
  year={2015},
  publisher={Springer}
}

@book {Fremlin3,
    AUTHOR = {Fremlin, David~H.},
     TITLE = {Measure Theory. {V}ol. 3},
 SUBTITLE = {Measure Algebras},
 PUBLISHER = {Torres Fremlin},
   ADDRESS = {Colchester},
      YEAR = {2002},
      ISBN = {978-0-9538129-3-6},
       URL = {https://www1.essex.ac.uk/maths/people/fremlin/mt.htm},
}

@book{Kechris-Miller,
	address = {Berlin},
	author = {Kechris, Alexander S. and Miller, Benjamin D.},
	isbn = {3-540-22603-6},
	mrclass = {37A15 (03E15 28D05 28D15 37A20)},
	mrnumber = {MR2095154},
	pages = {x+134},
	publisher = {Springer-Verlag},
	series = {Lecture Notes in Mathematics},
	title = {Topics in orbit equivalence},
	volume = {1852},
	year = {2004}}

@book{Gao08,
  title={Invariant descriptive set theory},
  author={Gao, Su},
  year={2008},
  publisher={Chapman and Hall/CRC}
}

@article{Gab05,
	author = {Gaboriau, Damien},
	fjournal = {Ergodic Theory and Dynamical Systems},
	issn = {0143-3857},
	journal = {Ergodic Theory Dynam. Systems},
	mrclass = {22F10 (28D15 37A15 37A20)},
	mrnumber = {MR2183295 (2006i:22024)},
	mrreviewer = {Michel Coornaert},
	notes = {\url{http://hal.ccsd.cnrs.fr/ccsd- 00004408}},
	number = {6},
	pages = {1809--1827},
	title = {Examples of groups that are measure equivalent to the free group},
	volume = {25},
	year = {2005}}

@article{Gab00a,
	author = {Gaboriau, Damien},
	coden = {INVMBH},
	fjournal = {Inventiones Mathematicae},
	issn = {0020-9910},
	journal = {Invent. Math.},
	mrclass = {28D15 (20E06 20E15 28D20 37Axx 46Lxx)},
	mrnumber = {1 728 876},
	number = {1},
	pages = {41--98},
	title = {Co{\^u}t des relations d'{\'e}quivalence et des groupes},
	volume = {139},
	year = {2000}}

@article{Kechris00,
	author = {Kechris, Alexander S.},
	coden = {JLMSAK},
	fjournal = {The Journal of the London Mathematical Society. Second Series},
	issn = {0024-6107},
	journal = {J. London Math. Soc. (2)},
	mrclass = {20K15},
	mrnumber = {1 783 636},
	number = {2},
	pages = {437--450},
	title = {On the classification problem for rank 2 torsion-free abelian groups},
	volume = {62},
	year = {2000}}

@misc{T-TD:amenable_in_tree_antitreeability,
      title={Amenable subrelations of treeable equivalence relations and applications to antitreeability}, 
      author={Anush Tserunyan and Robin D. Tucker-Drob},
      year={2026},
      note={in preparation},
}

@article {Kec92,
    AUTHOR = {Kechris, Alexander S.},
     TITLE = {Countable sections for locally compact group actions},
   JOURNAL = {Ergodic Theory Dynam. Systems},
  FJOURNAL = {Ergodic Theory and Dynamical Systems},
    VOLUME = {12},
      YEAR = {1992},
    NUMBER = {2},
     PAGES = {283--295},
      ISSN = {0143-3857,1469-4417},
   MRCLASS = {22D40 (03E15 54H11)},
  MRNUMBER = {1176624},
MRREVIEWER = {Ivan\ L.\ Reilly},
       DOI = {10.1017/S0143385700006751},
       URL = {https://doi-org.proxy3.library.mcgill.ca/10.1017/S0143385700006751},
}

@article{CCBT,
  title={Commability and envelopes of {B}aumslag-{S}olitar groups and generalizations},
  author={Carette, Mathieu and Cornulier, Yves and Le Boudec, Adrien and Tessera, Romain},
  journal={},
  year={},
  note={in preparation (personal communication)}
}
